\documentclass{article}
\usepackage{amssymb,amsthm,amsmath}
\usepackage{vmargin}
\setmarginsrb{2cm}{1cm}{2cm}{2cm}{1cm}{1cm}{1cm}{1cm}
\usepackage{amsmath}
\usepackage{color}

\newtheorem{Theorem}{Theorem}  
  
\newtheorem{Corollary}[Theorem]{Corollary}  
 
\newtheorem{Proposition}[Theorem]{Proposition}
\newtheorem{Remark}[Theorem]{Remark}
\newtheorem{Definition}[Theorem]{Definition}

\DeclareMathOperator{\Span}{span}
\DeclareMathOperator{\Ker}{ker}
\DeclareMathOperator{\Ran}{Ran}
\DeclareMathOperator{\rank}{rank}

  \newcounter{tak}

\newcommand{\ov}{\overset{ \mbox{\tiny def}}{=}}
\newcommand{\vct}[2]{{\left (\begin{array}{c}\displaystyle #1\\\displaystyle #2\end{array}\right )}}
\newcommand{\vctt}[3]{{\left (\begin{array}{c}\displaystyle #1\\\displaystyle #2\\\displaystyle #3\end{array}\right )}}
\newcommand{\mtx}[9]{ \left (
\begin{array}{ccc}
\displaystyle #1& \quad \displaystyle #2& \quad \displaystyle #3 \\
\displaystyle #4 & \quad \displaystyle #5& \quad \displaystyle #6 \\
\displaystyle #7 & \quad \displaystyle #8& \quad \displaystyle #9 \\
\end{array}
\right )}
\def\Tilde{\widetilde}

\title{On the Fattorini Criterion for Approximate Controllability and Stabilizability of Parabolic Systems}
\author{Mehdi Badra
\thanks{LMAP, UMR CNRS 5142,
    UNIV PAU \& PAYS ADOUR, 64013 Pau Cedex,
    France ({\tt mehdi.badra@univ-pau.fr}).}, Tak\'eo Takahashi
\thanks{Inria, Villers-l\`es-Nancy, F-54600, France ({\tt takeo.takahashi@inria.fr}),}
\thanks{Universit\'e de Lorraine, IECN, UMR 7502,
Vandoeuvre-l\`es-Nancy, F-54506, France,}
\thanks{CNRS, IECN, UMR 7502, Vandoeuvre-l\`es-Nancy, F-54506, France.}
}
\begin{document}

\tableofcontents

\maketitle

\begin{abstract}
In this paper, we consider the well-known Fattorini's criterion 
for approximate controllability of infinite dimensional linear systems of type $y'=A y+Bu$. We precise the result 
proved by H. O. Fattorini in \cite{Fattorini1966} for bounded input $B$, in the case where $B$ can be unbounded or in the case of finite-dimensional controls. 
More precisely, we prove that if Fattorini's criterion is satisfied and if the set of geometric
multiplicities of $A$ is bounded then approximate controllability can be achieved with finite dimensional controls.

An important consequence of this result consists in using the Fattorini's criterion to obtain the feedback stabilizability of linear and nonlinear parabolic systems
with feedback controls in a finite dimensional space. 

In particular, for systems described by partial differential equations, such a criterion reduces to a unique continuation theorem for a stationary system.
We illustrate such a method by tackling some coupled Navier-Stokes type equations (MHD system and micropolar fluid system) and
we sketch a systematic procedure relying on Fattorini's criterion for checking stabilizability of such nonlinear systems. In that case, the unique continuation theorems
rely on local Carleman inequalities for stationary Stokes type systems.
\end{abstract}

\section{Introduction}\label{sec-intro}

In this article, we consider some spectral criteria for the approximate controllability and  for the exponential stabilizability of the linear system
\begin{equation}
y'=A y+B u,\label{eq1.1}
\end{equation}
where $A$ is the infinitesimal generator of an analytic semigroup $(e^{tA})$ on a complex
Hilbert space $H$ 
and where $B : U \to \mathcal{D}(A^*)'$ is a linear bounded operator defined on a complex Hilbert space $U$. 

A well-known criterion for approximate controllability was obtained by H. O. Fattorini in \cite{Fattorini1966} in the case $B\in \mathcal{L}(U,H)$ and without the hypothesis that $(e^{tA})$ is an analytic semigroup:
  \begin{equation}
    \forall\varepsilon\in {\cal D}(A^*), \quad \forall \lambda\in
    \mathbb{C}, \quad
A^*\varepsilon= \lambda \varepsilon \quad \text{and} \quad B^*\varepsilon=0
\quad\Longrightarrow\quad 
\varepsilon = 0.  \tag{UC}\label{UC}
  \end{equation} 
In the above criterion, $A^*$ and $B^*$ denote the adjoints of $A$ and $B$ respectively.
Condition \eqref{UC} is an infinite dimensional version of the classical Hautus test 
introduced in \cite{Hautus1969} (see \cite{RussellWeiss1994} for a stronger version for exact controllability). 
Our aim is to consider similar criteria for approximate controllability and also for stabilizability, and to consider the case $B\in \mathcal{L}(U,\mathcal{D}(A^*)')$.
(Here and in what follows, $\mathcal{L}(X,Y)$ is the set of linear bounded mappings from $X$ to $Y$.)

More precisely, let us fix $\mu_0>\sup_{j\in \mathbb{N}}\Re\lambda_j$ so that fractional powers of $(\mu_0-A^*)$ and $(\mu_0-A)$ are well defined and 
for $\alpha\in \mathbb{R}$ we introduce the spaces:
\begin{equation} \label{Halpha}
H_{\alpha}\ov \left\{
\begin{array}{ll}
\mathcal{D}((\mu_0-A)^{\alpha}) &\mbox{ if }\alpha\geq 0\\
\mathcal{D}((\mu_0-A^*)^{-\alpha})' &\mbox{ if }\alpha< 0
\end{array}
\right.\quad \mbox{ and }\quad H_{\alpha}^*\ov \left\{
\begin{array}{ll}
\mathcal{D}((\mu_0-A^*)^{\alpha}) &\mbox{ if }\alpha\geq 0\\
\mathcal{D}((\mu_0-A)^{-\alpha})' &\mbox{ if }\alpha< 0,
\end{array}
\right.
\end{equation}
where for a Hilbert space $V\subset H$ such that $V$ is dense in $H$, $V'$ stands for the dual space of $V$ with respect to the pivot space $H$.

In what follows, we assume the following hypotheses
\begin{enumerate}
\item[$(\mathcal{H}_1)$] The spectrum of $A$ consists of isolated eigenvalues $(\lambda_j)$ with 
finite  algebraic multiplicity $N_j$.
\item[$(\mathcal{H}_2)$] The family of root vectors of $A$ is complete in $H$.
\item[$(\mathcal{H}_3)$] The semigroup $(e^{tA})$ is analytic.
\item[$(\mathcal{H}_4)$] The following interpolation equalities are satisfied:
$D((\mu_0-A)^\alpha)=[H,D(A)]_{\alpha}$ for all $\alpha\in [0,1]$.
\item[$(\mathcal{H}_5)$] There exists $\gamma\in [0,1)$ such that $B\in \mathcal{L}(U,H_{-\gamma})$.
\end{enumerate}
The hypothesis $(\mathcal{H}_1)$ means that the spectrum of $A$ is composed with complex eigenvalues
$\lambda_j$ for $j\in {\mathbb N}$ and that for all $j\in {\mathbb N}$ the projection operator:
\begin{equation}\label{DefProj}
R_{-1}(\lambda_j)=\frac{1}{2\pi
  \imath}\int_{|\lambda-\lambda_j|=\delta}(\lambda-A)^{-1}{\rm d\lambda},
\end{equation}
has finite dimensional range. In \eqref{DefProj} the value $\delta>0$ is chosen small enough so that the
circle $\{\lambda \in {\mathbb C}\mid |\lambda-\lambda_j|=\delta \}$
does not enclose other point of the spectrum than $\lambda_j$. It is well-known that there exists some
$n\in {\mathbb N}^*$ such that $\mbox{Ran}(R_{-1}(\lambda_j))=\Ker(\lambda_j-A)^{n}$ (see \cite{Kato1966}) and we denote by $m(\lambda_j)\in {\mathbb N}^*$ the 
smallest of such $n$. The finite dimensional subspace $\Ker(\lambda_j-A)^{m(\lambda_j)}=\mbox{Ran}(R_{-1}(\lambda_j))$ is called the
\textit{generalized eigenspace} of $A$ associated to $\lambda_j$, its
dimension $N_j\in {\mathbb N}^*$ is called the \textit{algebraic multiplicity}
of $\lambda_j$ and an element of $\Ker(\lambda_j-A)^{m(\lambda_j)}$ is called \textit{a root vector} of $A$. We also recall that 
the subspace $\Ker(\lambda_j-A)$ is called the
\textit{(proper) eigenspace} of $A$ associated to $\lambda_j$ and its dimension
$\ell_j\in {\mathbb N}^*$ is called the \textit{geometric multiplicity} of
$\lambda_j$ (recall that $\ell_j\leq N_j$). 

Using the hypotheses $(\mathcal{H}_3)$, $(\mathcal{H}_4)$ and $(\mathcal{H}_5)$ and applying classical results on parabolic systems 
, we deduce that for any $u\in L^2((0,\tau);U)$ and for any $y_0\in H_{1/2-\gamma}$ there exists a unique solution $y$ of \eqref{eq1.1} satisfying $y(0)=y_0$ and such that
\begin{equation}\label{regregreg}
y \in H^1((0,\tau);H_{-\gamma})\cap L^2((0,\tau);H_{1-\gamma})\cap {C([0,\tau];H_{1/2-\gamma}).}
\end{equation}


  In the hypothesis $(\mathcal{H}_4)$, $[\cdot,\cdot]_\alpha$ denotes the complex interpolation method (see \cite{Triebel}). Concerning operators satisfying the above interpolation equality we refer to 
\cite[II.1.6]{RCIDS2ED} (for instance, it is satisfied if $A$ is dissipative). The hypothesis $(\mathcal{H}_4)$ is only a technical assumption that simplifies the study 
and the notation. Without such an hypothesis the space $H_{1/2-\gamma}$ should be replace by $(\mu_0-A)^\gamma ([H,D(A)]_{1/2})$.

\begin{Definition}
System \eqref{eq1.1} is said to be 
\textit{approximately controllable}
if 
for any $y_0, y_1\in H_{1/2-\gamma}$ and for any $\varepsilon>0$, there exists $u\in L^2(0,\tau;U)$ such that the solution $y$ of \eqref{eq1.1} with $y(0)=y_0$ satisfies
$\|y(\tau)-y_1\|_{H_{1/2-\gamma}}<\varepsilon$. If the above $u$ can be chosen in a $K$-dimensional subspace $U_K\subset U$ then \eqref{eq1.1} is said to be 
\textit{approximately controllable by $K$-dimensional control}.  
\end{Definition}
Since we assume $(\mathcal{H}_3)$, the time $\tau>0$ in the above definition can be chosen arbitrary.

A problem related to the approximate controllability of \eqref{eq1.1} consists in considering controls of the form
\begin{equation}
u=\sum_{l\in \mathbb{N}} u_l v_l,\label{eq1.200}
\end{equation}
where $v\ov (v_j)_{j\in {\mathbb N}}$ is a family to be
chosen and where $\bar u{=}(u_j)_{ j\in {\mathbb N}}$ is only
depending on time and is the new control. Here we are
looking for families $v\ov (v_j)_{j\in {\mathbb N}}\in \ell^2(\mathbb{N};U)$
such that the corresponding system is still approximately
controllable.

Let us introduce for all $v\in \ell^2({\mathbb N},U)$ the operator
\begin{equation}
V(v): \ell^2({\mathbb N}) \longrightarrow H_{-\gamma},\quad  V(v) \bar
w\overset{ \mbox{\tiny def}}{=}\sum_{l\in {\mathbb N}} w_l B v_l
\quad\quad\mbox{ where }\quad \bar w=(w_l)_{l\in {\mathbb
    N}}.\label{eq1.2} 
\end{equation}
Here $\ell^2({\mathbb N})\ov \ell^2({\mathbb N},\mathbb{C})$ denotes the space of square-summable complex sequences. 
The system considered here and which is deduced from \eqref{eq1.1} can
be written as
\begin{equation}
y'+A y=V(v) \bar u.\label{eq1.4}
\end{equation}

\begin{Definition}\label{admiss-AC}
We say that $v$ is \textit{admissible} if the system
\eqref{eq1.4} is approximately controllable.
\end{Definition} 

A particular important case of the above class of problem is when the family $v$ is finite i.e. $v=(v_j)_{j=1,\dots,K}\in U^K$ for some integer $K$. Obviously, if there exists such a finite admissible family then \eqref{eq1.1} is approximately controllable by $K$-dimensional control.

Let us  consider $\varepsilon_k(\overline{\lambda}_j)$, $k=1,\dots,\ell_j$, a basis of $\Ker(\overline{\lambda}_j-A^*)$ and let us define
$$
\mathcal{E} \ov \overline{\bigoplus_{j\in \mathbb{N}}\Ker(\overline{\lambda}_j-A^*)}^{_{H_\gamma^*}}=\overline{{\rm span}_{\mathbb{C}}\,\{\varepsilon_k(\overline{\lambda}_j)\mid j\in  \mathbb{N},\;k=1,\dots,\ell_j\}}^{_{H_\gamma^*}}.
$$
Moreover, for a given family $v=(v_j)_{j\in \mathbb{N}}\in \ell^2(\mathbb{N};U)$ we introduce the bounded linear operator
\begin{equation}
  \label{Wj2}
  W_j(v) : \mathbb{C}^{\ell_j} \to \ell^2(\mathbb{N}), \quad x \mapsto
  \left(\sum_{k=1}^{\ell_j} x_k (v_l | B^*\varepsilon_k(\bar \lambda_j))_U\right)_{l\in \mathbb{N}},
\end{equation}
and for a finite family $v=(v_j)_{j=1,\dots,K}$ of $U^K$ we use the same notation for the following matrix of order $K\times \ell_j$:
\begin{equation} \label{Wj1}
W_j(v)\ov \bigg{(}(v_l \mid B^*
\varepsilon_k(\bar\lambda_j))_U\bigg{)}_{1\leq
  l\leq K,\, 1\leq k\leq \ell_j}.
\end{equation}
We are now in position to state the result on approximate controllability.
\begin{Theorem}\label{MainTheorem}
Assume $(\mathcal{H}_1)$, $(\mathcal{H}_2)$, $(\mathcal{H}_3)$, $(\mathcal{H}_4)$  and $(\mathcal{H}_5)$ and let $K\in \mathbb{N}^*$. Then the following results hold.
\begin{enumerate} 
\item System \eqref{eq1.1} is approximately controllable if 
and only if \eqref{UC} is satisfied. 
\item System \eqref{eq1.1} is approximately
controllable by a $K$-dimensional control if and only if \eqref{UC} is satisfied and 
\begin{equation}
\sup_{j\in \mathbb{N}}\ell_j\leq K.\label{CondK}
\end{equation}
\item A family $v$ is admissible in the sense of Definition \ref{admiss-AC} if and only if
\begin{equation}\label{RankCond0}
\rank W_j(v)=\ell_j\quad \forall j\in \mathbb{N}.
\end{equation}
\item Assume that \eqref{UC} is true. Then the set of admissible families of $\ell^2(\mathbb{N},U)$ forms a residual 
set of $\ell^2(\mathbb{N},U)$. Moreover, 
if $v\in \ell^2(\mathbb{N},U)$ is admissible then its orthogonal projection onto $\ell^2(\mathbb{N},B^*\mathcal{E})$ is admissible.
\item Assume that \eqref{UC} and \eqref{CondK} are true. Then the set of admissible families of $U^K$ forms a residual set of $U^K$. Moreover, if $v\in U^K$ is admissible then its orthogonal projection onto $(B^*\mathcal{E})^K$ is admissible.
\end{enumerate}
\end{Theorem}

Recall that a residual set of a topological space $X$ is the intersection of countable open and dense subsets of X.
In particular, it is a dense subset of $X$ if $X$ is a Banach space.

\begin{Remark}\label{Rem1}
Let us make some remarks on the above result:
the point 1. in the above theorem is the Fattorini result in the case $B\in \mathcal{L}(U,H_{-\gamma})$. Since we assume $(\mathcal{H}_3)$, even for $\gamma=0$, there is a small change: the set $\left\{y(\tau) \ ; \ u\in L^2((0,\tau);U)\right\}$ is dense in $H_{1/2}$ which implies that it is also dense in $H$.

To be more precise, the point 1. in the above theorem  (the Fattorini theorem) does not need the hypotheses $(\mathcal{H}_3)$, but here we are mainly interested in parabolic systems and with the possibility to use only finite dimensional controls. Without the hypothesis $(\mathcal{H}_3)$, one has to replace $H_{1/2-\gamma}$ by $H_{-\beta}$, {with suitable $\beta \leq \gamma$ such that $y\in C([0,\tau];H_{-\beta})$,} and the definition of approximate controllability has to be relaxed as follows: 
system \eqref{eq1.1} is said to be \textit{approximately controllable} if $\cup_{\tau>0}\Ran\Phi_\tau$ is dense in $H_{-\beta}$ where 
$$
\Phi_\tau (u)\ov \int_0^\tau e^{A(\tau-s)}Bu(s) \ {\rm d}s.
$$
It can be checked that system \eqref{eq1.1} is approximately controllable, if and only if, for any $y_1\in H_{-\beta}$ and for any
$\varepsilon>0$, there exists $u\in L^2(0,+\infty;U)$ such that the
solution $y$ of \eqref{eq1.1} with $y(0)=0$ satisfies $\|y(\tau)-y_1\|_{H_{-\beta}}<\varepsilon$ for some $\tau$, depending in general on $y_1$ and $\varepsilon$.
\end{Remark}

The condition \eqref{UC} is also a criterion for stabilizability property. We say that $(A, B)$ is  
\textit{stabilizable} if there exist $F\in \mathcal{L}(H,U)$ and constants 
$C>0$, $\epsilon>0$ such that the solutions of \eqref{eq1.1} with $u=Fy$ obey $\|y(t)\|_H\leq C e^{-\epsilon t}\|y(0)\|_H$ for all $t\geq 0$. Note that for a given $\sigma>0$ 
it is easily seen 
that $(A+\sigma, B)$ is stabilizable if and only if there exist $F\in\mathcal{L}(H,U)$ and $C>0$, $\epsilon>0$ such that the solutions of \eqref{eq1.1} with $u=Fy$ satisfy
\begin{equation}
\|y(t)\|_H\leq C e^{-(\sigma+\epsilon) t}\|y(0)\|_H\quad (t\geq 0).\label{EstStab}
\end{equation}
\begin{Definition}\label{admiss-ES}
We say that $(A, B)$ is \textit{stabilizable by $K$-dimensional control} if there exists a $K$-dimensional subspace $U_K\subset U$  
such that $(A,B|_{U_K})$ is stabilizable. A linearly independent family $(v_j)_{j=1,\dots,K}$ of $U^K$ 
generating such a $U_K$ (i.e. $U_K=\Span_{\mathbb{C}}\{v_j,j=1,\dots,K\}$) is said to be \textit{admissible for stabilizability of $(A,B)$}.
\end{Definition}
If $(v_j)_{j=1,\dots,K}$ of $U^K$  is admissible for stabilizability, then the stabilizing finite rank feedback law $F\in \mathcal{L}(H,U)$ can be represented 
  as follows:
\begin{equation}\label{FeedbackSum}
F y(t)= \sum_{j=1}^K (y(t), \widehat \varepsilon_j)_H v_j.
\end{equation}

Note that above definitions are consistent since, by combining $(\mathcal{H}_3)$, $(\mathcal{H}_5)$,   and $F\in\mathcal{L}(H,U)$ with a perturbation argument, we have that $A+BF$ generates an analytic semigroup on $H$ and in particular that $y(t)\in C([0,\infty);H)$.

In the case of stabilization, the condition \eqref{UC} is replaced by
\begin{equation}
    \forall\varepsilon\in {\cal D}(A^*), \quad \forall \lambda\in
    \mathbb{C},\; \Re\lambda\geq -\sigma\quad
A^*\varepsilon= \lambda \varepsilon \quad \text{and} \quad B^*\varepsilon=0
\quad\Longrightarrow\quad 
\varepsilon = 0.  \tag{$\text{UC}_{\sigma}$}\label{UCstab}
  \end{equation}
The space $\mathcal{E}$ used in Theorem \ref{MainTheorem} can be replaced by the following space depending on  $\sigma>0$:
\begin{equation}
\mathcal{E}_\sigma \ov \bigoplus_{\Re\lambda_j\geq -\sigma}\Ker(\overline{\lambda}_j-A^*)=\Span_{\mathbb{C}}\{\varepsilon_k(\overline{\lambda}_j)\mid \Re\lambda_j\geq -\sigma,\;k=1,\dots,\ell_j\}.\label{DefEsigma}
\end{equation}
  
We also consider the hypothesis that completes $(\mathcal{H}_1)$
\begin{enumerate}
\item[$(\mathcal{H}_6)$] The spectrum of $A$ has no finite {cluster point}. 
\end{enumerate}
Note that a sufficient condition for $(\mathcal{H}_6)$ is that $A$ has compact resolvent.

The main result for the stabilization of \eqref{eq1.1} is
\begin{Theorem}\label{CorStab}
Assume $(\mathcal{H}_1)$, $(\mathcal{H}_3)$, $(\mathcal{H}_5)$ and $(\mathcal{H}_6)$. 
Let $K\in \mathbb{N}^*$ and $\sigma>0$. Then the following results hold.
\begin{enumerate}
\item The pair $(A+\sigma, B)$ is stabilizable if and only if \eqref{UCstab} is satisfied. 
\item The pair $(A+\sigma, B)$ is stabilizable by a $K$-dimensional control if and only if \eqref{UCstab} is satisfied and
\begin{equation}
\sup_{\Re \lambda_j\geq -\sigma }\ell_j\leq K.\label{CondKsigma}
 \end{equation}
\item A family $v=(v_j)_{j=1,\dots,K}$ of $U^K$ is admissible for stabilizability  of $(A+\sigma, B)$ (in the sense of Definition \ref{admiss-ES}) if and only if
\begin{equation}
\rank W_j(v)=\ell_j\quad \forall \lambda_j,\;  \Re\lambda_j\geq -\sigma.\label{CondWk}
\end{equation}
\item Assume that \eqref{UCstab} and \eqref{CondKsigma} are true. Then the set of admissible families for stabilizability  of $(A+\sigma, B)$ forms a residual set of $U^K$. 
Moreover, if $v$ is admissible for stabilizability  of $(A+\sigma, B)$ then its orthogonal projection onto $(B^*\mathcal{E}_\sigma)^K$ is admissible for stabilizability
 of $(A+\sigma, B)$.
\end{enumerate}
 \end{Theorem}

\begin{Remark}\label{Rk1}
Using the above result, one can deduce the stabilization of a class of nonlinear parabolic systems (see Theorem \ref{ThmNL} and Theorem \ref{ThmNL2}).

Let us also underline that in many practical examples, the system \eqref{eq1.1} is originally defined on a real Hilbert space for control functions with values in a real control space.
Thus above approximate controllability and stabilization results are also stated in the real case (see Theorem \ref{MainTheoremReal} and Theorem \ref{CorStabReal} below). {We emphasize that, even in the real case, the minimal value of $K$ is the maximum of the geometric multiplicities $\ell_j$. It is independent of algebraic multiplicities and of real or complex nature of the eigenvalues as it is suggested in \cite[Rem. 3.9]{B-T-2004}.}
\end{Remark}
\begin{Remark} 
The stabilizing feedback control $u=F y$ given by Theorem \ref{CorStab} is constructed from a finite dimensional projected system and, as a consequence, it is infinitely time differentiable: 
\begin{equation}\label{TimeRegCont}
\quad \forall n\in \mathbb{N},\quad u=F y\in H^n(U).
\end{equation}
Actually, the construction given by Subsection \ref{subsect3.1} shows that in \eqref{FeedbackSum} the family $\{\widehat \varepsilon_j\mid j=1,\dots, K\}$ can be chosen in the set $\mathcal{E}_\sigma$ defined by \eqref{DefEsigma}. 
Then $F$ can be extended as an operator from $[D((A^*)^n)]'$ onto $U$ for all $n\in \mathbb{N}$. In particular, from $F\in \mathcal{L}([D(A^*)]',U)$ and $y'\in L^2([D(A^*)]')$ we first deduce that $u=F y$ belongs to $H^1(U)$, and thus \eqref{TimeRegCont} follows from an inductive argument  by successively differentiating equation \eqref{eq1.1}.
\end{Remark}
Using the Theorem \ref{CorStab} and Remark \ref{Rk1}, we can deal with several nonlinear parabolic real systems such as systems of heat equations or Navier--Stokes type systems. For instance, we show in Section \ref{sect5} how to obtain the local stabilizability around a stationary state $(w^S,\theta^S)$ of the magnetohydrodynamic (MHD) system
\begin{equation}\label{MHDIntro}
 \left\{
\begin{split}
 w_t-\Delta w+(w\cdot \nabla) w-({\rm curl\, \theta})\times \theta+\nabla r&=f^S+{\bf 1}_\omega u^1 &\quad \mbox{ in }Q,\\
 \theta_t+{\rm curl}({\rm curl}\,\theta)-{\rm curl}(w\times \theta) &={\bf 1}_\omega \Tilde u^2 &\quad \mbox{ in }Q,\\
{\rm div}\, w={\rm div}\, \theta&=0&\quad \mbox{ in }Q,\\
w=w^S,\quad \theta\cdot n&=\theta^S\cdot n&\quad \mbox{ on }\Sigma,\\
({\rm curl}\,\theta-w\times \theta)\times n&=({\rm curl}\,\theta^S-w^S\times \theta^S)\times n&\quad \mbox{ on }\Sigma,\\
w(0)=w^S+y_0,\quad \theta(0)&=\theta^S+\vartheta_0&\quad \mbox{ in }\Omega,
\end{split}
\right.
\end{equation}
{with a control $\Tilde u^2$ in the magnetic field equation that satisfies:
\begin{equation}\label{Propu2}
\begin{split}
{\rm div\,}({\bf 1}_\omega \Tilde u^2)&=0\; \mbox{ in }\; Q,\\
({\bf 1}_\omega \Tilde u^2)\cdot n&=0\; \mbox{ on }\; \Sigma.
\end{split}
\end{equation}}
In above settings, $\Omega$ is a bounded and simply connected domain of $\mathbb{R}^3$, $\omega$ is a nonempty open subset of $\Omega$, $Q\ov (0,+\infty)\times \Omega$, $\Sigma\ov (0,+\infty)\times  \partial\Omega$, ${\bf 1}_\omega$ is the extension operator defined on $(L^2(\omega))^3$ 
by ${\bf 1}_\omega(y)(x)=y(x)$ if $x\in \omega$ and ${\bf 1}_\omega(y)(x)=0$ else, and $u^1, \Tilde u^2$ are control functions in 
$(L^2((0,T)\times \omega;\mathbb{R}))^3$. {This result extends (and also improves) the 2d stabilization result \cite{Lefter2010}. In Section \ref{sect5}, we also prove a boundary stabilization result for a nonlinear micropolar fluid system for which controllability issue has been considered in \cite{Guerrero2007}. Those stabilization results are both based on the proof of the unique continuation property  \eqref{UC} and require the use of local Carleman inequalities for coupled Stokes-type equations. These inequalities are proved in the Appendix.}

\medskip
\par The first works dealing with approximate controllability of infinite dimensional linear systems are due to 
H. O. Fattorini in the pioneer papers \cite{Fattorini1966, Fattorini1967}. While \cite{Fattorini1967} focus on the particular case of self-adjoint generators, general semigroup 
generators are considered in \cite{Fattorini1966} where the above mentioned infinite dimensional Hautus test \eqref{UC} is proposed. 
Very surprisingly, the last quoted work has been published three years earlier than M. L. J. Hautus famous paper \cite{Hautus1969}. 
It is also suggested in \cite{Fattorini1966} that one can find a finite rank input operator such that approximate 
controllability holds provided that the rank is greater or equal than the maximum of the geometric multiplicities of $A$. However, nothing is said about the way of 
constructing such a finite rank input operator. 
Since Fattorini's paper is almost 50 years old, it is of course expected that there has been additional work on the subject, dealing in particular with finite dimensional controls. However, most of the results available in the literature concern self-adjoint or diagonal operators and are stated in a much less general way than Fattorini's result. For instance, in \cite{Triggiani1976} Triggiani provide a generalization of the Kalman rank condition which reduced to \eqref{RankCond0} in the case of normal operator with compact resolvent, and permits to characterize admissible finite dimensional control subspaces (see \cite[eq.(3.5)]{Triggiani1976}). About stabilizability problems, the idea of using finite dimensional feedback law stabilizing linear parabolic systems goes back to \cite{Triggiani1975, Triggiani1980} and the rank conditions \eqref{CondWk} arise as a stabilizability criterion in the case of self-adjoint operators with compact resolvent (see \cite[eq.(7.2)]{Triggiani1975}). This spectral rank criterion can also be found in the modern literature but still for self-adjoint, or at least diagonal, operators. 
We refer for instance to the book of Curtain and Zwart where the rank condition is stated for self-adjoint or diagonal operator having simple spectrum as well as a Riesz basis of eigenvectors (see \cite[Thm. 4.2.1 and Thm. 4.3.2]{CurtainZwart1995}), or to the book of Bensoussan, Da Prato, Delfour and Mitter still for diagonal operators (see \cite[Thm. 4.2]{RCIDS2ED}). However, in numerous applications the operator $A$ may not be diagonal or may not have a Riesz basis (for a simple example of non selfadjoint operator whose eigenvectors are complete but do not form a Riesz basis see for instance \cite{Davies1999}). As a consequence, we present here the finite dimensional control version of Fattorini's theorem written in full generality: see Theorem \ref{MainTheorem} and Remark \ref{Rem1}. 


 \medskip
\par The second reason motivating the writing down of Theorem \ref{MainTheorem} is the connected stabilization issues that are stated in Theorem \ref{CorStab}. Indeed, we must underline that the fact that \eqref{UCstab} (or its rank conditions version \eqref{CondWk}) can be used as a general criterion for the stabilizability of parabolic systems went unnoticed in many recent works about stabilization of fluid-flow systems. We can quote for instance the papers of Barbu, Lasiecka and Triggiani \cite{B-L-T-2006, BarbuLasieckaTriggiani2004, B-L-T-2007} where, to obtain the stabilizability of the linearized Navier-Stokes equations, the authors make the useless additional assumption that the projected operator $A_N$ is diagonal (and even postulate that it is also necessary for exact controllability of the projected system, see \cite[Rem. 3.6.2 and Rem. 3.6.4]{BarbuLasieckaTriggiani2004}). Or the work of Raymond and Thevenet \cite{RaymondThevenet2009} where the exact controllability of the finite dimensional unstable projected system is deduced  from the stabilizability by infinite dimensional control proved in \cite{Raymond_2005_2}, itself relying on the null controllability of the Oseen equations \cite{exactcontollFGIP}. We must also mention the series of works of Fursikov \cite{Fursikov2001-0, Fursikov2001, Fursikov2004} where stabilization of the Oseen equations and of the Navier-Stokes equations is achieved through a subtle combination of an extension of the domain procedure with a spectral decomposition. Up to our knowledge, it is the first work that reduces the PDE stabilizability question to a unique continuation property on an eigenvalue problem (see Fabre and Lebeau's type theorem proved in \cite[Thm 4.2]{Fursikov2004}), and for a non necessarily diagonal underlying linear operator. There, $B^*$ correspond to an internal observation and this permits to deduce the linear independence 
of the whole family $\{B^*\varepsilon_k(\lambda_j)\mid \Re\lambda_j\geq -\sigma,\;k=1,\dots,\ell_j \}$ which is at the basis of the construction of a key extension operator. Note that such a linear independence property is also used in \cite{B-T-2004} to construct finite dimensional distributed control. However, it is stronger than \eqref{UCstab} and proper to internal observation since linear independence of $\{B^*\varepsilon_k(\lambda_j)\mid \Re\lambda_j\geq -\sigma,\;k=1,\dots,\ell_j \}$ is false for boundary control in some simple geometrical situations (see the end of the introduction of \cite{BT}). Note also that in the above quoted works the question of finding the minimal number $K$ for which the stabilizability holds is not addressed or only partially.

The major interest of the Fattorini criterion is that proving \eqref{UC} is an easy alternative 
to obtain the finite cost condition needed to construct stabilizing feedback 
law from well-posed optimal quadratic problem, see \cite[Remark 2]{BADRA-DCDS-A2011}. For instance, boundary feedback stabilizability of the 
Navier-Stokes system can be simply reduced to a uniqueness result of type \cite{FL1996} and avoid more 
sophisticate approaches for the nonstationary system as in \cite{Imanuvilov2001, exactcontollFGIP} or \cite{Fursikov2001, Fursikov2004}. 
Moreover, a systematic generalization of the last quoted works to other analogous more complex systems 
such as coupled Stokes type systems or fluid--structure systems is not straightforward. Here, with two examples of flow systems described by coupled Stokes type equations,
we sketch a systematic procedure relying on \eqref{UC} for checking the stabilizability of nonlinear system. 
We prove local Carleman estimates that permit to check \eqref{UC} for general coupled Navier-Stokes systems and we deduce feedback and dynamical stabilization of nonlinear  MHD and micropolar systems
(see Corollary \ref{CarlLocalStokes}, Theorem \ref{StabMHD2D} and Theorem \ref{StabMicro3D}). Concerning the use of Fattorini's criterion for feedback stabilization of 
fluid--structure system we refer to \cite{BT3,BT4}. Finally, let us underline that a stabilizability property can also be used to tackle some controllability issues: 
we refer for instance to the work \cite{CoronTrelat2004} where the authors prove the global controllability  to steady trajectories of a 1d nonlinear heat equation by using a stabilization procedure involving only the unstable modes.

\medskip
\par The rest of the paper is organized as follows. Section \ref{sect2} is dedicated to approximate controllability issue. We prove Theorem \ref{MainTheorem} in subsection \ref{subsect2.1}, we make some comments in subsection \ref{subsect2.2} and subsection \ref{subsect2.3} is dedicated to the real version of Theorem \ref{MainTheorem}. Thus, we consider a simple example involving a system of coupled heat equations in subsection \ref{subsect2.4}.
Section \ref{sectstab0} is dedicated to stabilizability issue. We prove Theorem \ref{CorStab} in subsection \ref{subsect3.1}, subsection \ref{subsect3.2} is dedicated to nonlinear stabilization theorems and subsection \ref{subsect3.3} concerns the real case. 
Thus, as a simple example of application, we consider the minimal rank stabilizing feedback law issue for the heat equation in a rectangular domain. Section  \ref{sect5} is concerned with two applications to the stabilization of incompressible coupled Navier-Stokes type systems. Subsection \ref{stabMHD} is concerned with the MHD system \eqref{MHDIntro} and Subsection \ref{stabMicro} deals with the boundary feedback stabilizability of a micropolar fluid system. Finally, the appendix is dedicated to the proof of local Carleman inequalities for the Stokes equations and on their applications to obtain uniqueness theorem that permits to check \eqref{UC} related to MHD and micropolar fluid stabilization problems.
\section{Approximate controllability}\label{sect2}
\subsection{Proof of Theorem \ref{MainTheorem}}\label{subsect2.1}
First, let us recall that the solution $y$ of \eqref{eq1.1} such that $y(0)=y_0$ can be written as $y(t)=e^{A t} y_0+\Phi_t (u)$ with
\begin{equation}
\Phi_\tau (u)=\int_0^\tau e^{A(\tau-s)}Bu(s) \ {\rm d}s   .\label{DefPhitau}
\end{equation}
More precisely, using $(\mathcal{H}_3)$, one can check that $\Phi_\tau\in \mathcal{L}(L^2((0,\infty);U);H_{1/2-\gamma})$ and that
$$
\overline{\Ran\Phi_\tau}^{_{H_{1/2-\gamma}}}=H_{1/2-\gamma} \Leftrightarrow \ker\Phi_\tau^*=\{0\},
$$ 
where the adjoint of $\Phi_\tau$ is given by:
\begin{equation}\label{PhiAdj}
\forall \varepsilon\in H_{\gamma-1/2}^*,\qquad (\Phi_\tau^* \varepsilon)(t) = 
\left\{
\begin{array}{ll}
B^* e^{A^* (\tau-t)}\varepsilon &\mbox{ for }t\in [0,\tau],\\
0 &\mbox{ for }t>\tau.\\
\end{array}
\right.
\end{equation}
Therefore, approximate controllability of \eqref{eq1.1} is equivalent to the following condition
\begin{equation}\label{AppConttau}
\forall \varepsilon\in H_{\gamma-1/2}^*,\qquad  B^*e^{A^* t} \varepsilon = 0 \quad (t\in (0,\tau)) \Longrightarrow \quad \varepsilon=0.
\end{equation}
Moreover, since the semigroup is analytic, this condition is equivalent to 
\begin{equation}\label{II0.0}
\forall \varepsilon\in H_{\gamma-1/2}^*,\qquad B^*e^{A^* t} \varepsilon = 0 \quad (t\in (0,+\infty)) \Longrightarrow \quad \varepsilon=0.
\end{equation}

\paragraph{Proof of 1.}
Let us first prove that \eqref{UC} implies \eqref{II0.0}. 
For that, assume $\varepsilon\in H_{\gamma-1/2}^*$ satisfies
\begin{equation}
  \label{II0.10}
  B^*e^{A^* t} \varepsilon = 0 \quad (t\in (0,+\infty)).
\end{equation}
The Laplace transform of \eqref{II0.10} (see \cite[Thm. 4.3.7]{TucsnakWeiss}) first yields:
$$
B^* (\lambda-A^*)^{-1}\varepsilon=\int_0^\infty e^{-\lambda t} B^*e^{A^* t} \varepsilon \  {\rm d t}=0\quad \mbox{ for }\; \Re \lambda\geq \mu_0> \sup_{j\in \mathbb{N}} \Re\lambda_j
$$
and next $B^*(\lambda-A^*)^{-1}\varepsilon = 0$ for all for all $\lambda\in \rho(A^*)$ by analytic continuation. Now define
\begin{equation}\nonumber
R_{-n}^*(\Bar \lambda_j)=\frac{1}{2\pi \imath}\int_{|\lambda-\bar
  \lambda_j|=\delta}(\lambda-\Bar \lambda_j)^{n-1}(\lambda-A^*)^{-1}{\rm d\lambda},\quad  n\in \{1,\dots, m(\lambda_j)\}
\end{equation}
where $\delta>0$ is the same as in \eqref{DefProj}. Since $B^*(\lambda-A^*)^{-1}\varepsilon = 0$ for all $\lambda\in \rho(A^*)$ we deduce that
\begin{equation}
B^* R_{-n}^*(\Bar \lambda_j)\varepsilon=0\qquad \forall n\in \{1,\dots, m(\lambda_j)\}.\label{uc1.60}
\end{equation}
Moreover, by easy computations we verify that 
\begin{equation}
R_{-n-1}^*(\Bar \lambda_j)+(\Bar \lambda_j-A^*)R_{-n}^*(\Bar \lambda_j)=0\qquad \forall n\in \{1,\dots, m(\lambda_j)-1\}.\label{uc1.600}
\end{equation}
Finally, let us prove that \eqref{UC} with an inductive argument using \eqref{uc1.60}, \eqref{uc1.600} yields
\begin{equation}\label{R-1null}
R_{-1}^*(\Bar \lambda_j)\varepsilon=0.
\end{equation}
By \eqref{uc1.600} we have $R_{-m(\lambda_j)}^*(\Bar \lambda_j)=(A^*-\Bar \lambda_j)^{m(\lambda_j)-1}R_{-1}^*(\Bar \lambda_j)$, and 
since $\mbox{Ran}(R_{-1}(\lambda_j))=\Ker(\lambda_j-A)^{m(\lambda_j)}$ we deduce that 
$R_{-m(\lambda_j)}^*(\Bar \lambda_j)\varepsilon\in \ker(\bar \lambda_j-A^*)$. Then \eqref{uc1.60} with $n=m(\lambda_j)$ 
combined with \eqref{UC} first gives $R_{-m(\lambda_j)}^*(\Bar \lambda_j)\varepsilon=0$. As a consequence, \eqref{uc1.600} with $n=m(\lambda_j)-1$ yields 
$R_{-m(\lambda_j)+1}^*(\Bar \lambda_j)\varepsilon\in \ker(\bar \lambda_j-A^*)$, and \eqref{uc1.60} with $n=m(\lambda_j)-1$ combined with \eqref{UC} then gives 
$R_{-m(\lambda_j)+1}^*(\Bar \lambda_j)\varepsilon=0$. By reiterating the argument we successively obtain $R_{-n}^*(\Bar \lambda_j)\varepsilon=0$ from $n=m(\lambda_j)$ until $n=1$. 
Finally, since \eqref{R-1null} holds, we have $\varepsilon\in \Ran(R_{-1}(\lambda_j))^\perp$ for all $j\in \mathbb{N}$, 
and with assumption $(\mathcal{H}_2)$ we get $\varepsilon=0$.

\par Conversely, let us now prove that \eqref{II0.0} implies
\eqref{UC}. Suppose that  \eqref{UC} is false: there exist $j_0\in \mathbb{N}$
and an eigenvector $\varepsilon(\overline\lambda_{j_0})\neq 0$ (associated with the eigenvalue $\overline\lambda_{j_0}$) 
such that
$B^*\varepsilon(\overline\lambda_{j_0})=0$. Moreover, since we have (see \cite[Chapter 2]{PAZY})
$$
e^{A^* t}\varepsilon(\overline\lambda_{j_0})=e^{\bar \lambda_j t}\varepsilon(\overline\lambda_{j_0}),
$$
then $B^*e^{A^* t}\varepsilon(\overline\lambda_{j_0})=0$ for all $t\in (0,+\infty)$ 
and \eqref{II0.0} is false.

\paragraph{Proof of point 3.} Let us consider $v \in \ell^2(\mathbb{N},U)$. System \eqref{eq1.4} is approximately controllable if and only if 
the corresponding condition \eqref{UC} holds i.e.
\begin{equation}
    \forall\varepsilon\in {\cal D}(A^*), \quad \forall \lambda\in
    \mathbb{C}, \quad
A^*\varepsilon= \lambda \varepsilon \quad \text{and} \quad V(v)^*\varepsilon=0
\quad\Longrightarrow\quad 
\varepsilon = 0.\label{II0.6}
\end{equation}
It can be checked that the adjoint operator of $V(v)$ is given by
\begin{equation}
V(v)^*:H_\gamma^* \longrightarrow \ell^2({\mathbb N}),
\quad V(v)^*  \varepsilon=((v_l| B^* \varepsilon)_{U})_{ l\in {\mathbb N}},\label{eq1.3}
\end{equation}
and it is clear that \eqref{II0.6} is equivalent to \eqref{RankCond0}. We deduce that $v$ is admissible if an only if \eqref{RankCond0} is satisfied.

\paragraph{Proof of point 4.} Let us prove that the set of families $v\in \ell^2(\mathbb{N},U)$ satisfying \eqref{RankCond0} is the intersection of a 
countable family of open and dense subsets of $\ell^2(\mathbb{N},U)$.

For all $n\in \mathbb{N}$, we denote by $\mathcal{R}_n$ the 
subset of $\mathcal{L}(\mathbb{C}^n, \ell^2(\mathbb{N}))$ composed by
the linear bounded mapping of rank $n$. It is well-known that $\mathcal{R}_n$ is an open
and dense subset $\mathcal{L}(\mathbb{C}^n, \ell^2(\mathbb{N}))$. 

Let us recall that the linear mapping defined by \eqref{Wj2}
$$
W_j:\ell^2({\mathbb N},U)\to \mathcal{L}(\mathbb{C}^{\ell_j}, \ell^2(\mathbb{N}))
$$
is bounded. The set of families $v\in
\ell^2(\mathbb{N},U)$
such that  \eqref{RankCond0} holds for all $j\in\mathbb{N}$ can be written as 
$$
{\cal A} = \bigcap_{j\in \mathbb{N}}W_j^{-1}(\mathcal{R}_{\ell_j}).
$$
In order to prove that this set is residual, we prove that for all 
$j\in \mathbb{N}$, $W_j^{-1}(\mathcal{R}_{\ell_j})$ is an open
dense subset of $\ell^2(\mathbb{N},U)$. Let us fix $j\in \mathbb{N}$.

First, by using that $W_j$ is a continuous mapping, we deduce
that $W_j^{-1}(\mathcal{R}_{\ell_j})$ is an open subset of
$\ell^2(\mathbb{N},U)$. 

To prove the density, we proceed as in \cite{BT}: let us consider $v\in \ell^2(\mathbb{N},U)$.
Since \eqref{UC} holds, the family $B^* \varepsilon_i(\bar \lambda_j)$ for $i=1,\ldots,\ell_j$ is linearly independent and 
thus it can be seen that $v\in W_j^{-1}(\mathcal{R}_{\ell_j})$ if and only if the span of $v$ contains 
$B^* \varepsilon_i(\bar \lambda_j)$ for $i=1,\ldots,\ell_j$. In particular, if we decompose the first $\ell_j$ vectors of $v$ as
$$
v_k=\widehat v_k+\widetilde v_k, \quad \text{with}\ \widehat v_k\in \Span\{B^* \varepsilon_i(\bar \lambda_j), \ i=1,\ldots,\ell_j\}, \ 
\widetilde v_k\in \Span\{B^* \varepsilon_i(\bar \lambda_j), \ i=1,\ldots,\ell_j\}^\perp,
$$
then we see that $v$ is admissible if $(\widehat v_1,\ldots,\widehat v_{\ell_j})$ is a basis of $\Span\{B^* \varepsilon_i(\bar \lambda_j), \ i=1,\ldots,\ell_j\}$.
This is equivalent to say that the matrix of $(\widehat v_1,\ldots,\widehat v_{\ell_j})$ in the basis 
$(B^* \varepsilon_1(\bar \lambda_j),\ldots,B^* \varepsilon_{\ell_j}(\bar \lambda_j))$ is invertible. However such a matrix can always be approximated by invertible matrices and thus any $v$ can be approximated by $v^{\eta}\in \ell^2(\mathbb{N},U)$ such that $(\widehat v_1^\eta,\ldots,\widehat v_{\ell_j}^\eta)$ 
is a basis of $\Span\{B^* \varepsilon_i(\bar \lambda_j), \ i=1,\ldots,\ell_j\}$. Since such a $v^\eta$ is admissible this concludes the proof.

\paragraph{Proof of point 5.} The proof is completely similar to the proof of point 4. Just, observe 
that $K\geq \ell_j$ is clearly a necessary condition for \eqref{RankCond0} with $W_j(v)$ now defined by \eqref{Wj1}.

\paragraph{Proof of point 2.} The fact that $\eqref{UC}$ and \eqref{CondK} imply that \eqref{eq1.1} is approximately controllable follows from point 5. To
prove the converse implication one can use point 1 and the fact that if \eqref{CondK} does not hold, \eqref{RankCond0} is false.

\subsection{Remarks on Theorem \ref{MainTheorem} and on the hypotheses $(\mathcal{H}_1)$--$(\mathcal{H}_5)$}\label{subsect2.2}

\begin{Remark}
In Theorem \ref{MainTheorem}, one can only suppose $\gamma\geq 0$ in $(\mathcal{H}_5)$ and replace the hypothesis $(\mathcal{H}_3)$ by
\begin{equation}
  {\rm Ran}\Phi_\tau\subset H_{-\beta},\quad  \beta\in [\gamma-1,\gamma]\label{newPhiX}
\end{equation}
for some $\tau$ ($\Phi_\tau$ is defined by \eqref{DefPhitau}). First let us remark in that case \eqref{newPhiX} holds for all $\tau>0$ and $\Phi_\tau$ is 
bounded from $L^2(0,+\infty;U)$ into $H_{-\beta}$ for all $\tau>0$, see \cite[Prop. 4.2.2.]{TucsnakWeiss}.

Let us note that in the quoted work, the result is obtained for $\gamma=1$ and $\beta=0$, but it can be extended to $\gamma\geq 0$ and 
$\beta\in [ \gamma-1, \gamma]$ with the change of variable 
$\Tilde \Phi_\tau\ov (\mu_0-A)^{-\beta} \Phi_\tau$, because for such values $\gamma$, $\beta$ we have $(\mu_0-A)^{-\beta}B: U\to H_{-1}$ bounded.

However, because the expression of $\Phi_\tau^*$ given by \eqref{PhiAdj} is not necessarily valid for $\varepsilon\in H_{\beta}^*$, expressions \eqref{AppConttau} or \eqref{II0.0} must be slightly modified. 
Following \cite[Par. 4.3 and Par. 4.4]{TucsnakWeiss} we have 
\begin{equation}\label{PhiAdj2}
\forall \varepsilon\in H_\beta^*,\quad (\Phi_\tau^* \varepsilon)(t) = 
\left\{
\begin{array}{ll}
(\Psi \varepsilon)(\tau-t) &\mbox{ for }t\in [0,\tau]\\
0 &\mbox{ for }t>\tau,\\
\end{array}
\right.
\end{equation}
where $\Psi$ denotes the uniquely determined \textit{extended output map} $\Psi:H_{\beta}^*\to L_{\rm loc}^2([0,+\infty); U)$ satisfying
$$
(\Psi \varepsilon)(t)=B^*e^{A^* t}\varepsilon\qquad \forall \varepsilon\in H_{\gamma}^*, \; t\geq 0.
$$
About existence and uniqueness of such $\Psi$, see \cite[p. 123]{TucsnakWeiss}\footnote{ Since the quoted work only applies for $\gamma=1$ and $ \beta=0$ it can be used to 
justify the existence and the uniqueness of $\Tilde \Psi:H\to L_{\rm loc}^2([0,+\infty); U)$ satisfying 
$(\Tilde \Psi \varepsilon)(t)=B^*(\mu_0-A^*)^{-\beta} e^{A^* t}\varepsilon$ for all $\varepsilon\in H_{1}^*$, $t\geq 0$,
and we check that $\Psi = \Tilde \Psi (\mu_0-A^*)^{\beta}$ obeys the desired property.}.
Then from 
$$\overline{\cup_{\tau>0} \Ran\Phi_\tau}^{_H}=H \Longleftrightarrow \cap_{\tau>0} \ker \Phi_\tau^*=\{0\}$$ 
we also get that \eqref{eq1.1} is 
approximately controllable if and only if
\begin{equation}\label{AppConttauter}
\forall \varepsilon\in H_\beta^*,\quad ({\Psi} \varepsilon)(t) = 0 \quad (t\in (0,\infty)) \Longrightarrow \quad \varepsilon=0.
\end{equation}
From this relation, one can follow the proof of the point 1 of Theorem \ref{MainTheorem}.
\end{Remark}

\begin{Remark}
According to Keldy's Theorem, assumption $(\mathcal{H}_2)$ is satisfied by a class of perturbations of self-adjoint operator: 
if $A_0$ is a self-adjoint operator in $H$ with compact resolvent, if $A_1$ is an operator such that $A_1(-A_0)^{-\alpha}$ 
is bounded for some $0\leq \alpha<1$, and if the eigenvalues $(\lambda_j)_{j\in \mathbb{N}}$ of $A_0$ satisfy for some $1\leq p<+\infty$:
\begin{equation}
\sum_{j\in \mathbb{N}}\frac{1}{|\lambda_j|^{p}}<+\infty,\label{SumVP}
\end{equation}
then $A=A_0+A_1$ with domain $D(A)=D(A_0)$ is closed and its family of root vectors 
is complete in $H$, see \cite[Thm.10.1, p.276]{Krein} or \cite[Thm.5.6.1.3, p.394]{Triebel} 
combined with \cite[Thm.5.6.1.1, p.392 and Lem.5.6.1.2, p.395]{Triebel}. 
Note that Weyl's formula ensures that \eqref{SumVP} is satisfied by regular self-adjoint elliptic operator, see \cite[p.395]{Triebel}.
\end{Remark}

\begin{Remark} 
If $-A$ is positive we can define its fractional powers $(-A)^{\alpha}$ with $0<\alpha<1$, and we easily verify that
the set of eigenvectors of $(-A)^{\alpha}$ 
and $(-A^*)^{\alpha}$ coincide with the set of eigenvectors of $A$ and  $A^*$ respectively. Then assumptions $(\mathcal{H}_1)$, $(\mathcal{H}_2)$ and \eqref{UC} holds for 
$-(-A)^{\alpha}$ if they hold for $A$. Moreover, if $(A,B)$ satisfies $(\mathcal{H}_5)$ for some $\gamma\geq 0$ then $(-(-A)^\alpha,B)$ satisfies $(\mathcal{H}_5)$ 
for $\frac{\gamma}{\alpha}$. Then under assumptions $(\mathcal{H}_1)$, $(\mathcal{H}_2)$, $(\mathcal{H}_5)$ conclusions of Theorems \ref{MainTheorem} are also true for 
$-(-A)^{\alpha}$ with $0<\alpha<1$. This should be compared with the fact that the null controllability of $(A,B)$ does not imply the null controllability of 
$(-(-A)^\alpha,B)$ for $0<\alpha<1/2$, see \cite{MicuZuazua}. 
\end{Remark}

\subsection{Approximate controllability of real systems}\label{subsect2.3}

In many practical examples system \eqref{eq1.1} is originally defined on a real Hilbert space for control functions with values in a real control space: $A$ is originally defined 
as an unbounded operator on a real Hilbert space $G$ (i.e. $A: \mathcal{D}(A)\subset G\to G$) and the input $B$ is originally 
defined as a bounded operator from a real Hilbert space $W$ 
into the real space $G_{-\gamma}$ (defined as $H_{-\gamma}$ in \eqref{Halpha} but now from real operators $A$, $A^*$). Then for $y(0)\in G_{1/2-\gamma}$ and $u\in L^2(0,+\infty;W)$ 
the trajectories $t\mapsto y(t)$ are continuous with values in the 
real Hilbert space $G_{1/2-\gamma}$. In such situation, \eqref{eq1.1} is referred as a 
\textit{real system} and all the above definitions of approximate controllability and admissibility can be stated for real spaces in the same manner as it has been done for complex spaces. 
It follows that complex spaces $H$ and $U$ are simply the complexified spaces $H=G+\imath G$ and $U=W+\imath W$ and to recover the above complex framework it suffices to consider 
extensions of $A$ and $B$ to $H$ and $U$ respectively. 

It is clear that the approximate controllability 
of the complex system implies the approximate controllability of the real system. This follows by remarking that the complex system \eqref{eq1.1} can be decompose in two uncoupled 
real systems corresponding to real and imaginary parts of \eqref{eq1.1}. The same argument yields that if $(v_j)_j$ is an admissible family of $U$ for the complex 
system then $(\Re v_j,\Im v_j)_j$ is an admissible family of $W$ for the real system. In particular, if the complex system is approximately controllable by $K$-dimensional control then 
the real system is approximately controllable by $2K$-dimensional control. However, \eqref{CondK} and \eqref{UC} are also sufficient for approximate controllability by 
$K$-dimensional control of the real system. A slight modification of the proof of Theorem \ref{MainTheorem} permits to obtain a real version of Theorem \ref{MainTheorem} 
stated in Theorem \ref{MainTheoremReal} below. 

Before stating this result, we need some additional notation. Recall that for each $\lambda_j$ the family $\varepsilon_k(\overline{\lambda}_j)$, $k=1,\dots,\ell_j$, denotes a 
basis of $\Ker(\overline{\lambda}_j-A^*)$ and we denote by $\mathcal{F}$ the subspace of $G$ generated by real and imaginary parts of $\varepsilon_k(\overline{\lambda}_j)$. 
Note that when $\overline{\lambda}_j=\lambda_j$ is real we can suppose $\varepsilon_k(\overline{\lambda}_j)\in G$. Moreover, since $A^*$ is real its spectrum 
is symetric with respect to the real line and non real eigenvalues 
are pairwise conjugate with pairwise conjugate basis of eigenvectors, i.e for all $j\in \mathbb{N}$ the complex value $\lambda_j$ is also an eigenvalues 
of $A^*$ with corresponding basis of eigenvector $\{\overline{\varepsilon_k(\overline{\lambda}_j)}\mid k=1,\dots,\ell_j\}$. Then with
\begin{equation}\label{DefJi}
 \begin{split}
\mathcal{J}_0 &\ov \{j\in \mathbb{N} \mid \Im \lambda_j=0 \},\quad \mathcal{J}_+ \ov \{j\in \mathbb{N} \mid \Im \lambda_j > 0\}\quad \mbox{ and }\quad 
\mathcal{J}_- \ov \{j\in \mathbb{N} \mid \Im \lambda_j < 0\}
 \end{split}
\end{equation}
the subspace $\mathcal{F}$ is also defined by:
\begin{equation}\label{DefChi}
\mathcal{F} \ov \overline{\Span_{\mathbb{R}}\{\chi_k(\overline{\lambda}_j) \mid j\in \mathbb{N},\;k=1,\dots,\ell_j \}}^{_G}\quad\mbox{ where}\quad \chi_{k}(\overline{\lambda}_j)\ov 
\left\{\begin{array}{ll}
\varepsilon_k(\overline{\lambda}_j) &\mbox{ if }\,j\in \mathcal{J}_0,\\
\Re \varepsilon_k(\overline{\lambda}_j)&\mbox{ if }\, j\in \mathcal{J}_+,\\
\Im \varepsilon_k(\overline{\lambda}_j)&\mbox{ if }\, j\in \mathcal{J}_-.
\end{array}
\right.
\end{equation}
The following real version of Theorem \ref{MainTheorem} holds.
\begin{Theorem}\label{MainTheoremReal}
Suppose that $A$ and $B$ are real operators defined on real Hilbert spaces, $G$ and $W$ respectively. Assume that the complexified of 
$A$ and $B$ satisfy $(\mathcal{H}_1)$, $(\mathcal{H}_2)$, $(\mathcal{H}_3)$, $(\mathcal{H}_4)$ and $(\mathcal{H}_5)$ and let $K\in \mathbb{N}^*$. 
Then the following results hold.
\begin{enumerate} 
\item Real system \eqref{eq1.1} is approximately controllable if 
and only if \eqref{UC} is satisfied. 
\item Real system \eqref{eq1.1} is approximately
controllable by a $K$-dimensional control if and only if \eqref{UC} and \eqref{CondK} are satisfied.
\item A family $v\in \ell^2(\mathbb{N},W)$ is admissible if and only if \eqref{RankCond0} is satisfied.
\item Assume that \eqref{UC} is true. Then the set of admissible families of $\ell^2(\mathbb{N},W)$ forms a residual set of $\ell^2(\mathbb{N},W)$. Moreover, if $v\in \ell^2(\mathbb{N},W)$ is admissible then its orthogonal projection onto $\ell^2(\mathbb{N},B^*\mathcal{F})$
 is admissible.
\item Assume that \eqref{UC} and \eqref{CondK} are true. Then the set of admissible families of $W^K$ forms a residual set of $W^K$. Moreover, if $v\in W^K$ is admissible then its orthogonal projection onto $(B^*\mathcal{F})^K$ is admissible.
\end{enumerate}
\end{Theorem}
\begin{proof}
 
The above statement is a consequence of Theorem \ref{MainTheorem}. Let us only precise point 4. and in particular that, under assumption \eqref{UC}, the set of families $v\in \ell^2(\mathbb{N},W)$ such that \eqref{RankCond0} holds
form a residual set of $\ell^2(\mathbb{N},W)$. The proof of the fact that the set of admissible family $v\in \ell^2(\mathbb{N},W)$ is the intersection of open subsets of $\ell^2(\mathbb{N},W)$ is similar to the proof of 
Theorem \ref{MainTheorem}, and we skip it.
For the density the arguments in the proof of Theorem \ref{MainTheorem} must be adapted:
we decompose the first $\ell_j$ vectors of $v$ as
$$
v_k=\widehat v_k+\widetilde v_k, \quad \text{with}\ \widehat v_k\in \Span_\mathbb{R}\{B^* \chi_i(\bar \lambda_j), \ i=1,\ldots,\ell_j\}, \ 
\widetilde v_k\in \Span_\mathbb{R}\{B^* \chi_i(\bar \lambda_j), \ i=1,\ldots,\ell_j\}^\perp.
$$
Then we notice that
$$
\Span_\mathbb{C}\{B^* \chi_i(\bar \lambda_j), \ i=1,\ldots,\ell_j\}=\Span_\mathbb{C}\{B^* \varepsilon_i(\bar \lambda_j), \ i=1,\ldots,\ell_j\}.
$$
In particular, we see that $v$ is admissible if $(\widehat v_1,\ldots,\widehat v_{\ell_j})$ is a basis of $\Span_\mathbb{C}\{B^* \varepsilon_i(\bar \lambda_j), \ i=1,\ldots,\ell_j\}$.
This is equivalent to say that $(\widehat v_1,\ldots,\widehat v_{\ell_j})$ is linearly independent or to say that
the matrix of $(\widehat v_1,\ldots,\widehat v_{\ell_j})$ in the basis 
$(B^* \chi_1(\bar \lambda_j),\ldots,B^* \chi_{\ell_j}(\bar \lambda_j))$ is invertible. However such a matrix can always be approximated by invertible real matrices and thus any $v$ can be approximated by $v^{\eta}\in \ell^2(\mathbb{N},W)$ such that $(\widehat v_1^\eta,\ldots,\widehat v_{\ell_j}^\eta)$ 
is a basis of $\Span_\mathbb{C}\{B^* \varepsilon_i(\bar \lambda_j), \ i=1,\ldots,\ell_j\}$. Since such a $v^\eta$ is admissible this concludes the proof.
\end{proof}
\subsection{Approximate controllability of some coupled heat equations}\label{subsect2.4}
Here we give a simple example of non diagonalizable system which consists in two coupled heat equations:
\begin{equation}
\left\{
\begin{array}{ll}
\displaystyle y_t-\Delta y+z=0&\displaystyle \mbox{ in }(0,T)\times \Omega,\\
\displaystyle z_t-\Delta z={\bf 1}_{\omega} h&\displaystyle \mbox{ in }(0,T)\times \Omega,\\
\displaystyle y=z=0&\displaystyle \mbox{ on }(0,T)\times \partial\Omega,\\
\displaystyle y(0)=y_0\; \mbox{ and }\; z(0)=z_0&\displaystyle \mbox{ in } \Omega.\\
\end{array}
\right.\label{EXeq1}
\end{equation}
Above, $\Omega$ is an open subset of $\mathbb{R}^d$, $d\geq 1$,  $(y_0,z_0)\in L^2(\Omega)\times L^2(\Omega)$, $\omega$ is a non empty open subset of $\Omega$, 
${\bf 1}_{\omega}$ is the characteristic function of $\omega$ and $h\in L^2((0,T)\times \Omega)$ is the control function. 

The above system verify the hypotheses of the previous section and in particular, we have the following result
\begin{Proposition}\label{PHeat1}
System \eqref{EXeq1} is approximately controllable. It is approximately
controllable by a $K$-dimensional control if and only if 
\begin{equation}
\sup_{j\in \mathbb{N}} \widehat \ell_j\leq K,\label{CondKHeat}
\end{equation}
where $\widehat \ell_j$ are the geometric multiplicities of the operator $-\Delta  : H^2(\Omega)\cap H_0^1(\Omega) \to L^2(\Omega)$.
\end{Proposition}
On can also deduce the other consequences of Theorem \ref{MainTheorem} for system \eqref{EXeq1}.
Let us give here the main ideas of the proof of Proposition \ref{PHeat1}.
First, system \eqref{EXeq1} can be rewritten in the form \eqref{eq1.1} with 
$$
A=\left (
\begin{array}{cc}
\Delta & -I\\
0 &\Delta 
\end{array}
\right)
\quad \mbox{ and }\quad 
B=\left(
\begin{array}{ll}
0\\
{\bf 1}_{\omega}
\end{array}
\right).
$$
It is clear that $A$ with domain ${\cal D}(A)=(H^2(\Omega)\cap H_0^1(\Omega))\times (H^2(\Omega)\cap H_0^1(\Omega))$ satisfies 
$(\mathcal{H}_1)$, $(\mathcal{H}_2)$, $(\mathcal{H}_3)$ and $(\mathcal{H}_4)$ with $H=L^2(\Omega)\times L^2(\Omega)$ and that $B$ satisfies $(\mathcal{H}_5)$ with $U=L^2(\Omega)\times L^2(\Omega)$ and 
$\gamma=0$ (i.e. $B$ is bounded). Moreover, it is easily seen that 
$$ 
A^*=\left (
\begin{array}{cc}
\Delta & 0\\
-I &\Delta 
\end{array}
\right)
\quad \mbox{ and }\quad 
\mathcal{D}(A^*)=\mathcal{D}(A).
$$
Let denote by $\lambda_j$, $j\in \mathbb{N}$, the Dirichlet Laplacian eigenvalues with geometric multiplicities $\ell_j$ 
and related basis of eigenvectors $\{\xi_{k}(\lambda_j)\mid k=1,\dots\ell_j\}$. Simple computations show that the eigenvalues of $A^*$ are exactly the $\lambda_j$'s and that proper 
eigenspaces and generalized eigenspaces 
are given by 
\begin{eqnarray*}
\ker(\lambda_j-A^*)&=&\Span\left\{\varepsilon_k(\lambda_j)\mid \; k=1,\dots,\ell_j \right\}\\
\ker(\lambda_j-A^*)^2&=&\Span\left\{\varepsilon_k(\lambda_j), f_k(\lambda_j)\mid \; k=1,\dots,\ell_j \right\},
\end{eqnarray*}
where
$$
 \varepsilon_{k}(\lambda_j)\ov \left(
\begin{array}{c}
\displaystyle  0\\
\displaystyle  \xi_{k}(\lambda_j)
\end{array}
\right),
\quad 
f_{k}(\lambda_j)\ov \left(
\begin{array}{c}
\displaystyle  \xi_{k}(\lambda_j)\\
\displaystyle  0
\end{array}
\right).
$$
Thus we are in the situation where $m(\lambda_j)=2$ and where algebraic multiplicity of $\lambda_j$ is $N_j=2\ell_j$ while the geometric multiplicity of $\lambda_j$ is $\ell_j$.

Finally, \eqref{UC} reduces to:
\begin{equation}
\left\{
\begin{array}{ll}
\displaystyle \Delta y=\lambda y&\displaystyle \mbox{ in }\Omega,\\
\displaystyle \Delta z-y=\lambda z&\displaystyle \mbox{ in }\Omega,\\
\displaystyle y=z=0&\displaystyle \mbox{ on }\partial \Omega,\\
\end{array}
\right.\quad \mbox{ and }\quad z\equiv 0 \mbox{ in }\omega \quad \Longrightarrow \quad \left\{
\begin{array}{ll}
\displaystyle y\equiv 0&\displaystyle \mbox{ in }\Omega,\\
\displaystyle z\equiv 0&\displaystyle \mbox{ in }\Omega.
\end{array}
\right. \label{EXeq2}
\end{equation}
To prove \eqref{EXeq2}, suppose that the left part of \eqref{EXeq2} is satisfied with $\lambda=\lambda_j$ for some $j$ (otherwise the conclusion is obvious). 
Then $\,^t(y,z)\in \ker(\lambda_j-A^*)$ which implies $y=0$ and the conclusion follows from Holmgren's uniqueness Theorem, 
see for instance \cite[Thm.8.6.5, p.309]{HormanderI}. In conclusion, Theorem \ref{MainTheorem} applies and \eqref{EXeq1} is approximately controllable in any time.

\section{Stabilizability of parabolic systems}\label{sectstab0}

\subsection{Proof of Theorem \ref{CorStab}}\label{subsect3.1}
For the following we suppose that assumptions $(\mathcal{H}_1)$, $(\mathcal{H}_3)$, $(\mathcal{H}_5)$ and  $(\mathcal{H}_6)$
are satisfied and we first  rewrite  \eqref{eq1.1} in two equations, one related to the ``unstable'' modes and the other to the ``stable'' 
modes. Let $N\in {\mathbb N}^*$ be such that: 
\begin{equation}
\Re\lambda_0\geq\Re\lambda_1 \geq \dots\geq \Re\lambda_N\geq  -\sigma>\Re\lambda_{N+1}\geq\dots,\label{LNsigmaLN1}
\end{equation}
and set $\Sigma_N\ov \{\lambda_k\mid k=1,\dots,N\}$. Such an integer $N$ exists because of conditions $(\mathcal{H}_1)$, $(\mathcal{H}_3)$ and $(\mathcal{H}_6)$. Let us note that condition $(\mathcal{H}_3)$ yields that
the spectrum of $A$ is located in a sector of a left half-plane with an opening angle strictly lower than $\pi$.  
 Thus, we split \eqref{eq1.1} in two equations, one related to the ``unstable'' modes $\Sigma_N$ and the other to the ``stable'' 
modes $\Sigma\backslash \Sigma_N$ (see \cite[Par. III.4, Thm. 6.17, p.178]{Kato1966}). For that, we introduce the projection operator defined by 
\begin{equation}
P_N\ov \frac{1}{2\pi \imath}\int_{\Gamma_N}(\lambda -A)^{-1}{\rm d\lambda},
\end{equation}
where $\Gamma_N$  be a contour enclosing $\Sigma_N$ but no other point of the spectrum of $A$. Then the space $H$ is the direct sum of the two invariant subspaces $H_N=P_N H$ and $H_N^-=(I-P_N)H$ of $A$. We also introduce the adjoint of $P_N$ which has the following expression: 
\begin{equation}
P_N^*= \frac{1}{2\pi \imath}\int_{\Gamma_N}(\lambda -A^*)^{-1}{\rm d\lambda}.
\end{equation}
Then, the solution $y$ of \eqref{eq1.1}, which can be rewritten $y=y_N+y_N^-$ with $y_N=P_N y$ and $y_N^-=(I-P_N)y$, is solution to systems:
\begin{eqnarray}
y_N'&=A_N y_N+B_N p_N u\in H_{N},\label{equnstab}\\
y_N^{-'}&=A_N^- y_N^-+B_N^- p_N^- u \in  H_{N}^{-}.\label{eqstab}
\end{eqnarray}
In the above setting, $A_N$ and $A_N^-$ denote the restriction of $A$ to $H_N$ and $H_N^-$ respectively, and $B_N=P_N Bp_N^*$,  $B_N^{-}=(I-P_N) Bp_N^{*-}$ where $p_N:U\rightarrow U_{N}$ and $p_N^-:U\rightarrow U_{N}^-$ are the orthogonal projection operators on
$$
U_{N}\ov \{B^*\varepsilon\mid \varepsilon\in H_{N}^*\ov P_N^*H \}\quad \mbox{ and }\quad U_{N}^-\ov \{B^*\varepsilon\mid \varepsilon\in H_{N}^{*-}\ov (I-P_N^*)H \},
$$
respectively.  Note that their respective adjoints $p_N^*:U_{N}\rightarrow U$ and $p_N^{-*}:U_{N}^-\rightarrow U$ are the inclusion maps. For a detailed justification of the decomposition \eqref{equnstab}-\eqref{eqstab}, see \cite{BT,RaymondThevenet2009}. 

Let us prove that $(A+\sigma,B)$ is stabilizable, if and only if, \eqref{UCstab} is satisfied.

The ``only if'' part is obtained by remarking that if for $j\in \{1,\dots,N\}$ there exists $\varepsilon(\overline{\lambda}_j)\in \ker(A^*-\overline{\lambda}_j)$ 
obeying $\varepsilon(\overline{\lambda}_j)\neq 0$ and $B^*\varepsilon(\overline{\lambda}_j)=0$, then multiplying \eqref{eq1.1} by $\varepsilon(\overline{\lambda}_j)$ yields 
that every solution $y$ of \eqref{eq1.1} satisfies: 
$$(y(t)|\varepsilon(\overline{\lambda}_j))_H=e^{\lambda_j t}(y(0)|\varepsilon(\overline{\lambda}_j))_H.$$
Since the above equality is independent on $u$ and since $\Re \lambda_j\geq -\sigma$, then \eqref{EstStab} is false for any initial datum not orthogonal to $\varepsilon(\overline{\lambda}_j)$.

Suppose now that \eqref{UCstab} is true and let us prove that $(A+\sigma,B)$ is stabilizable. For that, let us first verify that
\eqref{equnstab} is null controllable. Assumptions $(\mathcal{H}_1)$, $(\mathcal{H}_3)$ and $(\mathcal{H}_4)$ are obviously satisfied 
for $A_N$ (because $A_N$ is finite dimensional) and the fact that $A_N$ obeys $(\mathcal{H}_2)$ is a direct consequence of the definition of $H_N$ and of $A_N$ (the restriction of $A$ to $H_N$). 
Moreover, since the spectrum of $A_N$ is exactly $\Sigma_N$, then \eqref{UC} for $A_N$ is exactly \eqref{UCstab}. Then by Theorem 
\ref{MainTheorem} system \eqref{equnstab} is approximately controllable, and since null controllability and approximate 
controllability are equivalent notions for finite dimensional systems we deduce that \eqref{equnstab} is null controllable.

Then it follows that $(A_N+\sigma,B_N)$ is 
stabilizable: there exists $F_N\in \mathcal{L}(H_{N},U_N)$ such that the solution $\widehat y_N$ to \eqref{equnstab} with $u=F_N \widehat y_N$ satisfies for some $\epsilon>0$:
\begin{equation}
\|\widehat y_N(t)\|_{H}\leq C e^{-(\sigma+\epsilon) t}\|P_N y(0)\|_H\quad (t\geq 0).\label{expDecYN}
\end{equation}
For instance, the finite dimensional feedback law $F_N$ can be constructed with a Riccati operator obtained from a quadratic minimizing problem as in \cite{BT}.
Moreover the solution $\widehat y_N^-$ of \eqref{eqstab} with $u=F_N \widehat y_N$ is given by:
\begin{equation}\nonumber
 \begin{split}
\widehat y_N^-(t)&=e^{A_N^- t}(I-P_N)y(0)+\int_0^t e^{A_N^-(t-s)}B_N^- p_N^- F_N \widehat y_N{\rm d}s   ,\\
&=e^{A_N^- t}(I-P_N)y(0)+\int_0^t (\lambda_0-A)^{\gamma}e^{A_N^-(t-s)}(I-P_N)(\lambda_0-A)^{-\gamma}B F_N \widehat y_N{\rm d}s   .
 \end{split}
\end{equation}
The last above equality follows by remarking that $(\lambda_0-A_N^-)^{-\gamma}=(\lambda_0-A)^{-\gamma}(I-P_N)$. Then with $(\mathcal{H}_5)$, with the fact that \eqref{LNsigmaLN1} and $(\mathcal{H}_3)$ guarantee that for $\epsilon'>0$ such that $-(\sigma+\epsilon')>\Re \lambda_{N+1}$ we have 
$\|e^{A_N^- t}\|_{H_N^-}\leq C e^{-(\sigma+\epsilon') t}$ and $\|(\lambda_0-A_N^-)^{\gamma} e^{A_N^- (t-s)}\|_{H_N^-}\leq C (t-s)^{-\gamma} e^{-(\sigma+\epsilon') (t-s)}$, 
and with \eqref{expDecYN} we deduce:
\begin{equation}\label{eqyNmoinsstab}
\|\widehat y_N^-(t)\|_H\leq C\left(e^{-(\sigma+\epsilon') t}+e^{-(\sigma+\epsilon) t}\int_0^t \frac{e^{(\epsilon-\epsilon')(t-\tau) }}{(t-\tau)^\gamma}d\tau \right)\|y(0)\|_H\quad (t\geq 0).
\end{equation} 
Then if we choose $\epsilon<\epsilon'$ the above inequality means that the feedback control constructed from the unstable part of system \eqref{eq1.1} does not destabilize the stable part of 
system \eqref{eq1.1}. Then we have proved that the solution of \eqref{eq1.1} with $u=F_N P_N y$ obeys \eqref{EstStab} which is to say that $(A+\sigma,B)$ is stabilizable.

Finally, points 2, 3 and 4 are a direct consequence of Theorem \ref{MainTheorem} applied to the projected system \eqref{equnstab}.


\begin{Remark}
Note that $(\mathcal{H}_2)$ is not required in Theorem \ref{CorStab}. This comes from the fact that Theorem \ref{CorStab} is obtained by applying Theorem \ref{MainTheorem}
to the projection of system \eqref{eq1.1} onto the finite dimensional subspace generated by unstable root vectors of $A$ (i.e. the unstable subspace). Then by definition the family of unstable 
root vectors of $A$ is complete in the unstable subspace and $(\mathcal{H}_2)$ is always satisfied for the projected system.
\end{Remark}
\begin{Remark}
 From the applications point of view rank conditions \eqref{CondWk} are of great interest since they are practical criterions to construct
admissible families of stabilizing actuators. For instance, a choice for $v$ could be 
$$v=(B^*\varepsilon_k(\bar\lambda_j))_{\Re\lambda_j\geq -\sigma, k=1,\dots,\ell_j}.$$
Indeed, for such a $v$ each matrix $W_j$ contains the full rank block $(B^*\varepsilon_k(\bar\lambda_j), B^* \varepsilon_l(\bar\lambda_j))_{0\leq k,l\leq \ell_j}$.
\end{Remark}

\subsection{Stabilizability of nonlinear parabolic systems}\label{subsect3.2}
Here, by following the path sketched in \cite{BT} we recall how Theorem \ref{CorStab} can be used to prove the stabilizability of nonlinear systems:  
\begin{equation}
y'=A y+B u+\mathcal{N}(y,u),\label{eq1.1nl}
\end{equation}
where $\mathcal{N}(\cdot,\cdot)$ is a nonlinear mapping satisfying adequate Lipschitz properties recalled below. For the following we denote by
$F$ the bounded stabilizing feedback operator given by Theorem \ref{CorStab} and we define the closed loop operator $A_F\ov A+B F$ with domain $\mathcal{D}(A_F)\ov \{y\in H\mid A y+B F y\in H\}$. It is well known that the semigroup generated by $A_F$ is analytic on $H$ (see \cite[Prop. 10]{BT}). 
Moreover, we define:
\begin{equation} \label{HFalpha}
H_{F, \alpha}\ov \mathcal{D}((-A_F)^{\alpha})\quad \alpha\geq 0.
\end{equation} 
Here we assume hypothesis $(\mathcal{H}_4)$, from which we deduce $H_{F, \alpha}=[H,\mathcal{D}(A_F)]_\alpha$ for $\alpha\in [0,1]$ (see \cite{BT} for details).

Next, we suppose that $\mathcal{N}(\cdot,\cdot)$ obeys for $s\in [0,1]$:
\begin{equation}\label{HypNL}
 \begin{split}
  \|\mathcal{N}(\xi, F\xi)\|_{H_{\frac{s-1}{2}}}&\leq C \|\xi \|_{H_{F,\frac{s}{2}}}\|\xi \|_{H_{F,\frac{s+1}{2}}}\\
\|\mathcal{N}(\xi, F\xi)-\mathcal{N}(\zeta, F\zeta)\|_{H_{\frac{s-1}{2}}}&\leq C \big{(}\|\xi-\zeta \|_{H_{F,\frac{s}{2}}}(\|\xi \|_{H_{F, \frac{s+1}{2}}}+
\|\zeta \|_{H_{F, \frac{s+1}{2}}})\\
&\quad + \|\xi-\zeta \|_{H_{F, \frac{s+1}{2}}}
(\|\xi \|_{H_{F, \frac{s}{2}}}+\|\zeta \|_{H_{F, \frac{s}{2}}}\big{)}.
 \end{split}
\end{equation}

Finally, to state the stabilization theorems for \eqref{eq1.1nl} we need to introduce some spaces of Hilbert valued functions of $t\geq 0$. For two Hilbert spaces $\mathcal{X}$, $\mathcal{Y}$ we denote by 
$L^2(0,T;\mathcal{X})$, $L^\infty(0,T;\mathcal{X})$, $H^1(0,T;\mathcal{Y})$ usual Lebesgue and Sobolev spaces, we set
$W(0,T;\mathcal{X}, \mathcal{Y})\ov L^2(0,T;\mathcal{X})\cap H^1(0,T;\mathcal{Y})$ and $W(\mathcal{X}, \mathcal{Y})\ov W(0,+\infty;\mathcal{X}, \mathcal{Y})$, 
we denote by $L_{\rm loc}^2(\mathcal{X})$, $L_{\rm loc}^\infty(\mathcal{X})$ the spaces of functions belonging for all $T>0$ to $L^2(0,T;\mathcal{X})$, 
$L^\infty(0,T;\mathcal{X})$ respectively, and for $\sigma>0$ we denote 
by $W_{\sigma}(\mathcal{X}, \mathcal{Y})$ the space of functions $y$ such that $e^{\sigma(\cdot)}y\in W(\mathcal{X}, \mathcal{Y})$. Finally, for $s\in [0,1]$ we use the shorter expression:
$$
W_\sigma^s\ov W_{\sigma}(H_{F, \frac{s+1}{2}}, H_{\frac{s-1}{2}}).
$$
Then the following theorem can be obtained analogously to \cite[Thm. 15]{BT}.
\begin{Theorem}\label{ThmNL}
Assume $(\mathcal{H}_1)$, $(\mathcal{H}_3)$, $(\mathcal{H}_4)$, $(\mathcal{H}_5)$ and $(\mathcal{H}_6)$, for $\sigma>0$ assume \eqref{UCstab} and let $F$ be the finite rank feedback law given
by Theorem \ref{CorStab}. For $s\in [0,1]$ assume also that $\mathcal{N}$ satisfies \eqref{HypNL} and 
$y_0\in H_{F, \frac{s}{2}}$. There exist $\rho>0$ and $\mu>0$ such that if $\|y_0\|_{H_{F, \frac{s}{2}}}<\mu$ 
then system \eqref{eq1.1nl} with $y(0)=y_0$ admits a solution $y\in W_\sigma^s$ such that $\|y\|_{W_\sigma^s}\leq \rho \|y_0\|_{H_{F, \frac{s}{2}}}$, which is unique within the 
class of functions in $$L_{\rm loc}^\infty(H_{F, \frac{s}{2}})\cap L_{\rm loc}^2(H_{F, \frac{s+1}{2}}).$$ Moreover, there exists $C>0$ such that for all $t\geq 0$
\begin{equation}
\|y(t)\|_{H_{F, \frac{s}{2}}}\leq C  e^{-\sigma t} \|y_0\|_{H_{F, \frac{s}{2}}}.\label{expdecr}
\end{equation}
\end{Theorem}

The main difficulty to apply Theorem \ref{ThmNL} in concrete examples is that one has to identify $H_{F, \alpha}$ for $\alpha\in [0,1]$. Indeed, 
assumption \eqref{HypNL} is usually obtained from
boundedness property of $\mathcal{N}(\cdot,\cdot)$ in original spaces that are not related to $F$. Moreover, Theorem \ref{ThmNL} is only valid for $y_0\in H_{F, \frac{s}{2}}$  that can
be too restrictive, see \cite{BADRASICON, BadraCOCV}.
 
An alternative to avoid such a difficulty and to obtain an exponential decrease in the norm of $H_{\frac{s}{2}}$ is to use a dynamical control. 
Consider a control function of the form:
\begin{equation}
 u(t)=\sum_{j=1}^K u_j(t)v_j,\quad \overline{u}=(u_1,\dots, u_K)\in \mathbb{C}^K\label{eqbaru}
\end{equation}
where $(v_j)$ is an admissible family for \eqref{eq1.1}.   
 First, we recall the notation \eqref{eq1.2}
$$
V\overline{u}\ov \sum_{j=1}^K u_j B v_j,
$$
so that system \eqref{eq1.1}, \eqref{eqbaru} with $\overline{g}\ov \overline{u}'$
can be rewritten as 
the extended system:
\begin{equation}\label{ExtSyst}
 \begin{split}
 \begin{bmatrix}
        y\\
      \overline{u}
       \end{bmatrix}'=\mathbb{A} \begin{bmatrix}
        y\\
      \overline{u}
       \end{bmatrix}+\mathbb{V} \overline{g},\quad \mbox{ where }\;\mathbb{A}\ov \begin{bmatrix}
        A & V \\
      0 & 0
       \end{bmatrix}, \quad \mathbb{V}\ov \begin{bmatrix}
        0\\
       {\rm Id}
       \end{bmatrix}.
\end{split}
\end{equation}
In the above system, $\overline{g}$ plays the role of the control. To study this system, we introduce for $\theta\in [0,1]$ the following spaces
$$
\mathbb{H}_\theta\ov \{(y,\overline{u}) \in H\times \mathbb{C}^K \mid y+(\mu_0-A)^{-1}V \overline{u}\in H_\theta\}.
$$
We verify that $\mathbb{H}_1$ is the domain of $\mathbb{A}$ and that if $(A,V)$ satisfies 
$(\mathcal{H}_1)$, $(\mathcal{H}_3)$, $(\mathcal{H}_5)$ and $(\mathcal{H}_6)$ then $(\mathbb{A},\mathbb{V})$ satisfies 
$(\mathcal{H}_1)$, $(\mathcal{H}_3)$, $(\mathcal{H}_5)$ and $(\mathcal{H}_6)$ with $\gamma=0$. Moreover, it is easily seen that if  
$(A,V)$  satisfies \eqref{UC} then so does $(\mathbb{A},\mathbb{V})$, which means that if $v$ is admissible for stabilizability of $(A,B)$ then 
\eqref{ExtSyst} is stabilizable, see \cite{BT} for details.
Then from Theorem \ref{CorStab} we have 
the existence of a linear finite rank operator $F:\mathbb{C}^K\times H\to \mathbb{C}^K\times \mathbb{C}^K$ such that the solutions of \eqref{eq1.1}, \eqref{eqbaru}
with $\overline{u}$ satisfying 
\begin{equation}
\overline{u}'= F(\overline{u},y),\label{couplagebaru0}
\end{equation}
obey
$$
\|y(t)\|_H+|\overline{u}(t)|\leq C(\|y(0)\|_H+|\overline{u}(0)|)e^{-\sigma t},
$$
where $|\cdot|$ denotes the euclidian norm of $\mathbb{C}^K$. Note that \eqref{couplagebaru0} can be written as 
\begin{equation}\label{couplagebaru}
\overline{u}'+\Lambda \overline{u}=\sum_{j=1}^K e_j(\widehat \varepsilon _j| y(t))_H,
\end{equation}
where $\Lambda$ is a complex matrix of size $K\times K$ representing the first component of $F$, where $\{\widehat \varepsilon _j\in H_N\,;\,j=1,\dots, K\}$ is a family of kernels functions which represent the second component of $F$ and 
where $\{e_j\,;\,j=1,\dots, K\}$ denotes a basis of $\mathbb{C}^K$.

Finally, we suppose that for $s\in [0,1]$:
\begin{equation}\label{HypNL2}
 \begin{split}
  \|\mathcal{N}(\xi, u)\|_{H_{\frac{s-1}{2}}}&\leq C \|(\xi,\overline{u})\|_{\mathbb{H}_{\frac{s}{2}}}\|(\xi,\overline{u})\|_{\mathbb{H}_{\frac{s+1}{2}}}\\
\|\mathcal{N}(\xi, u)-\mathcal{N}(\zeta, w)\|_{H_{\frac{s-1}{2}}}&\leq C 
\big{(}\|(\xi,\overline{u})-(\zeta,\overline{w})\|_{\mathbb{H}_{\frac{s}{2}}}(\|(\xi,\overline{u})\|_{\mathbb{H}_{\frac{s+1}{2}}}+\|(\zeta,\overline{w})\|_{\mathbb{H}_{\frac{s+1}{2}}})\\
&\quad\quad\quad \|(\xi,\overline{u})-(\zeta,\overline{w})\|_{\mathbb{H}_{\frac{s+1}{2}}}(\|(\xi,\overline{u})\|_{\mathbb{H}_{\frac{s}{2}}}
+\|(\zeta,\overline{w})\|_{\mathbb{H}_{\frac{s}{2}}})\big{)},\\
 \end{split}
\end{equation}
where we have used the notations:
$$
u=\sum_{j=1}^K u_jv_j,\quad w(t)=\sum_{j=1}^K w_jv_j,\quad \overline{u}=(u_1,\dots, u_K), \quad \overline{w}=(w_1,\dots, w_K).
$$
For $s\in [0,1]$ we define
$$
\mathbb{W}_\sigma^s=\left \{(y,\overline{u})\mid (e^{\sigma (\cdot)}y, e^{\sigma (\cdot)}\overline{u})\in 
L^2(\mathbb{H}_{\frac{s+1}{2}})\cap  H^1(H_{\frac{s-1}{2}} \times \mathbb{C}^K)\right \}.
$$
Then the following theorem can be obtained analogously to \cite[Thm. 18]{BT}.
\begin{Theorem}\label{ThmNL2}
  Assume that $(A,V)$ satisfies $(\mathcal{H}_1)$, $(\mathcal{H}_3)$, $(\mathcal{H}_4)$, $(\mathcal{H}_5)$, $(\mathcal{H}_6)$ and \eqref{UCstab}  for $\sigma>0$.
Let $K\in \mathbb{N}^*$ be given by Theorem \ref{CorStab}. For $s\in [0,1]$ assume also \eqref{HypNL2} and 
$y_0\in H_{\frac{s}{2}}$. There exist $\rho>0$ and $\mu>0$ such that if $\|y_0\|_{H_{\frac{s}{2}}}<\mu$ 
then system \eqref{eq1.1nl}-\eqref{eqbaru}-\eqref{couplagebaru} with $y(0)=y_0$ and $\overline{u}(0)=0$ admits a solution $(y,\overline{u})\in \mathbb{W}_\sigma^s$ 
satisfying
$\|(y,\overline{u})\|_{\mathbb{W}_\sigma^s}\leq \rho \|y_0\|_{H_{\frac{s}{2}}}$, which is unique within the class of functions in 
$L_{\rm loc}^\infty(\mathbb{H}_{\frac{s}{2}})\cap L_{\rm loc}^2(\mathbb{H}_{\frac{s+1}{2}})$. Moreover, 
there exists $C>0$ such that for 
all $t\geq 0$
$$
\|(y(t),\overline{u}(t))\|_{\mathbb{H}_\frac{s}{2}}\leq Ce^{-\sigma t} \|y_0\|_{H_\frac{s}{2}}.
$$
\end{Theorem}

\subsection{Stabilization of real systems}\label{SectStabreal}\label{subsect3.3}
We also have a real version of Theorem 
\ref{CorStab} stated in Theorem \ref{CorStabReal} below. If $A$ and $B$ are real operators defined on real Hilbert spaces $G$ and $W$ respectively 
the pair $(A,B)$ is referred as a \textit{real pair} and the related 
definitions of stabilizability introduced in Section \ref{sec-intro} are the same as above by replacing the complex spaces $H$ and $U$ by the real spaces $G$ and $W$. 
In such case, a (real) $K$-dimensional feedback law is of the form
\begin{equation}\label{FeedbackSumReal}
F y(t)= \sum_{j=1}^K (y(t), \widehat \chi_j)_G v_j,
\end{equation}
where $\widehat \chi_j\in G$, $j=1,\dots, K$.

To state a real version of Theorem \ref{CorStab}, we need to introduce the following subspace of $W$: 
$$
\mathcal{F}_\sigma \ov \Span_{\mathbb{R}} \{\chi_k(\bar\lambda_j) \mid \Re\lambda_j\geq -\sigma,\;k=1,\dots,\ell_j\},
$$
where family $(\chi_k(\bar\lambda_j))$ is defined by \eqref{DefChi}. The following corollary of Theorem \ref{MainTheoremReal} holds.
\begin{Theorem}\label{CorStabReal}
Suppose that $A$ and $B$ are real operators defined from respective real Hilbert spaces $G$ and $W$, 
assume that the complexified of $A$ and $B$ satisfy $(\mathcal{H}_1)$, $(\mathcal{H}_3)$, $(\mathcal{H}_5)$ and $(\mathcal{H}_6)$ and let $K\in \mathbb{N}^*$ and $\sigma>0$. 
 Then the following results hold.
\begin{enumerate}
\item The real pair $(A+\sigma, B)$ is stabilizable if and only if \eqref{UCstab} holds.
\item The real pair $(A+\sigma, B)$ is stabilizable by a $K$-dimensional control if and only if \eqref{UCstab} and \eqref{CondKsigma} are satisfied.
\item A family $v=(v_j)_{j=1,\dots,K}$ of $W^K$ is admissible for stabilizability  of $(A+\sigma, B)$ if and only if \eqref{CondWk} is satisfied.
\item Assume that \eqref{UCstab} and \eqref{CondKsigma} are true.  Then the set of admissible families for stabilizability  of $(A+\sigma, B)$ forms a residual set of $W^K$. 
Moreover, if $v$ is admissible for stabilizability  of $(A+\sigma, B)$ then its orthogonal projection onto $(B^*\mathcal{F}_\sigma)^K$ is admissible for stabilizability
 of $(A+\sigma, B)$.
\end{enumerate}
 \end{Theorem}
\begin{Remark}\label{RemReal}
If $F$ is the bounded stabilizing feedback operator \eqref{FeedbackSumReal} given by Theorem \ref{CorStabReal} then in a way similar as for \eqref{HFalpha} 
we can define $G_{F, \alpha}$. Then the  analogue of Theorem \ref{ThmNL} and Theorem \ref{ThmNL2} hold for 
the real closed-loop nonlinear system   by replacing the complex spaces $\mathbb{C}^K$, $H_{F, \alpha}$, $H_{F, \frac{s}{2}}$, $H_{\frac{s}{2}}$, etc. by the 
real spaces $\mathbb{R}^K$, 
$G_{F, \alpha}$, $G_{F, \frac{s}{2}}$, $G_{\frac{s}{2}}$, etc.
\end{Remark}
\begin{Remark}
Note that in \cite{RaymondThevenet2009} the authors choose the whole family of real and imaginary parts of generalized eigenvectors as a stabilizing family. 
However, according to rank criterion \eqref{CondWk} it is sufficient to choose the family of real and imaginary parts of proper eigenvectors. The family $v$ given below is admissible 
for stabilizability of $(A+\sigma,B)$:
$$v=(B^* \chi_k(\bar\lambda_j))_{\Re\lambda_j\geq -\sigma, k=1,\dots,\ell_j}.$$
Indeed, if $j\in J_0$ the above matrix contains the full rank block $(B^*\varepsilon_k(\bar\lambda_j), B^* \varepsilon_l(\bar\lambda_j))_{0\leq k,l\leq \ell_j}$,
and if $j\notin J_0$ the matrix $W_j$ contains the $\ell_j$-rank block $[R_j,I_j]$ 
where 
$$R_j= ( B^*\Re \varepsilon_k(\bar\lambda_j), B^* \varepsilon_l(\bar\lambda_j))_{0\leq k,l\leq \ell_j}\quad \mbox{ and }\quad 
I_j= ( B^*\Im \varepsilon_k(\bar\lambda_j), B^* \varepsilon_l(\bar\lambda_j))_{0\leq k,l\leq \ell_j}.$$
This last claim comes from the fact that eigenvectors associated to eigenvalues $\{\overline{\lambda}_j\mid \Re\lambda_j\geq -\sigma\}$ are pairwise conjugate.
\end{Remark}
\subsection{Minimal number of actuators for the heat equation in a rectangular domain}\label{subsect3.4}

Consider the $d$--dimensional controlled heat equation in a rectangle $\Omega=\prod_{i=1}^d (0,c_i)$ for $d\geq 2$:
\begin{equation}
\left\{
\begin{array}{ll}
\displaystyle y_t=\Delta y&\displaystyle \mbox{ in }(0,T)\times \Omega,\\
\displaystyle y={\bf 1}_{\Gamma} h&\displaystyle \mbox{ on }(0,T)\times \partial\Omega,\\
\displaystyle y(0)=y_0\;&\displaystyle \mbox{ in } \Omega.\\
\end{array}
\right.\label{EXeq3}
\end{equation}
In the above setting, ${\bf 1}_{\Gamma}$ is the characteristic function of a non empty open subset $\Gamma\subset \{0\}\times \prod_{i=2}^d (0,c_i)$ 
and $h\in L^2((0,T)\times \partial\Omega)$ is the control function. 

The eigenvalues of the Dirichlet Laplacian are given by
$$
\lambda_\alpha=-\pi^2\sum_{i=1}^d \left(\frac{\alpha_i}{c_i}\right)^2,\quad \alpha=(\alpha_1,\dots,\alpha_d)\in\mathbb{N}^{*d},
$$
with related eigenvectors
$$
\varphi_{\alpha}(x)=\prod_{i=1}^d\sin\left(\frac{\pi }{c_i}\alpha_i x_i\right),\quad x=(x_1,\dots,x_d)\in \Omega.
$$
It is classical that controlled system \eqref{EXeq3} can be written in the form \eqref{eq1.1} for linear operator $A$, $B$ satisfying   
$(\mathcal{H}_1)$, $(\mathcal{H}_2)$, 
$(\mathcal{H}_3)$, $(\mathcal{H}_4)$ and $(\mathcal{H}_5)$ with $\gamma>\frac{3}{4}$. For this last, we refer to \cite[Chapter 3, Section 3.1]{L-TlBook2000Vol1} as well as references therein. Moreover, \eqref{UC} reduces to
$$
\left\{
\begin{array}{rl}
\lambda \varphi-\Delta\varphi&=0\; \mbox{ in }\; \Omega,\\
  \varphi&=0 \; \mbox{ on }\; \partial\Omega,\\
\frac{\partial \varphi}{\partial x_1}&= 0\; \mbox{ on }\; \Gamma,
\end{array}
\right.\quad\Longrightarrow \quad \varphi=0 \; \mbox{ in }\; \Omega,
$$
which is an easy consequence of Holmgren's uniqueness Theorem (by using an extension of the domain procedure). As a consequence,
\eqref{EXeq3} is approximately controllable and stabilizable by finite dimensional control. A remaining question is how many actuators are required for 
approximate controllability or stabilizability?
\par Suppose that $1/c_i^2$, $i=1,\dots,d$ are $\mathbb{Q}$-linearly independent. Then a straightforward calculation shows that the mapping 
$\alpha\mapsto \lambda_\alpha$ is one-to-one and then that the spectrum is simple (note that it is well-known that the spectrum of the Dirichlet-Laplace operator is generically 
simple with respect to the domain see \cite{Micheletti1972}). It implies that \eqref{EXeq3} is approximately controllable, as well as stabilizable for all rate $\sigma>0$, 
with a one dimensional controller. 
\par Suppose now that $c_i=\pi/c$, $i=1,\dots,d$ for $c>0$. Then the mapping $\alpha\mapsto \lambda_\alpha=-(\pi/c)^2|\alpha|^2$, $|\alpha|^2\ov \sum_{i=1}^d\alpha_i^2$, is no longer one-to-one. 
This means that \eqref{EXeq3} is no longer approximately controllable with a one dimensional controller. In fact, the sequence of geometric multiplicities 
$$m_\alpha^d=\sharp\{\beta\in \mathbb{N}^{*d} \mid \lambda_{\beta}=\lambda_{\alpha}\},\quad \alpha\in \mathbb{N}^{*d},$$
is unbounded. This follows from the fact that $m_\alpha^d=r_d(|\alpha|^2)\geq r_2(\alpha_1^2+\alpha_2^2)$ where $r_d(n)$ denotes the total number of the representation of 
$n$ as a sum of $d$ square of positive integers, and that for instance $r_2(5^{2p})=p+1$ for $p\in \mathbb{N}^*$, see \cite[Thm. 278]{HardyWright}. 
It means that the geometric 
multiplicities of the Dirichlet Laplacian in the square are not bounded and that the approximate controllability with finite dimensional controllers is not possible. 
However, for $\sigma>0$ one can prove that the maximum of the geometric multiplicities $m_\alpha^d$ corresponding to eigenvalues $\lambda_\alpha \geq -\sigma$ is bounded by 
$\frac{(\sqrt{\pi})^{d-1}}{(2c)^{d-1}\Gamma(\frac{d+1}{2})}\sigma^{\frac{d-1}{2}}$ and then 
\eqref{EXeq3} is stabilizable for a rate $\sigma>0$ by means of a $K^\sigma$-dimensional control with 
$$K^\sigma=\left\lfloor \frac{(\sqrt{\pi})^{d-1}}{(2c)^{d-1}\Gamma(\frac{d+1}{2})}\sigma^{\frac{d-1}{2}} \right \rfloor,$$ 
where $\lfloor x \rfloor $ denotes the integer part of $x$.  In order to prove the bound on the maximum of the geometric multiplicities, we note that $r_d(n)$ is also the number of tuples of positive integers which are on the $d$-sphere of ray 
$\sqrt{n}$ we deduce first that $r_d(n)$ is bounded above by the number of tuples of positive integers 
in the $d-1$ ball of ray $\sqrt{n}$ and thus that $r_d(n)$ is bounded above by the volume occupied by the points of the $d-1$ ball of ray $\sqrt{n}$ which have positive coordinates: 
$r_d(n)\leq \frac{\pi^{\frac{d-1}{2}}}{2^{d-1}\Gamma(\frac{d+1}{2})}n^{\frac{d-1}{2}}$. 
Then the conclusion follows from:
$$
\max_{\lambda_\alpha\geq -\sigma} m_{\alpha}^d=\max_{|\alpha|^2 \leq \sigma/c^2} r_d(|\alpha|^2).
$$
Finally, let us underline that the strategy consisting in choosing as many controllers as the number of modes corresponding to eigenvalues greater than 
$-\sigma $ (which is the strategy of \cite{RaymondThevenet2009}) would lead to a number of controllers 
$K_u^\sigma=\lfloor \frac{(\sqrt{\pi})^d}{(2c)^{d}\Gamma(\frac{d}{2}+1)}\sigma^{\frac{d}{2}}\rfloor $, and that 
the ratio $K_u^\sigma/K^\sigma$ behaves as $\frac{(d+1)\sqrt{\pi}\sqrt{\sigma}}{4c}$ as $\sigma\to \infty$.

\section{Stabilizability of incompressible Navier-Stokes type systems}\label{sect5}
  In this section, we illustrate the interest of Theorem \ref{CorStabReal} by applying it to some Navier-Stokes type systems. More precisely, we show how a unique continuation results for stationary Stokes type system yields a result for the stabilization of a nonlinear Navier-Stokes type system. The case of Navier-Stokes system was obtained in \cite{BT}, and here we deal with the feedback stabilizability of a magnetohydrodynamic system and of a micropolar fluid system. The unique continuation theorems (to check criterion \eqref{UC}) are stated in the Appendix.

\paragraph{Notation.}   In what follows, $d=2$ or $d=3$ and for a nonempty open subset $\mathcal{D}$ of $\mathbb{R}^d$ we denote $L^2(\mathcal{D};\mathbb{R})$, $L^2(\mathcal{D};\mathbb{C})$, 
$H^1(\mathcal{D};\mathbb{R})$, $H^1(\mathcal{D};\mathbb{C})$, etc the usual Lebesgue and Sobolev spaces of functions with values in $\mathbb{R}$ or $\mathbb{C}$.
For a scalar function $\pi$ or a vector field $z=\,^t(z_1,\dots,z_d)$ ($\,^t$ denotes the transpose) we consider the classical notation: $\nabla  \pi=\,^t(\partial_{x_1}\pi,\dots,\partial_{x_d}\pi)$, 
$\nabla z=(\partial_{x_j} z_i)_{1\leq i,j\leq d}$, $\,^t\nabla z=(\partial_{x_i} z_j)_{1\leq i,j\leq d}$ and we also set $D^s z\ov \nabla z+\,^t\nabla z$ and $D^a z\ov \nabla z-\,^t\nabla z$.
 We recall that the divergence of $z$ is defined by 
${\rm div\,} z=\sum_{j=1}^d \partial_{x_j} z_j$ and the curl of $z$ or $\pi$ is defined by
$$
{\rm curl\,} z=\partial_{x_1}z_{2}-\partial_{x_2}z_{1}\quad\mbox{ and }\quad {\rm curl\,} \pi= \vct{\partial_{x_2}\pi}{-\partial_{x_1}\pi}\quad\mbox{ if $d=2$},
$$
and
$$
{\rm curl\,} z=\vctt{\displaystyle \partial_{x_2}z_{3}-\partial_{x_3}z_{2}}{\displaystyle \partial_{x_3}z_{1}-\partial_{x_1}z_{3}}{\displaystyle \partial_{x_1}z_{2}-\partial_{x_2}z_{1}}\quad\mbox{ if $d=3$}.
$$

\subsection{Stabilizability of MHD system}\label{stabMHD}
\par
Let $\Omega$ be a bounded open subset of $\mathbb{R}^3$ of class $C^{2,1}$ and consider a stationary solution $(w^S,\theta^S,r^S)$ of the following MHD system: 
\begin{equation}\label{eqMHDstat}
 \left\{
\begin{split}
 -\Delta w^S+(w^S\cdot \nabla) w^S-({\rm curl\, \theta^S})\times \theta^S+\nabla r^S&=f^S&\quad \mbox{ in }\Omega,\\
 {\rm curl}({\rm curl}\,\theta^S)-{\rm curl}(w^S\times \theta^S) &=0&\quad \mbox{ in }\Omega,\\
{\rm div}\, w^S={\rm div}\, \theta^S&=0&\quad \mbox{ in }\Omega.
\end{split}
\right.
\end{equation}
Here $w^S(x)\in \mathbb{R}^3$ represents the velocity of the fluid, $p^S(x)\in \mathbb{R}$ is the pressure, $\theta^S(x)$ is the magnetic field, $f^S(x)$ is a given 
stationary body force, $n$ denotes the unit exterior normal vector field defined on $\partial\Omega$ and $\times$ denotes the vector product. Let also underline that in 
\eqref{eqMHDstat}, usual non-dimensional constants that characterize the flow (Hartmann number, interaction parameter and magnetic Reynolds number) 
are supposed to be equal to one for simplicity. Concerning the well-posedness of system \eqref{eqMHDstat} we refer to \cite{Meir1993}. 
In the following, $w^S$ and $\theta^S$ are supposed to be smooth enough: $w^S\in (H^2(\Omega;\mathbb{R}))^3$ and
$\theta^S\in (H^2(\Omega;\mathbb{R}))^3$.

Our aim is to stabilize around a given solution of \eqref{eqMHDstat} by means of a distributed control localized in an open subset $\omega\subset\subset \Omega$. More precisely, 
for $(y_0,\vartheta_0)\in (L^2(\Omega;\mathbb{R}))^3\times (L^2(\Omega;\mathbb{R}))^3$ satisfying
\begin{equation}
 {\rm div}\, y_0={\rm div\,\vartheta_0}=0\; \mbox{ in }\Omega\quad \mbox{ and }\quad y_0\cdot n= \vartheta_0\cdot n=0\; \mbox{ on }\partial\Omega,\label{CCini}
\end{equation}
consider the following instationary MHD system
\begin{equation}\label{eqMHD}
 \left\{
\begin{split}
 w_t-\Delta w+(w\cdot \nabla) w-({\rm curl\, \theta})\times \theta+\nabla r&=f^S+{\bf 1}_\omega u^1 &\quad \mbox{ in }Q,\\
 \theta_t+{\rm curl}({\rm curl}\,\theta)-{\rm curl}(w\times \theta) &={\bf 1}_\omega P_\omega u^2 &\quad \mbox{ in }Q,\\
{\rm div}\, w={\rm div}\, \theta&=0&\quad \mbox{ in }Q,\\
w=w^S,\quad \theta\cdot n&=\theta^S\cdot n&\quad \mbox{ on }\Sigma,\\
({\rm curl}\,\theta-w\times \theta)\times n&=({\rm curl}\,\theta^S-w^S\times \theta^S)\times n&\quad \mbox{ on }\Sigma,\\
w(0)=w^S+y_0,\quad \theta(0)&=\theta^S+\vartheta_0&\quad \mbox{ in }\Omega,
\end{split}
\right.
\end{equation}
where $Q\ov (0,+\infty)\times \Omega$ and $\Sigma\ov (0,+\infty)\times  \partial\Omega$. Here ${\bf 1}_\omega$ is the extension operator defined on $(L^2(\omega))^3$ 
by ${\bf 1}_\omega(y)(x)=y(x)$ if $x\in \omega$ and ${\bf 1}_\omega(y)(x)=0$ else, $u=(u^1, u^2)$ is a control function in 
$(L^2((0,T)\times \omega;\mathbb{R}))^3\times (L^2((0,T)\times \omega;\mathbb{R}))^3$ and 
$P_\omega$ is the classical Helmholtz projector related to $\omega$ (i.e the orthogonal projection operator from  $(L^2(\omega;\mathbb{R}))^3$ onto 
the completion of $\{v\in (C_0^\infty(\omega;\mathbb{R}))^3\mid {\rm div\,} v=0\mbox{ in }\omega\}$ for the norm of $(L^2(\omega;\mathbb{R}))^3$). Note that here $ \Tilde u^2=P_\omega u^2$  satisfies \eqref{Propu2}. This guarantees that the right hand side of the magnetic field equation is divergence free and then that it is compatible with the left hand side. This follows from the fact that 
the orthogonal complement of the range of $P_\omega$ is composed with $\nabla p\in (L^2(\omega;\mathbb{R}))^3$ for $p\in H^1_{\rm loc}(\omega;\mathbb{R})$ 
(see \cite[Chap. III]{Galdi}), which implies 
$$
\forall p\in (H_0^{1}(\Omega))^3\quad \langle {\rm div \,} {\bf 1}_\omega P_\omega u^2 \,,\, p \rangle_{(H^{-1}(\Omega))^3, (H_0^{1}(\Omega))^3}=
-\int_\Omega {\bf 1}_\omega P_\omega u^2\cdot \nabla p \,{\rm d}x   =- \int_\omega P_\omega u^2\cdot \nabla p\, {\rm d}x   =0.
$$

The idea of using such a type of control for the magnetic field equation is due to \cite{Lefter2010} for the two dimensional version of system \eqref{eqMHD}. 
If $\Omega$ is supposed to be only bounded in directions $x_1$, $x_2$ and invariant in direction $x_3$, 
if all functions only depend on $x_1$, $x_2$ and if each vector field has its 
third component equal to zero (and then is identified to its two first components i.e $w(x)=\,^t(w_1(x_1,x_2), w_2(x_1,x_2), 0)\equiv \,^t(w_1(x_1,x_2), w_2(x_1,x_2))$ etc), 
then we can identify $\Omega$ to its bounded two dimensional section and \eqref{eqMHD} reduces to: 
\begin{equation}\label{eqMHD2D}
 \left\{
\begin{split}
 w_t-\Delta w+(w\cdot \nabla) w-({\rm curl\, \theta}) \theta^\perp+\nabla r&=f^S+{\bf 1}_\omega u^1 &\mbox{ in }Q,\\
 \theta_t+{\rm curl}({\rm curl}\,\theta)+{\rm curl}(w \cdot \theta^\perp) &={\bf 1}_\omega P_\omega u^2 &\mbox{ in }Q,\\
{\rm div}\, w={\rm div}\, \theta&=0 &\mbox{ in }Q,\\
w=w^S,\; \theta\cdot n&=\theta^S\cdot n &\mbox{ on }\Sigma\\
{\rm curl}\, \theta+w\cdot \theta^\perp&={\rm curl}\, \theta^S+w^S\cdot \theta^{S\perp} &\mbox{ on }\Sigma\\
w(0)=w^S+y_0,\quad \theta(0)&=\theta^S+\vartheta_0&\quad \mbox{ in }\Omega,
\end{split}
\right.
\end{equation}
where $(y_0,\vartheta_0)\in (L^2(\Omega;\mathbb{R}))^2\times (L^2(\Omega;\mathbb{R}))^2$ satisfies \eqref{CCini} and $u^1$, $u^2$ both belong to $(L^2((0,T)\times \omega;\mathbb{R}))^2$.
We have used the notation $\,^t(a_1,a_2)^\perp=\,^t(-a_2,a_1)$. 
In \cite{Lefter2010} the stabilizability of \eqref{eqMHD2D} is obtained for a stationary pair $(w^S,\theta^S)\in (W^{2,\infty}(\Omega))^2\times (W^{2,\infty}(\Omega))^2$ 
satisfying homogeneous boundary conditions and for $\Omega$ simply connected.   The main step of the proof in \cite{Lefter2010} consists in checking 
 the approximate controllability criterion \eqref{AppConttau} related to the linear non stationary adjoint system associated to \eqref{eqMHD2D}.

Here we extend such a stabilizability result to the three dimensional case and for stationary state only $H^2$.
More precisely, we show the existence of finite dimensional control functions $u^1$, $u^2$ in feedback form, such that for all $(y_0,\vartheta_0)$ sufficiently small 
the solution $(w,\theta)$ of \eqref{eqMHD} (or of \eqref{eqMHD2D}) satisfies $(w(t),\theta(t)) \to (w^S,\theta^S)$ as $t\to \infty$, see Theorem \ref{StabMHD2D} below. 
Note that our proof relies on a uniqueness theorem for a stationary system which is simpler to handle (we prove \eqref{UC} instead of \eqref{AppConttau}). 
Moreover, we will see that assuming $\Omega$ to be simply connected is essential for the stabilizability of the linear system obtained from \eqref{eqMHD} (or 
\eqref{eqMHD2D}) by linearizing around $(w^S,\theta^S)$, see Remark \ref{RkmultCon} below.

Our strategy consists in first rewriting \eqref{eqMHD} in the abstract form \eqref{eq1.1nl} with $A$, $B$, satisfying 
$(\mathcal{H}_1)$, $(\mathcal{H}_3)$, $(\mathcal{H}_5)$, $(\mathcal{H}_6)$; second proving Fattorini's criterion $\eqref{UC}$ by using Corollary \ref{CorMHD} of the Appendix with the fact that 
$\Omega$ is simply connected;
third applying Theorem \ref{ThmNL} (see Remark \ref{SectStabreal}). Note that is the strategy sketched in \cite{BT} for the Navier-Stokes system and 
for which $\eqref{UC}$ is obtained from the uniqueness result for the Oseen system recalled in Theorem \ref{UCOseen}. 
For sake of clarity, we only detail the calculations in the three dimensional case. Adaptation to the 2D case is straightforward and is left to the reader.
\medskip

{\bf Step 1. Abstract reformulation.}  With the following formulas of vectorial analysis:
$$
({\rm curl}\, a)\times a=(a\cdot \nabla) a -\nabla(|a|^2/2)\quad \mbox{ and }\quad {\rm curl}(a\times b)=(b\cdot \nabla) a-(a\cdot \nabla) b+({\rm div\,b })a-({\rm div\,a })b
$$
we first deduce that $(y,\vartheta, p)=(w-w^S,\theta-\theta^S, r-r^S+(|\theta|^2-|\theta^S|^2)/2)$ obeys
\begin{equation}\label{eqMHD2}
 \left\{
\begin{split}
 y_t-\Delta y+(w^S\cdot \nabla) y+(y\cdot \nabla) w^S-(\theta^S\cdot \nabla \vartheta)-(\vartheta\cdot \nabla \theta^S)+\nabla p&=(\vartheta\cdot \nabla \vartheta)-(y\cdot \nabla y)+{\bf 1}_\omega u^1&\; \mbox{ in }Q,\\
 \vartheta_t+{\rm curl}({\rm curl}\,\vartheta)-{\rm curl}(y\times \theta^S)-{\rm curl}(w^S\times \vartheta) &=(\vartheta\cdot \nabla) y-(y\cdot \nabla) \vartheta+{\bf 1}_\omega P_\omega\, u^2&\; \mbox{ in }Q,\\
{\rm div}\, y={\rm div}\, \vartheta&=0&\; \mbox{ in }Q,\\
y(0)=y_0,\quad \vartheta(0)&=\vartheta_0&\; \mbox{ in }\Omega,
\end{split}
\right.
\end{equation}
with boundary conditions
$$
y=0,\quad \vartheta\cdot n=0,\quad ({\rm curl}\, \vartheta-w^S\times \vartheta)\times n=0\quad \mbox{ on }\Sigma.
$$
Note that with $(w^S\times \vartheta)\times n=(w^S\cdot n)\vartheta-(\vartheta \cdot n)w^S$ and $\vartheta \cdot n=0$, the above boundary conditions become
\begin{equation}
y=0,\quad  \vartheta\cdot n=0,\quad ({\rm curl}\, \vartheta)\times n-(w^S\cdot n)\vartheta=0 \quad \mbox{ on }\Sigma.\label{BDCOND}
\end{equation}
Next, we introduce the (real) spaces
\begin{equation}\label{DefH}
G\ov {\bf V}_n^0(\Omega)\times {\bf V}_n^0(\Omega)\quad \mbox{ where }\quad{\bf V}_n^0(\Omega)\ov \{y\in (L^2(\Omega;\mathbb{R}))^3\mid {\rm div}\, y=0 \mbox{ in }\Omega,\; y\cdot n=0\;\mbox{ on }\partial\Omega\}.
\end{equation}
  We denote by $P$ the Helmholtz projection operator related to $\Omega$ (i.e the orthogonal projection operator from $(L^2(\Omega;\mathbb{R}))^3$ onto ${\bf V}_n^0(\Omega)$) and 
we define the following linear operator $A:\mathcal{D}(A)\subset G \to G$ by 
\begin{equation}\nonumber
\begin{split}
 A\begin{bmatrix}y\\ \vartheta \end{bmatrix}&=
\begin{bmatrix}
\displaystyle P\left(\Delta y-(w^S\cdot\nabla)y-(y\cdot\nabla)w^S+(\theta^S\cdot\nabla)\vartheta+(\vartheta\cdot\nabla)\theta^S\right)\\
\displaystyle -{\rm curl}({\rm curl}\,\vartheta)+{\rm curl}(w^S\times \vartheta)+{\rm curl}(y\times \theta^S)
\end{bmatrix}\\
\mathcal{D}(A)&=\big{\{}(y,\vartheta)\in (H^2(\Omega;\mathbb{R}))^3\times (H^2(\Omega;\mathbb{R}))^3\mid 
{\rm div}\, y={\rm div}\, \vartheta=0\; \mbox{ in }\Omega, \\
&\quad\qquad\quad y=0,\;\;  \vartheta\cdot n=0\;\;\mbox{ and }\; ({\rm curl}\, \vartheta\times n-
(w^S\cdot n)\vartheta)=0\mbox{ on }\partial\Omega \big{\}}.
\end{split}
\end{equation}
Note that $P$ does not appear in the second component of $A\,^t[y, \vartheta ]$ because ${\rm curl}({\rm curl}\,\vartheta-w^S\times \vartheta-y\times \theta^S)$ belongs to 
${\bf V}_n^0(\Omega)$. In particular, the fact that its normal component is zero on the boundary comes from $y=({\rm curl}\, \vartheta-w^S\times \vartheta)\times n=0$
on $\partial\Omega$  with the following calculation for all $q\in H^2(\Omega)$ and a density argument :
\begin{equation}\nonumber
 \begin{split}
 \int_{\partial\Omega}{\rm curl}({\rm curl}\,\vartheta-w^S\times \vartheta-y\times \theta^S)\cdot n q{\rm d\sigma }&=\int_\Omega {\rm curl}({\rm curl}\,\vartheta-w^S\times \vartheta-y\times \theta^S)\cdot\nabla q {\rm d}x   \\
&=\int_{\partial\Omega}({\rm curl}\,\vartheta-w^S\times \vartheta-y\times \theta^S)\times n \cdot \nabla q{\rm d\sigma }=0.
\end{split}
\end{equation}
Since $\Omega$ and $(w^S,\theta^S)$ are regular enough, with analogous arguments as in \cite{BT,BADRA-DCDS-A2011} we verify 
that $A$ generates an analytic semigroup and
that it has compact resolvent. It implies that $({\mathcal H}_1)$, $({\mathcal H}_3)$ and $({\mathcal H}_5)$ are satisfied. 
Moreover, we can verify that the adjoint of $A$ is given by:
\begin{equation}\label{DAstar} 
\begin{split} 
 A^*\begin{bmatrix}z\\\rho \end{bmatrix}&=
\begin{bmatrix}
\displaystyle P(\Delta z+(D^s z) w^S-(D^a \rho) \theta^S)\\
\displaystyle P(\Delta \rho-(D^s z) \theta^S+(D^a \rho) w^S)
\end{bmatrix}\\
\mathcal{D}(A^*)&=\big{\{}(z,\rho)\in (H^2(\Omega;\mathbb{R}))^3\times (H^2(\Omega;\mathbb{R}))^3\mid {\rm div}\, z={\rm div}\, \rho=0\; \mbox{ in }\Omega,\\  
&\qquad\qquad\qquad\qquad \qquad z=0\mbox{ on }\partial\Omega, \rho\cdot n=({\rm curl}\, \rho)\times n=0\mbox{ on }\partial\Omega \big{\}}.
\end{split}
\end{equation}
Moreover, if we define $B:W\to G$ as follows 
\begin{equation}\label{inputB}
W\ov (L^2(\omega;\mathbb{R}))^3\times (L^2(\omega;\mathbb{R}))^3,\quad B\begin{bmatrix}u^1\\ u^2\end{bmatrix} \ov \begin{bmatrix}P {\bf 1}_\omega u^1\\ {\bf 1}_\omega P_\omega\, u^2 \end{bmatrix},
\end{equation}
then $(\mathcal{H}_5)$ (for $\gamma=0$) is satisfied. 
Finally, \eqref{eqMHD2}-\eqref{BDCOND} can be rewritten in the form \eqref{eq1.1nl} with:
\begin{equation}\label{NLMHD}
\mathcal{N}\left(\begin{bmatrix}y \\ \vartheta \end{bmatrix},\begin{bmatrix}u^1 \\ u^2 \end{bmatrix} \right)
\ov \begin{bmatrix}P ((\vartheta\cdot \nabla \vartheta)-(y\cdot \nabla y)) \\ (\vartheta\cdot \nabla) y-(y\cdot \nabla) \vartheta \end{bmatrix}.
\end{equation}
\medskip

{\bf Step 2. Verification of Fattorini's criterion.} Since $A$ and $B$ satisfy $(\mathcal{H}_1)$, $(\mathcal{H}_3)$, $(\mathcal{H}_5)$ and $(\mathcal{H}_6)$, Theorem \ref{CorStabReal} applies and stabilizability of $(A+\sigma,B)$ for all $\sigma>0$
is reduced to \eqref{UC}, which is to say that for all $\lambda\in \mathbb{C}$, every $(z,\pi,\rho, \kappa )$ that satisfies:
\begin{equation}\label{MHDadj0}
 \left\{
\begin{split}
 \lambda  z-\Delta z-(D^s z) w^S+(D^a \rho) \theta^S +\nabla \pi &=0 &\quad \mbox{ in }\Omega,\\
 \lambda  \rho-\Delta \rho+(D^s z) \theta^S-(D^a \rho) w^S+\nabla \kappa &=0&\quad \mbox{ in }\Omega,\\
{\rm div}\, z={\rm div}\, \rho&=0&\quad \mbox{ in }\Omega,
\end{split}
\right.
\end{equation}
with boundary data
\begin{equation}\label{bdydataMHD}
 z=0,\quad  \rho\cdot n=0,\quad ({\rm curl}\, \rho)\times n=0\quad \mbox{ on }\partial\Omega,
\end{equation}
and such that
\begin{equation}
z=0,\quad P_\omega\rho=0\quad  {\rm in }\;\; \omega,\label{Nullbord00}
\end{equation}
must be identically equal to zero in $\Omega$. Note that \eqref{Nullbord00} is equivalent to 
\begin{equation}
z=0,\quad \rho=\nabla p\quad  {\rm in }\;\; \omega\quad \mbox{ for some $p\in H^1_{\rm loc}(\omega;\mathbb{R})$.}\label{Nullbord0}
\end{equation}
Then \eqref{Nullbord0} implies $z={\rm curl\,}\rho=0$ in $\omega$
and from Corollary \ref{CorMHD} of the Appendix, the fact that $\Omega$ is simply connected guarantees that the above uniqueness result is true.

\medskip

{\bf Step 3. Stabilizing control for linear and nonlinear systems.}    From Theorem \ref{CorStabReal}, for all $\sigma>0$, there exists $K>0$ and there exist families 
\begin{equation}
((v_j^1,v_j^2))_{j=1,\dots,K}\in (B^*\mathcal{F}_\sigma)^K\subset W^K\quad \mbox{ and }\quad ((\widehat z_j,\widehat \rho_j))_{j=1,\dots,K}\in G^K\label{Families}
\end{equation}
and a finite rank feedback operator $F: G\to W$ defined by
\begin{equation}\label{FeedbackF}
F \begin{bmatrix} \xi \\ \zeta \end{bmatrix}= \sum_{j=1}^K \begin{bmatrix} v_j^1 \\ v_j^2 \end{bmatrix}\int_\Omega \left (\xi\cdot \widehat z_j+\zeta \cdot \widehat \rho_j\right ){\rm d} x ,
\end{equation}
such that $A+BF$ with domain $\mathcal{D}(A+BF)=\mathcal{D}(A)$ (since $B$ is bounded) is the infinitesimal generator af an exponentially stable semigroup on $G$. 
Moreover, as it happens for the Navier-Stokes system treated in \cite{BT,BADRASICON}, Sobolev embeddings 
guarantee that   the nonlinearity $\mathcal{N}$ defined by \eqref{NLMHD} satisfies \eqref{HypNL} only for $s\in [\frac{d-2}{2}, 1]$. This yields the following stabilization theorem for
system \eqref{eqMHD} with feedback  control
\begin{equation}\label{ControlFeed}
 \begin{split}
  u^1(t)&=\sum_{j=1}^K v_j^1 \int_\Omega \left ((w(t)-w^S)\cdot \widehat z_j+(\theta(t)-\theta^S)\cdot  \widehat \rho_j\right ){\rm d} x , \\
u^2(t)&=\sum_{j=1}^K v_j^2 \int_\Omega \left ((w(t)-w^S)\cdot \widehat z_j+(\theta(t)-\theta^S)\cdot  \widehat \rho_j\right ){\rm d} x .
 \end{split}
\end{equation}
\begin{Theorem}\label{StabMHD2D}
 Let $s\in [\frac{d-2}{2}, 1]\backslash\{\frac{1}{2}\}$, let $\Omega$ be a bounded and simply-connected open subset of $\mathbb{R}^d$ of class $C^{2,1}$, let $(w^S,\theta^S)\in (H^2(\Omega;\mathbb{R}))^d\times (H^2(\Omega;\mathbb{R}))^d$ 
satisfy \eqref{eqMHDstat} and let $(y_0,\vartheta_0)\in (H_0^s(\Omega;\mathbb{R}))^d\times (H^s(\Omega;\mathbb{R}))^d$ 
satisfy \eqref{CCini}.

For all $\sigma>0$ there exists $\mu>0$ such that if
$$\|y_0\|_{(H_0^{s}(\Omega;\mathbb{R}))^d}+\|\vartheta_0\|_{(H^s(\Omega;\mathbb{R}))^d}\leq \mu,$$
then system \eqref{eqMHD}, \eqref{ControlFeed} if $d=3$ or system \eqref{eqMHD2D}, \eqref{ControlFeed} if $d=2$ admits a solution $(w,r,\theta)$ in
\begin{equation}\nonumber
\begin{split}
\{(w^S,r^S,\theta^S)\}+W_\sigma((H^{s+1}(\Omega;\mathbb{R}))^d, (H^{s-1}(\Omega;\mathbb{R}))^d) \times H^{\frac{s-1}{2}}(L^{2}(\Omega;\mathbb{R})/{\mathbb R})\times 
W_\sigma((H^{s+1}(\Omega;\mathbb{R}))^d, (H^{s-1}(\Omega;\mathbb{R}))^d)
\end{split}
\end{equation}
which is unique within the class of function in
\begin{equation}\nonumber
\begin{split}
\{(w^S,r^S,\theta^S)\}+ W_{{\rm loc}}((H^{s+1}(\Omega;\mathbb{R}))^d, (H^{s-1}(\Omega;\mathbb{R}))^d) \times H_{{\rm loc}}^{\frac{s-1}{2}}(L^{2}(\Omega;\mathbb{R})/{\mathbb R})\times 
W_{{\rm loc}}((H^{s+1}(\Omega;\mathbb{R}))^d, (H^{s-1}(\Omega;\mathbb{R}))^d).
\end{split}
\end{equation}
Moreover, there exists $C>0$ such that for all $t\geq 0$ the following estimate holds:
\begin{equation}\nonumber
\begin{split}
\|w(t)-w^S\|_{(H_0^{s}(\Omega;\mathbb{R}))^d}+\|\theta(t)-\theta^S\|_{(H^{s}(\Omega;\mathbb{R}))^d} 
\leq C e^{-\sigma t}(\|y_0\|_{(H_0^{s}(\Omega;\mathbb{R}))^d}+\|\vartheta_0\|_{(H^{s}(\Omega;\mathbb{R}))^d}).
\end{split}
\end{equation}
\end{Theorem}
\begin{Remark}
To extend the above theorem to $s=1/2$ the initial datum $y_0$ must be chosen in $(H_{00}^{1/2}(\Omega;\mathbb{R}))^d$ where $H_{00}^{1/2}(\Omega;\mathbb{R})$ stands for the subspace of $H_{0}^{1/2}(\Omega)$ composed with $y$ such that $\int_{\Omega}{\rm dist}(x,\partial \Omega)^{-1}|y|^2 \ {\rm d}x<+\infty$.
\end{Remark}
\begin{Remark}\label{RkmultCon}
According to Corollary \ref{CorMHD} the uniqueness result stated in step 2 is true if and only if $\Omega$ is simply connected. When 
$\Omega$ is multiply connected and $\lambda=0$, for each nonzero $\Tilde \rho\in X_N$ (defined by \eqref{defX} in appendix) we have that $(z,\pi,\rho,\kappa)\ov (0,0,\rho_N,0)$
satisfies \eqref{MHDadj0}-\eqref{bdydataMHD}-\eqref{Nullbord00}. 
Then there is a $N$-dimensional subspace of eigenvectors related to the zero eigenvalue which is uncontrollable and the linear system obtained from \eqref{eqMHD} 
by linearizing around $(w^S,\theta^S)$ is not stabilizable. {A way to obtain a stabilizability result in the multiply connected case is to impose on $\vartheta_0$ some zero flux conditions through cuts that joint the connected components of $\partial\Omega$, see \cite{BADRA-MHD-SUB-3}. }
\end{Remark}
\subsection{Stabilizability of micropolar system}\label{stabMicro}
The methodology described in the last section can also be used to stabilize the following 3D micropolar system:
\begin{equation}\label{eqMicrop}
 \left\{
\begin{split}
 w_t-\Delta w+(w\cdot \nabla) w-{\rm curl}\, \theta+\nabla r &=f^S&\quad \mbox{ in }Q,\\
 \theta_t-\Delta \theta+(w\cdot \nabla) \theta-{\rm curl}\, w-\nabla ({\rm div}\, \theta)&=g^S&\quad \mbox{ in }Q,\\
{\rm div}\, w&=0&\quad \mbox{ in }Q,\\
w(0)=w^S+y_0,\quad \theta(0)&=\theta^S+\vartheta_0&\quad \mbox{ in }\Omega.
\end{split}
\right.
\end{equation}
About micropolar systems and related control problems 
see for instance \cite{Lukaszewicz,Guerrero2007} and references therein. Here $\Omega$ is a bounded open subset of $\mathbb{R}^3$
of class $C^{2,1}$, $(w^S,\theta^S)\in (H^2(\Omega;\mathbb{R}))^3\times (H^2(\Omega;\mathbb{R}))^3$ verifies
\begin{equation}\label{eqMicroStat}
 \left\{
\begin{split}
 -\Delta w^S+(w^S\cdot \nabla) w^S-{\rm curl}\, \theta^S+\nabla r^S &=f^S&\quad \mbox{ in }\Omega,\\
 -\Delta \theta^S+(w^S\cdot \nabla) \theta^S-{\rm curl}\, w^S-\nabla ({\rm div}\, \theta^S)&=g^S&\quad \mbox{ in }\Omega,\\
{\rm div}\, w^S&=0&\quad \mbox{ in }\Omega,
\end{split}
\right.
\end{equation}
 and $(y_0,\vartheta_0)\in (L^2(\Omega;\mathbb{R}))^3\times (L^2(\Omega;\mathbb{R}))^3$ satisfies 
\begin{equation}\label{MicroCCini}
{\rm div\, }y_0=0\; \mbox{ in }\Omega\quad \mbox{ and }\quad y_0\cdot n=0\;\mbox{ on }\partial \Omega.
\end{equation}
We want to find a control function $u=(u^1,u^2)$ such that for all $(y_0,\vartheta_0)$ sufficiently small 
the solution $(w,\theta)$ of \eqref{eqMicrop} with nonhomegeneous boundary data of the form
\begin{equation}\label{bdycont}
w=w^S+m u^1\quad\mbox{ and }\quad  \theta=\theta^S+m u^2 \quad \mbox{ on }\Sigma
\end{equation}
are such that $(w(t),\theta(t)) \to (w^S,\theta^S)$ as $t\to \infty$. Here, $m\in C^2(\partial \Omega;\mathbb{R})$ is a non zero cut-off function that allows to 
localize the action of the control $(u^1,u^2)$ on a subset of $\partial\Omega$ and for all $t\geq 0$: 
\begin{equation}\label{DefU}
u^1(t)\in {\bf V}_m^0(\partial\Omega)\ov 
\left\{v\in (L^2(\partial\Omega;\mathbb{R}))^3\mid \int_{\partial\Omega} m v\cdot n{\rm d\sigma }=0\right\},\quad u^2(t)\in (L^2(\partial\Omega;\mathbb{R}))^3.
\end{equation}
For that, we follow the same strategy as in Subsection \ref{stabMHD}: first we rewrite \eqref{eqMicrop} in the abstract form \eqref{eq1.1nl} for $A$, $B$, satisfying 
$(\mathcal{H}_1)$, $(\mathcal{H}_3)$, $(\mathcal{H}_5)$, $(\mathcal{H}_6)$; second we prove Fattorini's criterion $\eqref{UC}$ by using Theorem \ref{UCMicro} of the Appendix;
third we apply Theorem \ref{ThmNL2} (see Remark \ref{SectStabreal}). 
\medskip

{\bf Step 1. Abstract reformulation.}  We first werify that $(y,\vartheta, p)=(w-w^S,\theta-\theta^S, r-r^S)$ obeys
\begin{equation}\label{eqMicro2}
 \left\{
\begin{split}
 y_t-\Delta y+(w^S\cdot \nabla) y+(y\cdot \nabla) w^S-{\rm curl}\, \vartheta+\nabla p&=-(y\cdot \nabla) y&\quad \mbox{ in }Q,\\
 \vartheta_t-\Delta \vartheta+(w^S\cdot \nabla) \vartheta+(y\cdot \nabla) \theta^S-{\rm curl}\, y-\nabla ({\rm div}\, \vartheta) &=-(y\cdot \nabla) \vartheta&\quad \mbox{ in }Q,\\
{\rm div}\, y&=0&\quad \mbox{ in }Q,\\
y(0)=y_0,\quad \vartheta(0)&=\vartheta_0&\quad \mbox{ in }\Omega,
\end{split}
\right.
\end{equation}
with boundary conditions
$$
y=m u^1\quad \mbox{ and }\quad \vartheta=m u^2 \mbox{ on }\Sigma.
$$
Next, we introduce the (real) spaces
\begin{equation}\label{DefG}
G\ov {\bf V}_n^0(\Omega)\times (L^2(\Omega;\mathbb{R}))^3\quad \mbox{ where }\quad{\bf V}_n^0(\Omega)\ov \{y\in (L^2(\Omega;\mathbb{R}))^3\mid {\rm div}\, y=0 \mbox{ in }\Omega,\; y\cdot n=0\;\mbox{ on }\partial\Omega\}
\end{equation}
and we define the following 
linear operator $A:\mathcal{D}(A)\subset G \to G$ by  
\begin{equation}\nonumber
\begin{split}
 A\begin{bmatrix}y\\ \vartheta \end{bmatrix}&=
\begin{bmatrix}
\displaystyle P(\Delta y-(w^S\cdot \nabla) y-(y\cdot \nabla) w^S+{\rm curl}\, \vartheta)\\
\displaystyle \Delta \vartheta-(w^S\cdot \nabla) \vartheta-(y\cdot \nabla) \theta^S+{\rm curl}\, y+\nabla ({\rm div}\, \vartheta)
\end{bmatrix}\\
\mathcal{D}(A)&=\big{\{}(y,\vartheta)\in (H^2(\Omega;\mathbb{R}))^3\times (H^2(\Omega;\mathbb{R}))^3\mid 
{\rm div}\, y=0\; \mbox{ in }\Omega,\; y=\vartheta= 0 \mbox{ on }\partial\Omega \big{\}}.
\end{split}
\end{equation}
Recall that $P$ denotes the orthogonal projection operator from $(L^2(\Omega;\mathbb{R}))^3$ onto ${\bf V}_n^0(\Omega)$.
Since $\Omega$ and $(w^S,\theta^S)$ are regular enough, with analogous arguments as in \cite{BT,BADRA-DCDS-A2011} we verify 
that $A$ generates an analytic semigroup and
that it has compact resolvent. It implies that $({\mathcal H}_1)$, $({\mathcal H}_3)$ and $({\mathcal H}_5)$ are satisfied. 
Moreover, we can verify that the adjoint of $A$ is given by: 
\begin{equation}\label{MicroDAstar} 
\begin{split} 
 A^*\begin{bmatrix}z\\\rho \end{bmatrix}&=
\begin{bmatrix}
\displaystyle P(\Delta z+(D^s z) w^S+\,^t(\nabla \rho) \theta^S +{\rm curl}\, \rho)\\
\displaystyle \Delta \rho +(\nabla \rho) w^S+{\rm curl}\, z+\nabla ({\rm div\,}\rho)
\end{bmatrix}\\
\mathcal{D}(A^*)&=\big{\{}(z,\rho)\in (H^2(\Omega;\mathbb{R}))^3\times (H^2(\Omega;\mathbb{R}))^3\mid {\rm div}\, z=0\; \mbox{ in }\Omega, z=\rho=0
\mbox{ on }\partial\Omega \big{\}}.
\end{split}
\end{equation}
By multiplying the two first equations of \eqref{eqMicro2} by $(z,\rho)\in \mathcal{D}(A^*)$ and integrating by parts we get :
\begin{equation}\label{Microinter1}
 \begin{split}
  \frac{{\rm d}}{{\rm d}t   }\left(\begin{bmatrix}Py \\ \vartheta \end{bmatrix}, \begin{bmatrix}z\\\rho \end{bmatrix}\right)_G &= \left(\begin{bmatrix}y \\ \vartheta \end{bmatrix}, 
\begin{bmatrix}\Delta z+(D^s z) w^S+\,^t(\nabla \rho) \theta^S +{\rm curl}\, \rho \\
\Delta \rho+(\nabla \rho) w^S+{\rm curl}\, z+\nabla ({\rm div\,}\rho) \end{bmatrix}\right)_G-
\left(\begin{bmatrix}(y\cdot \nabla) y \\ (y\cdot \nabla) \vartheta \end{bmatrix}, 
\begin{bmatrix}z\\\rho \end{bmatrix}\right)_G\\
&\,-\int_{\partial\Omega}(m u^2)\cdot \left(  \frac{\partial \rho}{\partial n}+({\rm div \, \rho}) n\right) \, {\rm d\sigma}-\int_{\partial\Omega}(m u^1)\cdot \frac{\partial z}{\partial n}\, {\rm d\sigma}.
\end{split}
\end{equation}
Thus, to rewrite \eqref{Microinter1} in the abstract form \eqref{eq1.1nl} let us first write 
$$ 
\begin{bmatrix}\Delta z+(D^s z) w^S+\,^t(\nabla \rho) \theta^S +{\rm curl}\, \rho \\
 \Delta \rho+(\nabla \rho) w^S+{\rm curl}\, z+\nabla ({\rm div\,}\rho) \end{bmatrix}=A^* \begin{bmatrix}z\\\rho \end{bmatrix}+\begin{bmatrix}\nabla\pi \\\ 
0 \end{bmatrix}
$$
where the pressure function $\pi$ is uniquely determined by
\begin{equation}\label{NormalisationPressions}
\nabla \pi=(I-P)(\Delta z+(D^s z) w^S+\,^t(\nabla \rho) \theta^S +{\rm curl}\, \rho)\quad \int_{\partial\Omega} m^2 \pi{\rm d\sigma}=0.
\end{equation}
Then an integration by parts gives: 
\begin{equation}\label{Microinter2}
\left(\begin{bmatrix}y \\ \vartheta \end{bmatrix}, 
\begin{bmatrix}\Delta z+(D^s z) w^S+\,^t(\nabla \rho) \theta^S +{\rm curl}\, \rho \\
\Delta \rho+(\nabla \rho) w^S+{\rm curl}\, z+\nabla ({\rm div\,}\rho) \end{bmatrix}\right)_G=
\left(\begin{bmatrix}Py \\\vartheta \end{bmatrix}, A^* \begin{bmatrix}z\\\rho \end{bmatrix}\right)_G
+\int_{\partial\Omega} m (u^1\cdot n) \pi \,{\rm d\sigma},
\end{equation}
and \eqref{Microinter1} becomes: 
\begin{equation}\label{Microinter1bis}
 \begin{split}
  \frac{{\rm d}}{{\rm d}t   }\left(\begin{bmatrix}Py \\  \vartheta \end{bmatrix}, \begin{bmatrix}z\\\rho \end{bmatrix}\right)_G &= \left(\begin{bmatrix}Py \\ \vartheta \end{bmatrix}, A^* \begin{bmatrix}z\\\rho \end{bmatrix}\right)_G
-\int_{\partial\Omega}m \left (u^2\cdot  \left(\frac{\partial \rho}{\partial n}+({\rm div\,}\rho) n\right)+u^1\cdot \left(\frac{\partial z}{\partial n}-\pi n\right)\right)\, {\rm d\sigma} \\
&-\left(\begin{bmatrix}(y\cdot \nabla) y  \\ (y\cdot \nabla) \vartheta \end{bmatrix}, 
\begin{bmatrix}z\\\rho \end{bmatrix}\right)_G,
\end{split}
\end{equation}
which suggests to define the control operator $B:W\to [\mathcal{D}(A^*)]'$ as follows 
\begin{equation}\label{MicroB}
W\ov {\bf V}_m^0(\partial\Omega)\times (L^2(\partial\Omega;\mathbb{R}))^3,\quad 
B^*\begin{bmatrix}z \\ \rho \end{bmatrix}\ov \begin{bmatrix}m\left(\pi n-\frac{\partial z}{\partial n}\right) \\ m\left(-({\rm div\,}\rho) n-\frac{\partial \rho}{\partial n}\right) \end{bmatrix}.
\end{equation}
Boundedness properties of trace operators ensure that $B$ satisfy $(\mathcal{H}_5)$ with 
$\gamma\in  (\frac{3}{4},1)$. Note that the normalization condition in \eqref{NormalisationPressions} guarantees that the range of $B^*$ is included in $W$. This is due to the fact $ \frac{\partial z}{\partial n}$ is tangential because $z$ is divergence free and vanishs on the boundary, see formula \eqref{DecFormulaDiv} below.

Next, define the lifting mapping $D v=\xi $ where $\xi$ is solution of
\begin{equation}\label{DefD}
 \left\{
\begin{split}
 -\Delta \xi+\nabla q&=0& \mbox{ in }\Omega,\\
{\rm div}\, \xi&=0& \mbox{ in }\Omega,\\
\xi&=m v  &\mbox{ on }\partial\Omega.
\end{split}
\right.
\end{equation}
Then from $y-D u^1\in {\bf V}_n^0(\Omega)$ we get $(I-P)y=(I-P)D u^1$ and 
$$
\begin{bmatrix}y \\ \vartheta \end{bmatrix}=\begin{bmatrix} P y \\ \vartheta \end{bmatrix}+\begin{bmatrix} (I-P)D u^1 \\ 0 \end{bmatrix},
$$
and the nonlinearity in \eqref{Microinter1bis} can be rewritten in terms of $Py,\vartheta , u$ as follows: 
\begin{equation}\label{NLMicro}
\mathcal{N}\left(\begin{bmatrix}Py \\ \vartheta \end{bmatrix}, \begin{bmatrix}u^1 \\ u^2 \end{bmatrix}\right)
\ov - \begin{bmatrix} P\left[((Py+(I-P)D u^1)\cdot \nabla) (Py+(I-P)D u^1)\right] \\ ((Py+(I-P)D u^1)\cdot \nabla) \vartheta \end{bmatrix}.
\end{equation}
Finally, \eqref{Microinter1bis} is reduced to the following abstract formulation of type \eqref{eq1.1nl}:
\begin{equation}\label{MicrolinAbst}
\begin{bmatrix}Py \\ \vartheta \end{bmatrix}'=A \begin{bmatrix}Py \\ \vartheta \end{bmatrix}+B \begin{bmatrix}u^1 \\ u^2 \end{bmatrix}+\mathcal{N}\left(\begin{bmatrix}Py \\ \vartheta \end{bmatrix}, \begin{bmatrix}u^1 \\ u^2 \end{bmatrix}\right)\quad \mbox{ in }\;[\mathcal{D}(A^*)]'.
\end{equation}
\medskip
{\bf Step 2. Verification of Fattorini's criterion.} Since $A$ and $B$ satisfy $(\mathcal{H}_1)$, $(\mathcal{H}_3)$, $(\mathcal{H}_5)$ and $(\mathcal{H}_6)$, Theorem \ref{CorStabReal} applies and stabilizability of $(A+\sigma,B)$ for all $\sigma>0$
is reduced to \eqref{UC}, which is to say (see \eqref{MicroDAstar} and \eqref{MicroB}) that for all $\lambda\in \mathbb{C}$, every $(z,\pi,\rho, \kappa )$ satisfying 
 
\begin{equation}\label{Microadj}
 \left\{
\begin{split}
 \lambda  z- \Delta z-(D^s z) w^S-\,^t(\nabla \rho) \theta^S -{\rm curl}\, \rho+\nabla \pi &=0 &\quad \mbox{ in }\Omega,\\
  \lambda  \rho-\Delta \rho-(\nabla \rho) w^S-{\rm curl}\, z+\nabla \kappa  &=0&\quad \mbox{ in }\Omega,\\
{\rm div}\, z=0,\quad {\rm div}\, \rho+\kappa  &=0&\quad \mbox{ in }\Omega.
\end{split}
\right.
\end{equation}
with boundary conditions
$$
z=\rho=0\quad \mbox{ on }\partial\Omega
$$
and satisfying moreover 
\begin{equation}
\frac{\partial z}{\partial n}-\pi n=0\quad \mbox{ and }\quad  \frac{\partial \rho}{\partial n}-\kappa  n=0 \quad  {\rm on }\;\; \Gamma\ov {\rm Supp}(m),\label{Nullbord2}
\end{equation}
must be identically equal to zero in $\Omega$. 

With a classical extension of the domain procedure, such above uniqueness theorem can be obtained from Theorem 
\ref{UCMicro} of the Appendix applied in an extended domain. 
Let us briefly recall the argument. First by recalling the decomposition formula 
\begin{equation}\label{DecFormulaDiv}
{\rm div}\, z={\rm div}_\tau \,z_\tau +({\rm div }\, n) z\cdot n+\frac{\partial z}{\partial n}\cdot n\quad \mbox{ on }\Gamma
\end{equation} (see for instance \cite[Section 5.4.3]{HenrotPierre2005}) 
where ${\rm div}_\tau$ denotes the tangential divergence operator and $z_\tau$ the tangential component of $z$, 
from $z=0$ and ${\rm div}\, z=0$ on $\Gamma$ we get that $\frac{\partial z}{\partial n}$ is tangential on $\Gamma$. Then 
$\frac{\partial z}{\partial n}-\pi n=0$ yields $\pi=0$, $\frac{\partial z}{\partial n}=0$ and finally $\nabla z=0$ on $\Gamma$. Similarly, $\rho=0$ on $ \Gamma$ 
implies $\frac{\partial \rho}{\partial n}\cdot n={\rm div}\,\rho=-\kappa$ on $\Gamma$. Since from the second equality in \eqref{Nullbord2}
we also have $\frac{\partial \rho}{\partial n}\cdot n=\kappa$ we deduce that $\kappa=0$ and then $\frac{\partial \rho}{\partial n}=0$ on $\Gamma$, and since 
$\rho=0$ on $ \Gamma$ we get $\nabla \rho =0$ on $ \Gamma$.
Finally, we have proved that $ \nabla z$, $\nabla\rho$, $\pi$, $\kappa$ vanish on $\Gamma$ and we can smoothly extend $(z, \pi, \rho, \kappa )$ 
by zero in a larger domain containing an outside open set $\omega$. Then it suffices to invoke Theorem \ref{UCMicro} for this larger domain.
\medskip

{\bf Step 3. Stabilizing control for linear and nonlinear systems.}  From Theorem \ref{CorStabReal} we obtain the existence of families of type \eqref{Families}
(where $W$ and $G$ are given by \eqref{DefG}, \eqref{MicroB}) and a finite rank feedback operator $F: G\to W$ of the form \eqref{FeedbackF} such that $A+BF$ with domain 
$\mathcal{D}(A+BF)$ is the infinitesimal generator of an exponentially stable semigroup on $G$. However, because $B$ is now unbounded, $\mathcal{D}(A+BF)$ is not equal to
$\mathcal{D}(A)$ as it was the case in Subsection \ref{stabMHD}. More precisely, by following the method considered in \cite{BT} it can be proved that $\mathcal{D}(A+BF)$ is composed of 
$(P\xi,\zeta)$ where $(\xi,\zeta)\in (H^2(\Omega))^3\times (H^2(\Omega))^3$ obey
\begin{equation}\label{traceFeedback}
\begin{bmatrix} \xi \\ \zeta\end{bmatrix}=
m\sum_{j=1}^K \begin{bmatrix} v_j^1 \\ v_j^2 \end{bmatrix}\int_\Omega \left (\xi\cdot \widehat z_j+\zeta \cdot \widehat \rho_j\right ){\rm d} x \quad \mbox{ on }\partial\Omega,
\end{equation}
and that for each $s\in (0,1)$ the space $\mathcal{D}((-A_F)^\frac{s}{2})$ is a closed subspace of 
$(H^{s}(\Omega;\mathbb{R}))^3\times (H^{s}(\Omega;\mathbb{R}))^3$ composed with elements satisfying trace conditions on $\partial \Omega$ 
related to the feedback law $F$ when $s\geq 1/2$. Moreover, as it happens for the Navier-Stokes system treated in \cite{BT,BADRASICON}, Sobolev embeddings 
guarantee that the nonlinearity defined in \eqref{NLMicro} satisfies \eqref{HypNL} only for $s\in [\frac{d-2}{2}, 1]$ where $d$ is the space dimension. Then feedback stabilization of 3D micropolar 
system \eqref{eqMicro2} cannot be 
obtained unless very specific initial trace compatibility conditions related to $F$ are satisfied. However, since when $d=2$ inequality \eqref{HypNL} is satisfied for $s\in [0, 1/2)$, 
a relevant stabilization theorem could be obtained from Theorem \ref{ThmNL} for the two dimensional version of system \eqref{eqMicrop}-\eqref{bdycont} 
with feedback  control \eqref{ControlFeed}.

To obtain a 3D stabilization result we have to consider a dynamical control. From \eqref{UC} we have the existence of a complex matrix  $\Lambda$ of size $K\times K$ and of a family $\{\widehat \varepsilon _j=(\widehat \varepsilon _j^1,\widehat \varepsilon _j^2)\in {\bf V}_n^0(\Omega)\times (L^2(\Omega;\mathbb{R})^3 \,;\,j=1,\dots, K\}$,
such that dynamical controls:
\begin{equation}\label{Microfin1}
\begin{split}
\begin{bmatrix}u^1(t) \\ u^2(t) \end{bmatrix}&=\sum_{j=1}^K u_j(t) \begin{bmatrix}v_j^1 \\ v_j^2 \end{bmatrix},\quad \overline{u}\ov (u_1,\dots,u_K) \\
\overline{u}'+\Lambda \overline{u}&=\sum_{j=1}^K e_j \int_\Omega \left( \widehat \varepsilon _j^1\cdot (w-w^S)+ \widehat \varepsilon _j^1\cdot (\theta-\theta^S)\right){\rm d}x \\
\overline{u}(0)&=0
\end{split}
\end{equation}
stabilize the linear system obtained from \eqref{MicrolinAbst} by linearizing around zero. In the above setting $\{e_j\,;\,j=1,\dots, K\}$ denotes a basis of $\mathbb{R}^K$. The following stabilization theorem can be deduced from
 Theorem \ref{ThmNL2} with $s=1$ and Remark \ref{RemReal}.
\begin{Theorem}\label{StabMicro3D}
Let $\Omega$ be a bounded and open subset of $\mathbb{R}^d$ of class $C^{2,1}$, let $(w^S,\theta^S)\in (H^2(\Omega;\mathbb{R}))^d\times (H^2(\Omega;\mathbb{R}))^d$ 
satisfy \eqref{eqMicroStat} and let $(y_0,\vartheta_0)\in (H_0^1(\Omega;\mathbb{R}))^d\times (H_0^1(\Omega;\mathbb{R}))^d$ 
satisfy \eqref{MicroCCini}.
For all $\sigma>0$ there exists $\mu>0$ such that if
$$\|y_0\|_{(H^{1}(\Omega;\mathbb{R}))^3}+\|\vartheta_0\|_{(H^{1}(\Omega;\mathbb{R}))^3}\leq \mu,$$
then system \eqref{eqMicrop}-\eqref{bdycont}-\eqref{Microfin1} admits a solution $(w,r,\theta, \overline{u})$ in
\begin{equation}\nonumber
\begin{split}
\{(w^S,r^S,\theta^S,0)\}&+W_\sigma((H^{2}(\Omega;\mathbb{R}))^3, (L^{2}(\Omega;\mathbb{R}))^3) \times L^{2}(H^{1}(\Omega;\mathbb{R})/{\mathbb R})\\
&\times W_\sigma((H^{2}(\Omega;\mathbb{R}))^3, (L^{2}(\Omega;\mathbb{R}))^3)\times H^1(\mathbb{R}^+;\mathbb{R}^K),
\end{split}
\end{equation}
which is unique within the class of function in
\begin{equation}\nonumber
\begin{split}
\{(w^S,r^S,\theta^S,0)\}&+W_{{\rm loc}}((H^{2}(\Omega;\mathbb{R}))^3, (L^{2}(\Omega;\mathbb{R}))^3) \times L_{{\rm loc}}^{2}(H^{1}(\Omega;\mathbb{R})/{\mathbb R})\\
&\times W_{{\rm loc}}((H^{2}(\Omega;\mathbb{R}))^3, (L^{2}(\Omega;\mathbb{R}))^3)\times H_{{\rm loc}}^1(\mathbb{R}^+;\mathbb{R}^K),
\end{split}
\end{equation}
Moreover, there exists $C>0$ such that for all $t\geq 0$ the following estimate holds:
\begin{equation}\nonumber
\begin{split}
\|w(t)-w^S\|_{(H^{1}(\Omega;\mathbb{R}))^3}+\|\theta(t)-\theta^S\|_{(H^{1}(\Omega;\mathbb{R}))^3}+\|\overline{u}(t)\|_{\mathbb{R}^K}\\
\qquad \leq C e^{-\sigma t}(\|w_0-w^S\|_{(H^{1}(\Omega;\mathbb{R}))^3}+\|\theta_0-\theta^S\|_{(H^{1}(\Omega;\mathbb{R}))^3}).
\end{split}
\end{equation}
\end{Theorem}

\section{Appendix: uniqueness theorems for coupled Stokes type systems}\label{Appendix}
In the present appendix, we first provide local Carleman inequalities for the Stokes system that are useful to prove Fattorini's criterion \eqref{UC} corresponding to Stokes and 
coupled Stokes like systems. Thus, as a first simple illustration of such inequalities, we recover in a uniqueness theorem for Oseen equations originally obtained by Fabre and Lebeau in \cite{FL1996}. Next, we prove two uniqueness theorem for coupled Stokes like system that are needed in Section \ref{sect5}: one for an
adjoint MHD system and one for an adjoint micropolar system. 

In what follows we use the notation of Section \ref{sect5}. Moreover, 
$\mathcal{O}$ denotes a nonempty bounded open subset of $\mathbb{R}^d$ of class $C^2$ and $\psi:{\mathcal{O}}\to \mathbb{R}$ is a function satisfying
\begin{equation}
\psi\in C^2({\mathcal{O}};\mathbb{R}),\quad \psi>0\quad \mbox{ and }\quad |\nabla \psi|>0\quad \mbox{ on }\;\mathcal{O}.\label{eqpsi}
\end{equation}
\subsection{Carleman inequalities for the Stokes system}

Let us recall a well-known Carleman inequality for the Laplace equation.
\begin{Theorem}\label{CarlEll}
Let $k\in \{0,1\}$, $F_0\in L^2(\mathcal{O};\mathbb{C})$ and $F_1\in (L^2(\mathcal{O};\mathbb{C}))^d$. 
There exist $C>0$ and $\widehat \tau >1$ such that for all $\tau \geq \widehat \tau $ there exists $\widehat s(\tau )$ such that for all $s\geq \widehat s(\tau )$
and for all $u\in H_0^1(\mathcal{O};\mathbb{C})$ solution of
$$
-\Delta u=F_0+{\rm div\,}F_1 \quad \mbox{ in }\quad \mathcal{O}
$$
the following inequality holds:
\begin{equation}\label{CarlEstEll}
 \int_{\mathcal{O}} \left(e^{(k-1)\tau \psi}|\nabla u|^2+s^2\tau ^2 e^{(k+1)\tau \psi} |u|^2 \right) e^{2 s e^{\tau \psi}}{\rm d}x   \leq C\int_{\mathcal{O}}  \left(
s e^{k\tau \psi} |F_1|^2 +s^{-1}\tau ^{-2}  e^{(k-2)\tau \psi} |F_0|^2\right) e^{2 s e^{\tau \psi}}{\rm d}x   .
\end{equation}
\end{Theorem}
Inequality \eqref{CarlEstEll} for $k=1$ can be obtained for instance from \cite[Thm A.1]{ImmPuel2003} and \eqref{CarlEstEll} for $k=0$ is obtained by applying 
\eqref{CarlEstEll} with $k=1$ to the equation satisfied by $e^{-\frac{\tau}{2} \psi} u$. Note that the proof of \eqref{CarlEstEll} proposed in the above quoted work 
is performed for a $\mathbb{R}$-valued function $u$, but it is easily checked that it can be done in the same way for a $\mathbb{C}$-valued function $u$.

From Theorem \ref{CarlEll}, we deduce the following Carleman inequalities for the Stokes system.
\begin {Corollary}\label{CarlLocalStokes}
Let $\alpha\in \mathbb{R}\backslash\{-1\}$. There exist $C>0$ and $\widehat \tau >1$ such that for all $\tau \geq \widehat \tau $ there exists $\widehat s(\tau )$ 
such that for all $s\geq \widehat s(\tau )$ and
for all $(z,\pi)\in (H_0^2({\mathcal{O}};\mathbb{C}))^d\times H_0^1({\mathcal{O}};\mathbb{C})$ the following inequalities hold:
\begin{eqnarray}
 \int_{\mathcal{O}} \left(|\nabla z|^2 +s^2\tau ^2e^{2\tau \psi} |z|^2\right) e^{2 s e^{\tau \psi}}{\rm d}x   &\leq& 
C\int_{\mathcal{O}} (se^{\tau \psi} |{\rm div}\, z+\alpha \pi |^2+\tau ^{-2}|\nabla \pi-\Delta z|^2) e^{2 s e^{\tau \psi}}{\rm d}x   ,\label{CarlEstStokes1}\\
\int_{\mathcal{O}}  s\tau ^2  e^{\tau \psi} \left (|{\rm curl}\, z|^2+|{\rm div}\, z-\pi|^2\right) e^{2 s e^{\tau \psi}}{\rm d}x   
&\leq& C \int_{\mathcal{O}}  |\nabla \pi-\Delta z|^2 e^{2 s e^{\tau \psi}}{\rm d}x   .\label{CarlEstStokes2}
\end{eqnarray}
\end {Corollary}
\begin{proof}
Set $f\ov -\Delta z+\nabla \pi$ and $g\ov {\rm div}\, z$. From $-\Delta z={\rm curl}\,({\rm curl}\, z)-\nabla({\rm div}\, z)$ we get:
\begin{eqnarray}
-\Delta z&=&\frac{1}{1+\alpha}\left({\rm curl}\,({\rm curl}\, z)-\nabla (g+\alpha\pi) +\alpha f\right) \quad \mbox{ in }\quad \mathcal{O},\label{eqPrCarl1}\\
-\Delta ({\rm curl}\, z)&=&{\rm curl}\,f \quad \mbox{ in }\quad \mathcal{O},\label{eqPrCarl2}\\
-\Delta (\pi-g)&=&-{\rm div}\,f \quad \mbox{ in }\quad \mathcal{O}.\label{eqPrCarl3}
\end{eqnarray}
Then \eqref{CarlEstStokes2} is obtained by applying \eqref{CarlEstEll} for $k=0$ to \eqref{eqPrCarl2} and to \eqref{eqPrCarl3}. 
Finally, \eqref{CarlEstStokes1} is obtained by first applying \eqref{CarlEstEll} for $k=1$ to \eqref{eqPrCarl1} and next using the estimate of 
${\rm curl}\, z$ given by \eqref{CarlEstStokes2}.
\end{proof}
 
\subsection{Uniqueness theorem for the Oseen system}
Here we give a first illustration of an application of Corollary \ref{CarlLocalStokes} that allows to recover in an easy way a uniqueness theorem for the Oseen equations. It also provides a sketch of 
the method of proof which is used in the next section to obtain uniqueness theorems for coupled Stokes type systems.
For $\lambda \in \mathbb{C}$ and $w^S\in (L_{\rm loc}^\infty(\Omega;\mathbb{R}))^d$ consider the following eigenvalue problem:
\begin{equation}\label{OseenVP}
\begin{split}
 \lambda  z -\Delta z-(D^s z)w^S+\nabla \pi&=0 \quad \mbox{ in }\quad \Omega,\\
{\rm div}\, z&=0 \quad \mbox{ in }\quad \Omega.
\end{split}
\end{equation}
Let us prove that every $z\in (H_{\rm loc}^1(\Omega;\mathbb{C}))^d$ solution of \eqref{OseenVP} which vanishes in a non empty open subset $\omega\subset\subset \Omega$  must 
be identically zero in $\Omega$. 
Such a \textit{distributed observability theorem} has been first proved by 
Fabre and Lebeau in \cite{FL1996} for Stokes equations with bounded potential. 
When $\omega$ has a smooth $C^2$ boundary and $w^S\in (H^2(\Omega;\mathbb{R}))^d$ the result has been obtained in 
\cite{Triggiani2009}. Such a uniqueness theorem implies Fattorini's criterion \eqref{UC} corresponding to the linearized Navier-Stokes equations and permits to construct
 a feedback or a dynamical control stabilizing the Navier-Stokes equations around a given stationary state $w^S$, see \cite{BT} for details.
\par More precisely, we prove the following theorem.
\begin{Theorem}\label{UCOseen}
Let $\Omega$ be a connected open subset of $\mathbb{R}^d$ with $d=2$ or $d=3$, $\lambda \in \mathbb{C}$ and $w^S\in (L_{\rm loc}^\infty(\Omega;\mathbb{R}))^d$. Suppose
that $(z,\pi)\in (H_{\rm loc}^1(\Omega;\mathbb{C}))^d\times L^2_{\rm loc}(\Omega;\mathbb{C})$ satisfies \eqref{OseenVP} and that $z$ vanishes on a non empty open subset of $\Omega$. Then 
$z$ is identically equal to zero in $\Omega$ and the function $\pi$ is constant.
\end{Theorem} 
\begin{proof}
The proof consists in three steps. 
\medskip

{\bf Step 1. A local Carleman inequality.} Let $\mathcal{O}\subset\subset \Omega$ and $\psi$ satisfying \eqref{eqpsi}. From \eqref{CarlEstStokes1} and \eqref{CarlEstStokes2} with $\alpha=0$ 
we get that every $(z,\pi)\in (H_0^2({\mathcal{O}};\mathbb{C}))^d\times H_0^1({\mathcal{O}};\mathbb{C})$ satisfies:
\begin{equation}
\begin{split}
\int_{\mathcal{O}} &\left(|\nabla z|^2 +s^2\tau ^2 e^{2\tau \psi} |z|^2+s\tau ^2 e^{\tau \psi} |\pi-{\rm div} z|^2\right) e^{2 s e^{\tau \psi}}{\rm d}x   \leq\\
&C\int_{\mathcal{O}} (se^{\tau \psi}  |{\rm div}\, z|^2+| \lambda  z -\Delta z-(D^s z) w^S+\nabla \pi|^2) e^{2 s e^{\tau \psi}}{\rm d}x   .\label{CarlEstOseen1}
\end{split}
\end{equation}
Indeed, it suffices to choose $\tau $ in \eqref{CarlEstStokes1} large enough so that the first order term $(D^s z)w^S$ is absorbed by the left side of the inequality 
(since $w^S\in (L^\infty({\mathcal{O}};\mathbb{R}))^d$), and thus to choose $s$ large enough
so that the zero order term $\lambda  z$ is absorbed by the left hand side of the inequality. Thus, we combine the resulting inequality with \eqref{CarlEstStokes2}.

Next, suppose that $(z,\pi)\in (H_{\rm loc}^1(\Omega;\mathbb{C}))^d\times L_{\rm loc}^2(\Omega;\mathbb{C})$ is a solution of \eqref{OseenVP}. 
Let $\mathcal{O}_1$ and  $\mathcal{O}^*$ be two open subsets of $\mathbb{R}^d$ satisfying
\begin{equation}\label{Defopen}
{\mathcal{O}}_1\subset\subset {\mathcal{O}}\quad\mbox{ and }\quad  {\mathcal{O}}^*\ov {\mathcal{O}}\setminus \overline {\mathcal{O}}_1,
\end{equation}
and let $\chi:\mathcal{O}\to \mathbb{R}$ a cut-off function such that 
\begin{equation}\label{Defchi}
\chi\in C_c^\infty({\mathcal{O}};\mathbb{R}),\quad \chi\equiv 1\;\mbox{ in }\;{\mathcal{O}}_1\subset\subset {\mathcal{O}}\quad\mbox{ and }\quad   0\leq \chi\leq  1\; \mbox{ in }\; {\mathcal{O}}^*. 
\end{equation}
From \eqref{OseenVP} and regularity results for Stokes equations we get that 
$(\chi z,\chi\pi)\in (H_0^2({\mathcal{O}};\mathbb{C}))^d\times H_0^1({\mathcal{O}};\mathbb{C})$ satisfies:
\begin{equation}\label{eqOseenchi}
\begin{split}
\lambda  (\chi z) -\Delta (\chi z)-(D^s (\chi z))w^S+\nabla (\chi \pi)&=-[\Delta +(w^S\cdot\nabla)+(w^S\cdot\,^t\nabla), \chi]z+[\nabla, \chi]\pi\quad \mbox{ in }\quad \Omega\\
{\rm div}\, (\chi z)&=[{\rm div}, \chi]z \quad \mbox{ in }\quad \Omega,
\end{split}
\end{equation}
where $[\cdot,\cdot]$ denote commutators. Then applying \eqref{CarlEstOseen1} to $(\chi z,\chi \pi)$ and using the fact that commutators in \eqref{eqOseenchi} are supported in 
$\mathcal{O}^*$ we deduce that solutions of \eqref{OseenVP} obey:
\begin{equation}\label{CarlFinale}
\int_{\mathcal{O}_1} e^{2 s e^{\tau \psi }}(|z|^2+|\pi|^2){\rm d}x   \leq C_{\chi,\tau ,\psi} \int_{\mathcal{O}^*}e^{2 s e^{\tau \psi }}(|z|^2+|\nabla z|^2+|\pi|^2){\rm d}x   .
\end{equation}
\medskip

{\bf Step 2. A local uniqueness theorem.}  Now let us use inequality \eqref{CarlFinale} to prove that if $z$ vanishes in a relatively compact ball of $\Omega$ then $z$ is necessarily 
zero in a neighborhood of such a ball. More precisely, let $x_0\in \mathbb{R}^d$ and $R>0$, denote by $B(x_0,R)$ the open ball centered at $x_0$ with radius $R$ and suppose that $B(x_0,R)\subset\subset \Omega$ and $z\equiv 0$ in 
$B(x_0,R)$. Since $\partial B(x_0,R)$ is compact, to prove that $z$ vanishes in an open neighborhood of $\partial B(x_0,R)$ it suffices to prove that for each $x_1\in \partial B(x_0,R)$
there is a ball $B(x_1,\varepsilon)$, $\varepsilon>0$, in which $z$ vanishes.

For $x_1\in \partial B(x_0,R)$ choose $\mathcal{O}=B(x_1,r)$, $\mathcal{O}_1=B(x_1,r/2)$, $\mathcal{O}^*=\mathcal{O}\backslash \overline{\mathcal{O}}_1$ with 
$0<r<R/2$ small enough such that $B(x_0,R)\cup \mathcal{O}\subset \Omega$ and set 
$$\psi(x)=C_0 -|x-x_0|^2-\frac{1}{2} |x-x_1|^2,$$ 
where $C_0>0$ is large enough so that $\psi > 0$ in $\mathcal{O}$. Moreover, we verify that $|\nabla\psi|>0$ and then \eqref{eqpsi} is satisfied. Thus, set 
$\mathcal{V}=\mathcal{O}\setminus B(x_0,R)$, $\mathcal{V}_\varepsilon=B(x_1,\varepsilon)\setminus B(x_0,R)$ for $\varepsilon>0$ and $\mathcal{V}^*=\mathcal{O}^*\setminus B(x_0,R)$. 
Since we have
$\psi(x)<\psi(x_1)$ for all $x\in \overline{\mathcal{V}^*}$, by continuity we get for $\varepsilon>0$ small enough:
\begin{equation}
\max_{\mathcal{V}^*}{\psi}\ov \psi^*< \psi_1\ov \min_{\mathcal{V}_\varepsilon}\psi.\label{maxminIneq}
\end{equation}
From \eqref{OseenVP} and $z\equiv 0$ in 
$B(x_0,R)$ we get that $\nabla\pi$ is zero in 
$B(x_0,R)$ and then $\pi=\overline{\pi}\in \mathbb{R}$ is constant in $B(x_0,R)$. 
Then we apply \eqref{CarlFinale} to  $(z,\pi-\overline{\pi})$ (which satisfies \eqref{OseenVP}) and since $z$ and $\pi-\overline{\pi}$ vanish in $B(x_0,R)$ we obtain
$$ 
\int_{\mathcal{V}_\varepsilon} (|z|^2+|\pi-\overline{\pi}|^2){\rm d}x   \leq e^{2 s (e^{\tau \psi^*}-e^{\tau \psi_1})}C_{\chi,\tau ,\psi} 
\int_{\mathcal{V}^*}(|z|^2+|\nabla z|^2+|\pi-\overline{\pi}|^2){\rm d}x   .
$$
Inequality \eqref{maxminIneq} implies that the right side of the above inequality tends to zero as $s$ goes to infinity, which 
yields $z\equiv 0$ and $\pi\equiv \overline{\pi}$ in $\mathcal{V}_\varepsilon$. 
\medskip

{\bf Step 3. A Connectivity argument.} Suppose that $z$ is a non zero solution of \eqref{OseenVP} which is vanishing in an open subset of $\Omega$, that is to say 
${\rm Supp}(z)\neq \Omega$ and ${\rm Supp}(z)\neq \emptyset$. Because $\Omega$ is connected and ${\rm Supp}(z)$ is a closed subset of $\Omega$ we have that 
${\rm Supp}(z)$ is not an open subset of $\Omega$. As a consequence, there exists $x^*\in {\rm Supp}(z)$ and $R^*>0$ small enough such that 
$B(x^*,R^*)\cap (\Omega\backslash {\rm Supp}(z))\neq \emptyset$ and $B(x^*,R^*)\subset\Omega$. Then we can choose 
$x_0\in \Omega \setminus {\rm Supp}(z)$ and $R_0>0$ such that 
$B(x_0,R_0) \subset B(x^*,R^*) $ and $x^*\in B(x_0,R_0)$ (for instance, for $x_0^*\in B(x^*,R^*)\cap (\Omega\backslash {\rm Supp}(z))$ choose 
$x_0=x^*+\frac{1}{3}(x_0^*-x^*)$ and $R_0=\frac{2}{3}|x_0^*-x^*|$). To summerize, we have:
\begin{equation}
x_0\in \Omega \setminus {\rm Supp}(z),\quad B(x_0,R_0)\subset \Omega \quad \mbox{ and }\quad B(x_0,R_0)\cap {\rm Supp}(z)\neq \emptyset.\label{Ball0}
\end{equation}
Next, because $\Omega\backslash {\rm Supp}(z)$ is open in $\mathbb{R}^d$, there exists $r_0\in (0,R_0)$ such that $B(x_0,r_0)\subset \Omega\backslash {\rm Supp}(z)$,
and we can introduce the value $\widehat R\in [r_0,R_0)$ defined by
$$
\widehat R\ov \sup\{R>0\mid B(x_0,R)\subset \Omega\backslash {\rm Supp}(z)\}.
$$
Since $\Omega\backslash {\rm Supp}(z)$ is open the above infimum is a maximum and $B(x_0,\widehat R)\subset \Omega\backslash {\rm Supp}(z)$.
Moreover, recall that \eqref{Ball0} guarantees $B(x_0,\widehat R)\subset B(x_0,R_0)\subset \Omega$ and $\widehat R<R_0$ which implies $B(x_0,\widehat R)\subset\subset \Omega$.
Then according to Step 2 above there exists $\varepsilon>0$ such that $B(x_0,\widehat R+\varepsilon)\subset \Omega \setminus {\rm Supp}(z)$ which contradicts the definition 
of $\widehat R$. In conclusion, ${\rm Supp}(z)=\emptyset$ or ${\rm Supp}(z)=\Omega$, which is to say that a 
non zero solution $z$ of \eqref{OseenVP} cannot vanish on an open subset of $\Omega$.
\end{proof}

\begin{Remark}
Note that the key argument in the proof of Theorem \ref{UCOseen} is to obtain the local Carleman inequality \eqref{CarlEstOseen1} related to Oseen system. Steps 2 and 3 of the proof
are standard and have been recalled for readability convenience, see for instance \cite{LebeauLeRousseauCOCV} or \cite[Appendix IV]{TucsnakWeiss}.
\end{Remark}

\subsection{Uniqueness theorem for adjoint MHD system}
Corollary \ref{CarlLocalStokes} also permits to prove a \textit{distributed observability theorem} for coupled Oseen systems. Here we prove a uniqueness theorem
related to the stabilizability of a MHD system (see Subsection \ref{stabMHD}). 

For $w^S\in (L_{\rm loc}^\infty(\Omega;\mathbb{R}))^d$ and $\theta^S\in (L_{\rm loc}^\infty(\Omega;\mathbb{R}))^d$ consider the following adjoint linearized MHD system:
\begin{equation}\label{MHDadj}
 \left\{
\begin{split}
 \lambda  z-\Delta z-(D^s z) w^S+(D^a \rho) \theta^S +\nabla \pi &=0 &\quad \mbox{ in }\Omega,\\
 \lambda  \rho-\Delta \rho+(D^s z) \theta^S-(D^a \rho) w^S+\nabla \kappa &=0&\quad \mbox{ in }\Omega,\\
{\rm div}\, z={\rm div}\, \rho&=0&\quad \mbox{ in }\Omega.
\end{split}
\right.
\end{equation}
\begin{Theorem}\label{UCMHD}
Let $\Omega$ be a connected open subset of $\mathbb{R}^d$ with $d=2$ or $d=3$, $\lambda \in \mathbb{C}$, $w^S\in (L_{\rm loc}^\infty(\Omega;\mathbb{R}))^d$ and 
$\theta^S\in (L_{\rm loc}^\infty(\Omega;\mathbb{R}))^d$. 
Suppose that $(z,\pi)\in (H_{\rm loc}^1(\Omega;\mathbb{C}))^d\times L^2_{\rm loc}(\Omega;\mathbb{C})$ and 
$(\rho,\kappa )\in (H_{\rm loc}^1(\Omega;\mathbb{C}))^d\times L^2_{\rm loc}(\Omega;\mathbb{C})$ satisfy \eqref{MHDadj} and 
that $z$ and ${\rm curl\, \rho}$ vanish on a non empty open subset $\omega\subset\subset\Omega$. Then $z$ and ${\rm curl\, \rho}$ are identically equal to zero in $\Omega$, the function $\pi$ is
constant and $\nabla  \kappa=-\lambda\rho$.
\end{Theorem}
\begin{proof}
The point is to use Corollary \ref{CarlLocalStokes} to get an inequality for \eqref{MHDadj} analogue to \eqref{CarlFinale}. 
Then the result will follow from a local uniqueness theorem and a connectivity argument completely analogous to steps 2 and 3 of the proof of Theorem \ref{UCOseen}. 

 First, set $\zeta\ov {\rm curl}\, \rho$, apply the operator ${\rm\, curl}$ to the second equality in \eqref{MHDadj} to get
\begin{equation}\label{MHDadjbis}
 \left\{
\begin{split}
 \lambda  z-\Delta z-(D^s z) w^S+S(\theta^S)\zeta +\nabla \pi &=0 &\quad \mbox{ in }\Omega,\\
 \lambda  \zeta-\Delta \zeta+{\rm curl}\left( (D^s z) \theta^S-S(w^S)\zeta \right)&=0&\quad \mbox{ in }\Omega,\\
{\rm div}\, z&=0&\quad \mbox{ in }\Omega.
\end{split}
\right.
\end{equation}
To obtain \eqref{MHDadjbis} we have used the relation $(D^a \rho) b=S(b){\rm curl}\, \rho$
where  
$$
S(b)\ov \vct{-b_2}{b_1}\,\mbox{ if $d=2$}\quad \mbox{ and }\quad  S(b) \ov \mtx{0}{b_3}{-b_2}{-b_3}{0}{b_1}{b_2}{-b_1}{0}\quad\mbox{ if $d=3$}.
$$
Next, let ${\mathcal{O}}\subset\subset \Omega$ and $\psi:{\mathcal{O}}\to \mathbb{R}$ satisfying \eqref{eqpsi}, suppose for the moment that 
$(z,\zeta, \pi)\in (H_0^2({\mathcal{O}};\mathbb{C}))^d\times (H_0^2({\mathcal{O}};\mathbb{C}))^{2 d-3} \times H_0^1({\mathcal{O}};\mathbb{C})$ and let us prove that such $(z,\zeta, \pi)$ 
obeys
\begin{equation}\label{CarlE0stMHD1}
\begin{split}
 &\int_{\mathcal{O}} \left(s^2\tau ^2 e^{2\tau \psi} |z|^2+s\tau ^2 e^{\tau \psi} (|\pi-{\rm div}\, z|^2+|\zeta|^2)\right) e^{2 s e^{\tau \psi}}{\rm d}x   \\
&\leq
C\int_{\mathcal{O}} (se^{\tau \psi}  |{\rm div}\, z|^2+| \mathcal{L}_1(z,\zeta,\pi)|^2+| \mathcal{L}_2(z,\zeta)|^2) e^{2 s e^{\tau \psi}}{\rm d}x   ,
\end{split}
\end{equation}
where
\begin{equation}\label{eqLcal}
\begin{split}
\mathcal{L}_1(z,\zeta,\pi)&\ov \lambda  z -\Delta z-(D^s z) w^S+S(\theta^S)\zeta+\nabla \pi,\\
\mathcal{L}_2(z,\zeta)&\ov \lambda  \zeta-\Delta \zeta+{\rm curl}\left( (D^s z) \theta^S-S(w^S)\zeta \right).
\end{split}
\end{equation}
For that, first apply \eqref{CarlEstEll} to $\zeta$ (with $k=0$) and obtain
\begin{equation}\label{ineinterm}
\begin{split}
 \int_{\mathcal{O}}s\tau ^2 e^{\tau \psi} |\zeta|^2 e^{2 s e^{\tau \psi}}{\rm d}x   \leq C
\int_{\mathcal{O}} \left(s^{-2}\tau ^{-2}  |\mathcal{L}_2(z,\zeta)-\lambda \zeta|^2+|(D^s z) \theta^S-S(w^S)\zeta |^2\right) e^{2 s e^{\tau \psi}}{\rm d}x   .
\end{split}
\end{equation}
Thus, apply \eqref{CarlEstStokes1} (with $\alpha=0$) and \eqref{CarlEstStokes2} to $(z,\pi)$, add the resulting inequality to \eqref{ineinterm} and choose 
$\tau $ and $s$ large enough so that the first order terms $(D^s z) w^S$, $(D^s z) \theta^S$ and the zero order terms 
$S(\theta^S) \zeta$, $S(w^S) \zeta$, $\lambda  z$, $\lambda  \zeta$ are absorbed by the left hand side of the inequality.

Finally, let $\mathcal{O}_1$ and  $\mathcal{O}^*$ two open subsets of $\mathbb{R}^d$ satisfying \eqref{Defopen} and let $\chi:\mathcal{O}\to \mathbb{R}$ a cut-off function 
satisfying \eqref{Defchi}. If $(z,\zeta, \pi)\in (H_{\rm loc}^1(\Omega;\mathbb{C}))^d\times (H_{\rm loc}^1(\Omega;\mathbb{C}))^{2d-3} \times L^2_{\rm loc}(\Omega;\mathbb{C})$ satisfies
\eqref{MHDadjbis} then by elliptic and Stokes regularity we get that $(\chi z,\chi \zeta, \chi \pi)\in (H_0^2({\mathcal{O}};\mathbb{C}))^d\times (H_0^2({\mathcal{O}};\mathbb{C}))^{2d-3}\times H_0^1({\mathcal{O}};\mathbb{C})$ and 
\eqref{CarlE0stMHD1} applied to $(\chi z,\chi \zeta, \chi \pi)$ yields:
\begin{equation}\label{CarlMHD}
\int_{\mathcal{O}_1} e^{2 s e^{\tau \psi }}(|z|^2+|\pi|^2+|\zeta|^2){\rm d}x   \leq C_{\chi,\tau ,\psi} \int_{\mathcal{O}^*}
e^{2 s e^{\tau \psi }}(|z|^2+|\nabla z|^2+|\pi|^2+|\zeta|^2){\rm d}x   .
\end{equation}
Then we conclude with a local uniqueness theorem and a connectivity argument as for Theorem \ref{UCOseen}. 

Note that the facts that $\pi$ is
constant and that $\nabla  \kappa=-\lambda\rho$ are direct consequences of the equations \eqref{MHDadj} with $z\equiv 0$ and ${\rm curl\, \rho}\equiv 0$ in $\Omega$ 
(since ${\rm div\,}\rho=0$ implies $-\Delta\rho={\rm curl\,} {\rm curl\, \rho}$).
\end{proof}
Now assume that $\Omega$ is a bounded domain of $\mathbb{R}^d$ with a Lipschitz-continuous boundary, denote by $n$ the unit exterior normal vector field defined on $\partial\Omega$ 
and introduce the following $N$-dimensional space:
\begin{equation}
X_N\ov \{y\in (L^2(\Omega;\mathbb{C}))^d\mid {\rm div\,} y=0 \mbox{ in }\Omega,\quad {\rm curl\,} y=0 \mbox{ in }\Omega,\quad y\cdot n=0\mbox{ on } \partial\Omega\}.\label{defX}
\end{equation}
The fact that the above space is finite dimensional is well-known, see for instance \cite[Chap. IX]{DautrayLionsV5} 
for a detailed characterization of $X_N$. We only recall that if $\Omega$ is simply-connected we have $N=0$ and 
$X_N$ reduces to $\{0\}$ and if $\Omega$ is multiply-connected then $N\geq 1$ is the number of cuts required to make $\Omega$ simply-connected. 
Note that if $d=2$ then $N+1$ is exactly 
to the number
of connected components of $\partial\Omega$. 

 The following straightforward consequence of Theorem \ref{UCMHD} holds.
\begin{Corollary}\label{CorMHD}
In addition to the hypotheses of Theorem \ref{UCMHD}, assume that $\Omega$ is a bounded domain of $\mathbb{R}^d$ with a Lipschitz-continuous boundary and suppose that
\begin{equation}
\rho\cdot n=0\quad \mbox{ on }\quad \partial\Omega.\label{rhonnull}
\end{equation}
Then the following results hold.
\begin{enumerate}
\item If $\lambda\neq 0$ then $z$ and $\rho$ are identically equal to zero and the functions $\pi$ and $\kappa$ are constant in $\Omega$.
\item If $\lambda=0$ then $z$ is identically equal to zero in $\Omega$ and $\rho\in X_N$. Moreover, the functions $\pi$ and $\kappa$ are constant in $\Omega$. 
In particular, if $\Omega$ is simply-connected then $\rho$ is identically equal to zero in $\Omega$.
\end{enumerate}
\end{Corollary}

\subsection{Uniqueness theorem for adjoint micropolar system}
Corollary \ref{CarlLocalStokes} also provides a \textit{distributed observability theorem} related to stabilizability of a micropolar system (see Subsection \ref{stabMicro}).
 
\begin{Theorem}\label{UCMicro}
Let $\Omega$ be a connected open subset of $\mathbb{R}^d$ with $d=2$ or $d=3$, $\lambda \in \mathbb{C}$, $w^S\in (L_{\rm loc}^\infty(\Omega;\mathbb{R}))^d$ and 
$\theta^S\in (L_{\rm loc}^\infty(\Omega;\mathbb{R}))^d$. 
Suppose that $(z,\pi)\in (H_{\rm loc}^1(\Omega;\mathbb{C}))^d\times L^2_{\rm loc}(\Omega;\mathbb{C})$ and 
$(\rho,\kappa )\in (H_{\rm loc}^1(\Omega;\mathbb{C}))^d\times L^2_{\rm loc}(\Omega;\mathbb{C})$ satisfy \eqref{Microadj} and that $z$ and $\rho$ vanish on a non empty open subset 
$\omega\subset\subset\Omega$. Then $z$ and $\rho$ are identically equal to zero in $\Omega$ and the functions $\pi$ and $\kappa $ are constant.
\end{Theorem}
\begin{proof}
Let us first consider two pairs 
$(z,\pi)$ and $(\rho,\kappa )$ of $(H_0^2({\mathcal{O}};\mathbb{C}))^d\times H_0^1({\mathcal{O}};\mathbb{C})$. 
By applying \eqref{CarlEstStokes1}, \eqref{CarlEstStokes2} to $(z,\pi)$ with $\alpha=0$ and 
\eqref{CarlEstStokes1}, \eqref{CarlEstStokes2} to $(\rho,\kappa )$ with $\alpha=1$, and choosing $\tau$ and $s$ large enough to absorb the zero and first order terms in $z$, $\rho$
we get:
\begin{equation}\label{CarlE0stMicro1}
\begin{split}
 &\int_{\mathcal{O}} \left(s^2\tau ^2 e^{2\tau \psi} (|z|^2+|\rho|^2)+s\tau ^2 e^{\tau \psi} (|\pi-{\rm div}\, z|^2+|\kappa -{\rm div}\, \rho|^2)\right) e^{2 s e^{\tau \psi}}{\rm d}x   \\
&\leq
C\int_{\mathcal{O}} (se^{\tau \psi}  (|{\rm div}\, z|^2+|{\rm div}\, \rho+\kappa |^2)+(| \mathcal{L}_1(z,\rho,\pi)|^2+| \mathcal{L}_2(z,\rho,\kappa )|^2)) e^{2 s e^{\tau \psi}}{\rm d}x   ,
\end{split}
\end{equation}
where
\begin{equation}\nonumber
\begin{split}
\mathcal{L}_1(z,\rho,\pi)&\ov \lambda  z- \Delta z-(D^s z) w^S-\,^t(\nabla \rho) \theta^S -{\rm curl}\, \rho+\nabla \pi,\\
\mathcal{L}_2(z,\rho,\kappa )&\ov \lambda  \rho-\Delta \rho-(\nabla \rho) w^S-{\rm curl}\, z+\nabla \kappa .
\end{split}
\end{equation}
Thus, let $\mathcal{O}_1$ and  $\mathcal{O}^*$ two open subsets of $\mathbb{R}^d$ satisfying \eqref{Defopen} and let $\chi:\mathcal{O}\to \mathbb{R}$ a cut-off function 
satisfying \eqref{Defchi}. For $(z,\pi,\rho,\kappa )$ satisfying \eqref{Microadj} we apply inequality 
\eqref{CarlE0stMicro1} to $(\chi z,\chi \pi,\chi\rho,\chi\kappa )$ and obtain 
\begin{equation}\label{CarlMHDbisbis}
\int_{\mathcal{O}_1} e^{2 s e^{\tau \psi }}(|z|^2+|\pi|^2+|\rho|^2+|\kappa|^2){\rm d}x   \leq C_{\chi,\tau ,\psi} \int_{\mathcal{O}^*}
e^{2 s e^{\tau \psi }}(|z|^2+|\nabla z|^2+|\pi|^2+|\rho|^2+|\nabla \rho|^2+|\kappa|^2){\rm d}x   .
\end{equation}
Then we conclude with a local uniqueness theorem and a connectivity argument as for Theorem \ref{UCOseen}. 
\end{proof}
\bibliography{ma_biblio}

\def\cprime{$'$} \def\cprime{$'$}
\begin{thebibliography}{10}

\bibitem{BadraCOCV}
M.~Badra.
\newblock Feedback stabilization of the 2-{D} and 3-{D} {N}avier-{S}tokes
  equations based on an extended system.
\newblock {\em ESAIM Control Optim. Calc. Var.}, 15(4):934--968, 2009.

\bibitem{BADRASICON}
M.~Badra.
\newblock Lyapunov function and local feedback boundary stabilization of the
  {N}avier-{S}tokes equations.
\newblock {\em SIAM J. Control Optim.}, 48(3):1797--1830, 2009.

\bibitem{BADRA-DCDS-A2011}
M.~Badra.
\newblock Abstract settings for stabilization of nonlinear parabolic system
  with a {R}iccati-based strategy. {A}pplication to {N}avier-{S}tokes and
  {B}oussinesq equations with {N}eumann or {D}irichlet control.
\newblock {\em Discrete and Continuous Dynamical Systems - Series A},
  32(4):1169 -- 1208, 2011.

\bibitem{BADRA-MHD-SUB-3}
M.~Badra.
\newblock Local controllability to trajectories of the magnetohydrodynamic
  equations.
\newblock {\em J. Math. Fluid Mech.}, to appear.

\bibitem{BT}
M.~Badra and T.~Takahashi.
\newblock Stabilization of parabolic nonlinear systems with finite dimensional
  feedback or dynamical controllers. {A}pplication to the {N}avier-{S}tokes
  system.
\newblock {\em SIAM J. Control Optim.}, 49(2):420--463, 2011.

\bibitem{BT4}
M.~Badra and T.~Takahashi.
\newblock Feedback stabilization of a fluid--rigid body interaction system.
\newblock preprint.

\bibitem{BT3}
M.~Badra and T.~Takahashi.
\newblock Feedback stabilization of a simplified 1d fluid -- particle system.
\newblock {\em Ann. Inst. H. Poincar\'e Anal. Non Lin\'eaire}, to appear.

\bibitem{B-L-T-2006}
V.~Barbu, I.~Lasiecka, and R.~Triggiani.
\newblock Abstract settings for tangential boundary stabilization of
  {N}avier-{S}tokes equations by high- and low-gain feedback controllers.
\newblock {\em Nonlinear Anal.}, 64(12):2704--2746, 2006.

\bibitem{BarbuLasieckaTriggiani2004}
V.~Barbu, I.~Lasiecka, and R.~Triggiani.
\newblock Tangential boundary stabilization of {N}avier-{S}tokes equations.
\newblock {\em Mem. Amer. Math. Soc.}, 181(852):x+128, 2006.

\bibitem{B-L-T-2007}
V.~Barbu, I.~Lasiecka, and R.~Triggiani.
\newblock Local exponential stabilization strategies of the {N}avier-{S}tokes
  equations, {$d=2,3$}, via feedback stabilization of its linearization.
\newblock In {\em Control of coupled partial differential equations}, volume
  155 of {\em Internat. Ser. Numer. Math.}, pages 13--46. Birkh\"auser, Basel,
  2007.

\bibitem{B-T-2004}
V.~Barbu and R.L Triggiani.
\newblock Internal stabilization of {N}avier-{S}tokes equations with
  finite-dimensional controllers.
\newblock {\em Indiana Univ. Math. J.}, 53(5):1443--1494, 2004.

\bibitem{RCIDS2ED}
A.~Bensoussan, G.~Da~Prato, M.~C. Delfour, and S.~K. Mitter.
\newblock {\em Representation and control of infinite dimensional systems}.
\newblock Systems \& Control: Foundations \& Applications\. Birkh\"auser Boston
  Inc., Boston, MA, second edition, 2007.

\bibitem{CoronTrelat2004}
J.-M. Coron and E.~Tr{\'e}lat.
\newblock Global steady-state controllability of one-dimensional semilinear
  heat equations.
\newblock {\em SIAM J. Control Optim.}, 43(2):549--569 (electronic), 2004.

\bibitem{CurtainZwart1995}
R.~F. Curtain and H.~Zwart.
\newblock {\em An introduction to infinite-dimensional linear systems theory},
  volume~21 of {\em Texts in Applied Mathematics}.
\newblock Springer-Verlag, New York, 1995.

\bibitem{DautrayLionsV5}
R.~Dautray and J.-L. Lions.
\newblock {\em Analyse math\'ematique et calcul num\'erique pour les sciences
  et les techniques. {V}ol. 5}.
\newblock INSTN: Collection Enseignement. [INSTN: Teaching Collection]. Masson,
  Paris, 1988.
\newblock Spectre des op{\'e}rateurs. [The operator spectrum], With the
  collaboration of Michel Artola, Michel Cessenat, Jean Michel Combes and Bruno
  Scheurer, Reprinted from the 1984 edition.

\bibitem{Davies1999}
E.~B. Davies.
\newblock Pseudo-spectra, the harmonic oscillator and complex resonances.
\newblock {\em R. Soc. Lond. Proc. Ser. A Math. Phys. Eng. Sci.},
  455(1982):585--599, 1999.

\bibitem{FL1996}
C.~Fabre and G.~Lebeau.
\newblock Prolongement unique des solutions de l'equation de {S}tokes.
\newblock {\em Comm. Partial Differential Equations}, 21(3-4):573--596, 1996.

\bibitem{Fattorini1966}
H.~O. Fattorini.
\newblock Some remarks on complete controllability.
\newblock {\em SIAM J. Control}, 4:686--694, 1966.

\bibitem{Fattorini1967}
H.~O. Fattorini.
\newblock On complete controllability of linear systems.
\newblock {\em J. Differential Equations}, 3:391--402, 1967.

\bibitem{Guerrero2007}
E.~Fern{\'a}ndez-Cara and S.~Guerrero.
\newblock Local exact controllability of micropolar fluids.
\newblock {\em J. Math. Fluid Mech.}, 9(3):419--453, 2007.

\bibitem{exactcontollFGIP}
E.~Fern{\'a}ndez-Cara, S.~Guerrero, O.~Yu. Imanuvilov, and J.-P. Puel.
\newblock Local exact controllability of the {N}avier-{S}tokes system.
\newblock {\em J. Math. Pures Appl. (9)}, 83(12):1501--1542, 2004.

\bibitem{Fursikov2001-0}
A.~V. Fursikov.
\newblock Stabilizability of a quasilinear parabolic equation by means of
  boundary feedback control.
\newblock {\em Mat. Sb.}, 192(4):115--160, 2001.

\bibitem{Fursikov2001}
A.~V. Fursikov.
\newblock Stabilizability of two-dimensional {N}avier-{S}tokes equations with
  help of a boundary feedback control.
\newblock {\em J. Math. Fluid Mech.}, 3(3):259--301, 2001.

\bibitem{Fursikov2004}
A.~V. Fursikov.
\newblock Stabilization for the 3{D} {N}avier-{S}tokes system by feedback
  boundary control.
\newblock {\em Discrete Contin. Dyn. Syst.}, 10(1-2):289--314, 2004.
\newblock Partial differential equations and applications.

\bibitem{Galdi}
G.~P. Galdi.
\newblock {\em An introduction to the mathematical theory of the
  {N}avier-{S}tokes equations. {V}ol. {I}. Linearized steady problems},
  volume~38 of {\em Springer Tracts in Natural Philosophy}.
\newblock Springer-Verlag, New York, 1994.

\bibitem{Krein}
I.~C. Gohberg and M.~G. Kre{\u\i}n.
\newblock {\em Introduction to the theory of linear nonselfadjoint operators}.
\newblock Translated from the Russian by A. Feinstein. Translations of
  Mathematical Monographs, Vol. 18. American Mathematical Society, Providence,
  R.I., 1969.

\bibitem{HardyWright}
G.~H. Hardy and E.~M. Wright.
\newblock {\em An introduction to the theory of numbers}.
\newblock Oxford University Press, Oxford, sixth edition, 2008.
\newblock Revised by D. R. Heath-Brown and J. H. Silverman, With a foreword by
  Andrew Wiles.

\bibitem{Hautus1969}
M.~L.~J. Hautus.
\newblock Controllability and observability conditions of linear autonomous
  systems.
\newblock {\em Nederl. Akad. Wetensch. Proc. Ser. A 72 = Indag. Math.},
  31:443--448, 1969.

\bibitem{HenrotPierre2005}
A.~Henrot and M.~Pierre.
\newblock {\em Variation et optimisation de formes}, volume~48 of {\em
  Math\'ematiques \& Applications (Berlin) [Mathematics \& Applications]}.
\newblock Springer, Berlin, 2005.
\newblock Une analyse g{\'e}om{\'e}trique. [A geometric analysis].

\bibitem{HormanderI}
L.~H{\"o}rmander.
\newblock {\em The analysis of linear partial differential operators. {I}}.
\newblock Classics in Mathematics. Springer-Verlag, Berlin, 2003.
\newblock Distribution theory and Fourier analysis, Reprint of the second
  (1990) edition [Springer, Berlin; MR1065993 (91m:35001a)].

\bibitem{ImmPuel2003}
O.~Y. Imanuvilov and J.-P. Puel.
\newblock Global {C}arleman estimates for weak solutions of elliptic
  nonhomogeneous {D}irichlet problems.
\newblock {\em Int. Math. Res. Not.}, (16):883--913, 2003.

\bibitem{Imanuvilov2001}
O.~Yu. Imanuvilov.
\newblock Remarks on exact controllability for the {N}avier-{S}tokes equations.
\newblock {\em ESAIM Control Optim. Calc. Var.}, 6:39--72 (electronic), 2001.

\bibitem{Kato1966}
T.~Kato.
\newblock {\em Perturbation theory for linear operators}.
\newblock Classics in Mathematics. Springer-Verlag, Berlin, 1995.
\newblock Reprint of the 1980 edition.

\bibitem{L-TlBook2000Vol1}
I.~Lasiecka and R.~Triggiani.
\newblock {\em Control theory for partial differential equations: continuous
  and approximation theories. {I}}, volume~74 of {\em Encyclopedia of
  Mathematics and its Applications}.
\newblock Cambridge University Press, Cambridge, 2000.
\newblock Abstract parabolic systems.

\bibitem{LebeauLeRousseauCOCV}
J.~Le~Rousseau and G.~Lebeau.
\newblock On {C}arleman estimates for elliptic and parabolic operators.
  applications to unique continuation and control of parabolic equations.
\newblock {\em ESAIM Control Optim. and Calc. Var.}

\bibitem{Lefter2010}
C.-G. Lefter.
\newblock On a unique continuation property related to the boundary
  stabilization of magnetohydrodynamic equations.
\newblock {\em An. \c Stiin\c t. Univ. Al. I. Cuza Ia\c si. Mat. (N.S.)},
  56(1):1--15, 2010.

\bibitem{Lukaszewicz}
G.~{\L}ukaszewicz.
\newblock {\em Micropolar fluids}.
\newblock Modeling and Simulation in Science, Engineering and Technology.
  Birkh\"auser Boston Inc., Boston, MA, 1999.
\newblock Theory and applications.

\bibitem{Meir1993}
A.~J. Meir.
\newblock The equations of stationary, incompressible magnetohydrodynamics with
  mixed boundary conditions.
\newblock {\em Comput. Math. Appl.}, 25(12):13--29, 1993.

\bibitem{Micheletti1972}
A.~M. Micheletti.
\newblock Perturbazione dello spettro dell'operatore di {L}aplace, in relazione
  ad una variazione del campo.
\newblock {\em Ann. Scuola Norm. Sup. Pisa (3)}, 26:151--169, 1972.

\bibitem{MicuZuazua}
S.~Micu and E.~Zuazua.
\newblock On the controllability of a fractional order parabolic equation.
\newblock {\em SIAM J. Control Optim.}, 44(6):1950--1972 (electronic), 2006.

\bibitem{PAZY}
A.~Pazy.
\newblock {\em Semigroups of linear operators and applications to partial
  differential equations}, volume~44 of {\em Applied Mathematical Sciences}.
\newblock Springer-Verlag, New York, 1983.

\bibitem{Raymond_2005_2}
J.-P. Raymond.
\newblock Feedback boundary stabilization of the two-dimensional
  {N}avier-{S}tokes equations.
\newblock {\em SIAM J. Control Optim.}, 45(3):790--828 (electronic), 2006.

\bibitem{RaymondThevenet2009}
J.-P. Raymond and T.~Thevenet.
\newblock Boundary feedback stabilization of the two dimensional
  {N}avier-{S}tokes equations with finite dimensional controllers.
\newblock {\em Issue in Discrete and Continuous Dynamical Systems A},
  27(3):1159--1187, 2010.

\bibitem{RussellWeiss1994}
D.~L. Russell and G.~Weiss.
\newblock A general necessary condition for exact observability.
\newblock {\em SIAM J. Control Optim.}, 32(1):1--23, 1994.

\bibitem{Triebel}
H.~Triebel.
\newblock {\em Interpolation theory, function spaces, differential operators}.
\newblock Johann Ambrosius Barth, Heidelberg, second edition, 1995.

\bibitem{Triggiani1975}
R.~Triggiani.
\newblock On the stabilizability problem in {B}anach space.
\newblock {\em J. Math. Anal. Appl.}, 52(3):383--403, 1975.

\bibitem{Triggiani1976}
R.~Triggiani.
\newblock Extensions of rank conditions for controllability and observability
  to {B}anach spaces and unbounded operators.
\newblock {\em SIAM J. Control Optimization}, 14(2):313--338, 1976.

\bibitem{Triggiani1980}
R.~Triggiani.
\newblock Boundary feedback stabilizability of parabolic equations.
\newblock {\em Appl. Math. Optim.}, 6(3):201--220, 1980.

\bibitem{Triggiani2009}
R.~Triggiani.
\newblock Unique continuation from an arbitrary interior subdomain of the
  variable-coefficient {O}seen equation.
\newblock {\em Nonlinear Anal.}, 71(10):4967--4976, 2009.

\bibitem{TucsnakWeiss}
M.~Tucsnak and G.~Weiss.
\newblock {\em Observation and control for operator semigroups}.
\newblock Birkh\"auser Advanced Texts: Basler Lehrb\"ucher. [Birkh\"auser
  Advanced Texts: Basel Textbooks]. Birkh\"auser Verlag, Basel, 2009.

\end{thebibliography}
\bibliographystyle{plain}
\end{document}